\documentclass{elsart}
\usepackage{amssymb}
\usepackage{color}
\usepackage{graphicx}\usepackage{subfigure}
\usepackage{amsmath}
\usepackage{amsfonts}
\usepackage{cite}
\usepackage{epstopdf}
\usepackage[title]{appendix}
\allowdisplaybreaks
\usepackage{psfrag}
\usepackage{paralist}
\usepackage[colorlinks=true]{hyperref}

\usepackage{multirow}

\newtheorem{theorem}{Theorem}[section]
\renewcommand{\vec}[1]{\mbox{\boldmath \small $#1$}}
\newtheorem{example}{Example}[section]
\newtheorem{lemma}{Lemma}[section]
\newtheorem{remark}{Remark}[section]
\numberwithin{equation}{section}
\numberwithin{figure}{section}
\numberwithin{table}{section}

\newenvironment{proof}[1][Proof]{\begin{trivlist}
\item[\hskip \labelsep {\bfseries #1}]}{\end{trivlist}}
\renewcommand{\qed}{\hfill \nobreak \ifvmode \relax \else
      \ifdim\lastskip<1.5em \hskip-\lastskip
      \hskip1.5em plus0em minus0.5em \fi \nobreak
      \vrule height0.75em width0.5em depth0.25em\fi}
\begin{document}

\begin{frontmatter}

\title{A direct Eulerian GRP scheme for radiation hydrodynamical equations in  diffusion limit}

  \author{Yangyu Kuang}
 \ead{kyy@pku.edu.cn}
\address{HEDPS, CAPT \& LMAM, School of Mathematical Sciences, Peking University,
Beijing 100871, P.R. China}
\author[label2]{Huazhong Tang}
\thanks[label2]{Corresponding author. Tel:~+86-10-62757018;
Fax:~+86-10-62751801.}
\ead{hztang@math.pku.edu.cn}
\address{HEDPS, CAPT \& LMAM, School of Mathematical Sciences, Peking University,
Beijing 100871, P.R. China; School of Mathematics and Computational Science,
 Xiangtan University, Hunan Province, Xiangtan 411105, P.R. China}
 \date{\today{}}

\maketitle

\begin{abstract}
The paper proposes a second-order accurate direct Eulerian generalized Riemann
problem (GRP) scheme for the radiation hydrodynamical equations (RHE) in the zero diffusion limit.
The  difficulty comes from no explicit expression of the flux in terms of the conservative vector.
The characteristic fields and the relations between the left and right states across
the elementary-waves are first studied, and then the solution of the one-dimensional Riemann problem
is analyzed and given.
Based on those, the direct Eulerian GRP scheme is derived by directly using the generalized Riemann invariants and the Runkine-Hugoniot jump conditions to analytically resolve the left
and right nonlinear waves of the local GRP in the Eulerian formulation.
 Several numerical examples show that the GRP scheme can achieve second-order accuracy and high resolution
 of strong discontinuity.
\end{abstract}

\end{frontmatter}


\section{Introduction}
\label{sec:intro}
The radiation hydrodynamical equations (RHE) frequently appear in astrophysics \cite{Mihalas1984}.
In the zero diffusion limit, they can be written as a hyperbolic system of conservation laws, but
are different from the Euler equations of gas dynamics.
One of major difficulties of standard numerical methods for the RHEs is to
resolve and keep track of strong shock waves, while another major difficulty
comes from no explicit expression of the flux in terms of the conservative vector  due to the nonlinearity in the radiation pressure and energy.
 Dai and Woodward \cite{Dai-Woodward1998}
proposed the Godunov type schemes including linear and nonlinear Riemann solvers for  the RHEs.
The numerical results showed that such method kept the principle advantages of the Godunov scheme, but was
much time-consuming. The gas kinetic schemes are also extended to the RHEs in the zero diffusion limit.
Tang and Wu  first studied the kinetic flux vector splitting schemes \cite{Tang2000Kinetic}, and
then Jiang and Sun \cite{Jiang2007} extended it to the second-order BGK schemes.

The GRP scheme, as an analytic high-order accurate extension of the Godunov method,
was originally devised for compressible fluid dynamics \cite{Ben-Falcovitz1984}
by utilizing a piecewise linear function to approximate the ``initial'' data
and then analytically resolving a local GRP at each interface to yield numerical flux,
see the comprehensive description in \cite{Ben-Falcovitz3003}.
There exist two versions of the original GRP scheme: the Lagrangian and Eulerian.
The Eulerian version is always derived by using the Lagrangian framework with a transformation,
which is quite delicate, particularly for the sonic  and multi-dimensional cases.
To avoid those difficulties, second-order accurate direct Eulerian GRP schemes were respectively developed for the shallow water equations \cite{Li2006}, the Euler equations \cite{Ben-Li-Warnecke2006},
and a more general weakly coupled system \cite{Ben-Li2007} by directly resolving the local GRPs
in the Eulerian formulation via the {generalized} Riemann invariants and Rankine-Hugoniot jump conditions.
A comparison of the GRP scheme with the gas-kinetic scheme showed the good performance
of the GRP solver for some inviscid flow simulations \cite{Li2011}.
Combined the GRP scheme with the moving mesh method \cite{Tang2003},
the adaptive direct Eulerian GRP scheme was well developed in \cite{Han2010}
with improved resolution as well as accuracy.
 The accuracy and performance of the adaptive GRP scheme were further studied in \cite{Han2011}
 in simulating two-dimensional complex wave configurations formulated
 with the two-dimensional Riemann problems of Euler equations.
 The adaptive GRP scheme was also extended to unstructured triangular meshes \cite{Li2013}.
 The second order and third-order accurate extensions of the direct Eulerian GRP scheme
 were recently studied for one- and two-dimensional non-relativistic and relativistic Euler equations in \cite{Yang-He-Tang2011,Yang-Tang2012,Wu-Yang-Tang2014A,Wu-Yang-Tang2014B,Wu-Tang2016} and the general one-dimensional hyperbolic balance laws in \cite{Qian2014}. Moreover, a two-stage fourth order time-accurate GRP scheme was proposed for hyperbolic conservation laws in \cite{Li2016}.

The aim of the paper is to develop a second-order accurate direct Eulerian GRP scheme
for the RHEs. Due to the nonlinearity in the  radiation pressure and  energy,
 there is no explicit expression of the flux in
terms of the conservative vector and developing the GRP schemes for the RHEs becomes not trivial and much more technical than for the Euler equations of gas dynamics.
%
The paper is organized as follows.
Section \ref{sec:Preliminaries}
introduces the one-dimensional  RHEs in the zero diffusion limit,  gives corresponding eigenvalues and eigenvectors, Rankine-Hugoniot jump conditions, and {generalized} Riemann invariants as well as their total differentials,
and discusses  the characteristic fields.
The second-order accurate direct Eulerian GRP schemes are developed in Section \ref{sec:GRP}.
Section \ref{subsec:OUTLINE} gives the outline of the one-dimensional scheme.
Section \ref{subsec:GRP} resolves the local classical and  generalized Riemann probems analytically,
The two-dimensional extension of the GRP scheme is given in Section \ref{subsec:2d}. Several numerical experiments are conducted in Section \ref{sec:NE} to demonstrate the performance and accuracy of the proposed GRP schemes. Section \ref{sec:con} concludes the paper. 

\section{Preliminaries and notations}
\label{sec:Preliminaries}
This section introduces the one-dimensional radiation hydrodynamical equations (RHEs) in the zero diffusion limit,  gives corresponding eigenvalues and eigenvectors,  Rankine-Hugoniot jump conditions, and {generalized} Riemann invariants as well as their total differentials, and discusses  the characteristic fields.

\subsection{One-dimensional RHEs}
\label{subsec:gov}
If assuming that the radiative temperature is equal to the fluid temperature,
and the gas is (extremely) radiative opaque, that is, the equilibrium diffusion is dealt with
 and the mean-free-path of photons is much smaller than the typical length of flows,
then the one-dimensional  RHEs without the radiative heat diffusivity can be described as follows \cite{Mihalas1984}
\begin{equation}
\label{eq:RHE}
\frac{\partial \vec{U}}{\partial t}+\frac{\partial \vec{F(\vec{U})}}{\partial x}=0,
\end{equation}
where
\[
\vec{U}=(\rho,m,E)^T,\quad \vec{F(\vec{U}})=(\rho u, \rho u^2+p_{tot},u(E+p_{tot}))^T.
\]
Here the total pressure $p_{tot}:=p+\frac{1}{3}\hat{a}_RT^4$, the momentum $m=\rho u$, $\rho,u,E,p$ and $T$ denote the density, velocity, total energy, pressure and temperature, respectively,
 $\hat{a}_R$ is the radiation coefficient and assumed to be a nonnegative constant.
 The total energy  $E$  is described as  $E=\frac{1}{2}\rho u^2+\rho e_{tot}$,
  where the total specific internal energy $e_{tot}:=e+\hat{a}_R \rho^{-1}T^4$,   the specific internal energy $e$  satisfies $e=C_v T$, and
  $C_v$ is  the specific heat at constant volume and assumed to be a positive constant.
Without loss of
  generality,  $C_v$ is assumed to be 1   \cite{Dai-Woodward1998,Tang2000Kinetic,Jiang2007}.
  The total specific enthalpy $h_{tot}$ can be written into
$h_{tot}:=e_{tot}+\rho^{-1}p_{tot}$.
The terms $\frac{1}{3}\hat{a}_RT^4$ and $\hat{a}_RT^4$ denote the radiation pressure and  energy, respectively.
 To close the above system, the equation of state (EOS) is needed, $p=p(\rho,e)$.
 The paper only considers the EOS for the perfect gas, that is, $p=(\gamma -1)\rho e$, where $\gamma$ denotes adiabatic index.
 If the radiation coefficient $\hat{a}_R$ is zero, then the RHEs \eqref{eq:RHE}
reduces to the one-dimensional Euler equations in gas dynamics.
The readers are referred to \cite{Mihalas1984} for the details of the RHEs.

For the purpose of numerically solving the RHEs \eqref{eq:RHE}, 
the ``primitive'' variable vector $\vec W=(\rho,u,p_{tot})^T$ can be first calculated from the known conservative vector $\vec{U}$ at each time step in order to calculate the values of flux $\vec{F(\vec{U}})$ and corresponding eigenvalues etc.
 It is  trivial for the Euler equations in  gas dynamics, but becomes nontrivial for \eqref{eq:RHE} via
 (numerically) solving a quartic equation in terms of the specific internal energy
\[
\vec{E}=\frac{1}{2}\rho^{-1}m^2+\rho e+\frac{\hat{a}_R e^4}{C_v^{4}},
\]
by any standard root-finding algorthm, e.g. Newton's method, to get the specific internal energy $e$,
where the values of $\vec{U}$ are given. Fortunately,
 the right hand side of the equation is convex so that the Newton{'s} method converges with any positive initial guess. As soon as the value of $e$ is gotten, the value of the total pressure $p_{tot}$ can be calculated.

When the solution of the RHEs \eqref{eq:RHE} is smooth, it can be cast in the quasi-linear form
\begin{equation}
\label{eq:RHEW}
\frac{\partial \vec{W}}{\partial t}+\vec{\tilde{A}}(\vec{W})\frac{\partial \vec{W}}{\partial x}=0,
\end{equation}
with
\[
\vec{\tilde{A}}(\vec{W})=
\left( \begin{array}{ccc}
u & \rho & 0 \\
0 & u & \frac{1}{\rho} \\
0 & \rho c^2 & u \\
\end{array} \right),
\]
where $c$ is the speed of sound given by
\begin{equation}
\label{eq:c}
c:=\sqrt{\left(\frac{\partial p_{tot}}{\partial \rho}\right)_{S}}=\sqrt{\frac{p}{\rho}+\frac{\hat{K}^2}{\rho e \tilde{K}}},\quad \hat{K}:=p+\frac{4\hat{a}_Re^4}{3C_v^4}, \quad \tilde{K}:=\rho+\frac{4\hat{a}_Re^3}{C_v^4}.
\end{equation}
Here $S$ denotes the entropy related to the other thermodynamical variables via the
thermodynamic relation
\[
TdS=de-\frac{p}{\rho^2}d\rho,
\]
and the definition and meaning of the entropy are similar to those in the Euler equations, see \cite{Mihalas1984},
so that one has
\begin{equation}
\label{eq:S}
S=C_v\ln\frac{p}{\rho^{\gamma}}+\frac{4\hat{a}_Re^3}{3\rho C_v^3}.
\end{equation}
Three real eigenvalues and corresponding right eigenvectors of matrix $\vec{A}(\vec{W})$ can be derived as follows
\begin{equation}\label{eq:lambda}
\lambda_1=u-c,\quad \lambda_2=u,\quad \lambda_3=u+c,
\end{equation}
and
\begin{equation}\label{eq:R}
\vec{\tilde{R}}_1=\left(\begin{array}{c}
1 \\
-\frac{c}{\rho} \\
c^2 \\
\end{array}\right),\
\vec{\tilde{R}}_2=\left(\begin{array}{c}
1 \\
0 \\
0 \\
\end{array}\right), \
\vec{\tilde{R}}_3=\left(\begin{array}{c}
1\\
\frac{c}{\rho} \\
c^2\\
\end{array}\right).
\end{equation}

\subsection{Characteristic fields, and relations between the left and right states across
the elementary-waves}
\label{subsec:cha}
For the one-dimensional RHE \eqref{eq:RHE}, three characteristic fields defined by $\lambda_i$ \cite[Section 2.4.3]{Toro} satisfy the following properties, $i=1,2,3$.
\begin{lemma} \label{lem:cha}
(i) The second characteristic field  is linearly degenerate.
(ii) the first and third characteristic fields  are
genuinely nonlinear if  $1<\gamma<15.95$ but  not genuinely nonlinear if $\gamma>15.96$.
\end{lemma}
\begin{proof}
\begin{itemize}
\item[(i)] From \eqref{eq:lambda} and \eqref{eq:R},   
it is easy to verify that  the  eigenvalue $\lambda_2$  and corresponding eigenvector
$\vec{\tilde{R}}_2$
satisfy
\[
\nabla_{\vec{W}}\lambda_2(\vec{W}) \cdot \vec{\tilde{R}}_2(\vec{W})=0,
\]
where $\nabla_{\vec{W}}\lambda_2(\vec{W})=(\partial \lambda_2/\partial \rho, \partial \lambda_2/\partial u,
\partial \lambda_2/\partial p_{tot})$.
It tells us that the second characteristic field  is linearly degenerate.

\item[(ii)] For the first  characteristic field, 
one has
\begin{equation}\label{eq.2.0001}
\begin{split}
\nabla_{\vec{W}}\lambda_1(\vec{W}) \cdot \vec{\tilde{R}}_1(\vec{W})
 =-\frac{c}{\rho}-\frac{\partial c}{\partial \rho}-c^2\frac{\partial c}{\partial p_{tot}}
 =-\frac{\partial \rho c}{\rho\partial \rho}-c^2\frac{\partial c}{\partial p_{tot}},
\end{split}
\end{equation}
where $c$ has been considered as a function of density and {total pressure}, i.e. $c=c(\rho,p_{tot})$.

For the sake of convenience, denote  $\gamma_{1}:=\gamma-1$, and
    $
    r_{e}:={\hat{a}_{R}T^{4}}/{\rho e},
    $
which denotes the ratio of radiation energy to internal energy.
Thus the product $\rho c$ can be written as follows
\[
\rho c =\frac{1}{3}\rho\sqrt{e}\sqrt{\frac{(9\gamma_{1}^2 + 60\gamma_{1}r_{e} + 9\gamma_{1} + 16\gamma_{1}^2)}{4r_{e} + 1}}.
\]
Taking the first order partial derivative of $\rho c$ with respect to $\rho$ gives
    \[
    \frac{\partial \rho c}{\partial \rho}=\frac{(540\gamma_{1}^3r_{e}+1152\gamma_{1}^2r_{e}^2+1344\gamma_{1}r_{e}^3+256r_{e}^4+27\gamma_{1}^3+72\gamma_{1}^{2}r_{e}
    +72\gamma_{1}r_{e})\sqrt{e}}
    {6(4r_{e}+1)^{\frac{3}{2}}(3\gamma_{1}+4r_{e})\sqrt{9\gamma_{1}^2+60\gamma_{1}r_{e}+16r_{e}^2+9\gamma_{1}}},
    \]
   and taking the first order partial derivative of $c$ with respect to {$p_{tot}$} yields
    \[
    \frac{\partial c}{\partial p_{tot}}=\frac{256r_{e}^3+240\gamma_1r_{e}^2+112r_{e}^2+168\gamma_1r_{e}-72\gamma_1^2r_{e}+9\gamma_1^2+9\gamma_1}
    {2\rho(4r_{e}+1)^{\frac{3}{2}}(3\gamma_{1}+4r_{e})\sqrt{e}\sqrt{9\gamma_1^2+60\gamma_1 r_{e}+16r_{e}^2+9\gamma_1}}.
    \]
Substituting them into \eqref{eq.2.0001} gives
    \begin{align*}
\nabla_{\vec{W}}\lambda_{1}(\vec{W})& \cdot\vec{\tilde{R}}_{1}(\vec{W})
=-\frac{f(r_{e},\gamma_{1}) \sqrt{e}}
    {18\rho(4r_{e}+1)^{\frac{5}{2}}\sqrt{9\gamma_{1}^2+60\gamma_{1}r_{e}+16r_{e}^2+9\gamma_{1}}},
    \end{align*}
where
 \begin{align*}
    f(r_{e},\gamma_{1}):=& (-216r_{e}+27)\gamma_{1}^3+(1728r_{e}^2+1080r_{e}+81)\gamma_{1}^2
    \\ & +(7488r_{e}^3+4464r_{e}^2+756r_{e}+54)\gamma_{1}+1792r_{e}^4+640r_{e}^3.
     \end{align*}
whose first order partial derivative with respect to  $\gamma_{1}$ is equal to
    \[
    \frac{\partial f}{\partial \gamma_{1}}=81(-8r_{e}+1)\gamma_1^2
    +18(192r_{e}^2+120r_{e}+9)\gamma_1+7488r_{e}^3+4464r_{e}^2+756r_{e}+54.
    \]

In the following,  the zero of $f(r_{e},\gamma_{1})$ is discussed case by case.

{\tt Case (i):  $0\leq r_{e}\leq \frac{1}{8}$.}
It is obvious  that  $\frac{\partial f}{\partial \gamma_{1}}$ is always bigger than $0$ for
 $\gamma_{1}\in [0,\infty)$
 so that 
  $f(r_{e},\gamma_{1})>f(r_{e},0)>0$.

{\tt Case (ii):  $r_{e}>\frac{1}{8}$}.
 Because
    \[
    \frac{\partial f}{\partial \gamma_{1}}(r_{e},0)>0, \quad \frac{\partial^2 f}{\partial \gamma_{1}^2}=162(-8r_{e}+1)\gamma_{1}
    +54(64r_{e}^2+40r_{e}+3),
    \]
    the function $\frac{\partial f}{\partial \gamma_{1}}$
    is  first  increasing  and then  decreasing monotonically as $\gamma_1$ increases, and has only one positive zero with respect
    to $\gamma_{1} \in [0,\infty)$.
  Thus the function  $f(r_{e},\gamma_{1})$ is also  first  increasing  and then  decreasing monotonically in
  the interval $[0,\infty)$ as $\gamma_1$ increases.
On the other hand, because  $f(r_{e},0)>0$, the function $f(r_{e},\gamma_{1})$ has only one positive zero with respect
    to $\gamma_{1}$.
    Hence, if there exist some $r_{e}\geq0$ and $\tilde{\gamma}_{1}$ { such that $f(r_{e},\tilde{\gamma}_{1})=0$}, then
   the inequality $f(r_{e},\gamma_{1})<0$ holds  for such $r_{e}$ and all $\gamma_{1}$ bigger than $\tilde{\gamma}_{1}$.
It implies that
    \[
    \mathcal{R}:=\{\gamma_{1}|f(r_{e},\gamma_{1})>0,\forall r_{e}\geq0\}=\{\gamma_{1}|0<\gamma_{1}<\hat{\gamma}_{1}\}
    \subset [0,\infty),
    \]
    where $\hat{\gamma}_{1}:={\inf}\{\gamma_{1}|\exists r_{e}\geq0,~\mbox{\rm s.t. } f(r_{e},\gamma_{1})<0\}$.
    So one may observe the maximum real zero of $f(r_{e},\gamma_{1})$
with respect to the independent variable $r_{e}$, and denote it by $r_{e}^{\max}$. If $r_{e}^{\max}\geq0$, then
there exist some $r_{e}=r_{e}^{\max}\geq0$ and $\tilde{\gamma}_{1}$ { such that $f(r_{e}^{\max},\tilde{\gamma}_{1} )=0$}
so that $\gamma_{1}\notin\mathcal{R}$.
Fig.~\ref{Fig.2.1} shows the maximum real zero of $f(r_{e},\gamma_{1})$
with respect to $r_{e}$.  It is seen that $\hat{\gamma}_{1}\in(14.95,14.96)$.

Combing the above two cases gives that
if  $1<\gamma<15.95$, then $f(r_{e},\gamma_{1})>0$ for all $r_{e}$ bigger than zero
so that  $\nabla_{\vec{W}}\lambda_1(\vec{W}) \cdot \vec{\tilde{R}}_1(\vec{W})<0$  for any physically admissible state $\vec W$.
Thus  the first characteristic field is
genuinely nonlinear when $1<\gamma<15.95$, but not  genuinely nonlinear when $\gamma>15.96$.

\item[(iii)] The conclusion on the third characteristic field can be similarly verified. The proof is completed.
    \end{itemize}
  \qed\end{proof}
\begin{figure}[htb]
  \centering
  \includegraphics[scale=0.45]{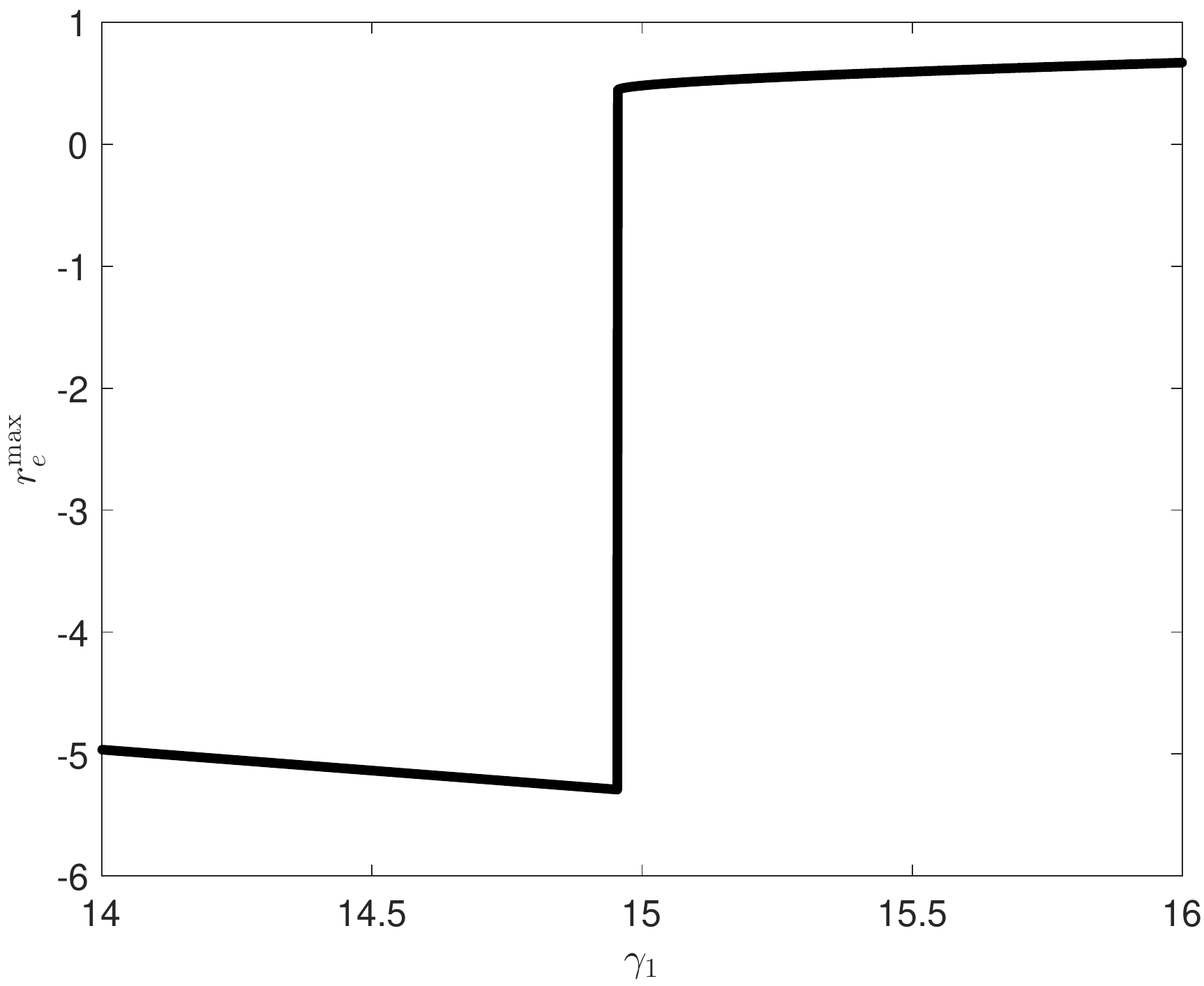}
  \caption{The maximum real zero of $f(r_{e},\gamma_{1})$ with respect to $r_{e}$.}
  \label{Fig.2.1}
\end{figure}

Let us study the relations between the left and right states across
the elementary-waves, which need the generalized Riemann invariants and
the Rankine-Hugoniot relations.

The generalized Riemann invariants provide a powerful tool of analysis of hyperbolic
conservation laws. For the RHEs \eqref{eq:RHE}, the generalized Riemann invariants corresponding to three characteristic fields can be
calculated as follows
\begin{equation}\label{eq:Riemanninvar}
\begin{aligned}
\lambda_1=u-c: &\quad S, \quad &\psi_{1}:=u+\int^{\rho}\frac{c(\omega,S)}{\omega}d\omega, \\
\lambda_2=u:&\quad u, &p_{tot}, \\
\lambda_3=u+c: &\quad S, \quad &\psi_{3}:=u-\int^{\rho}\frac{c(\omega,S)}{\omega}d\omega.
\end{aligned}
\end{equation}
whose total differentials will be presented later.
For a discontinuous wave solution of speed $s$,
 the Rankine-Hugoniot relations of the RHEs \eqref{eq:RHE} state
\begin{equation}
\label{rh}
\begin{aligned}
&\rho_L \hat{u}_L=\rho_R \hat{u}_R, \\
&\rho_L \hat{u}^2_L+p_{tot,L}=\rho_R \hat{u}^2_R+p_{tot,R}, \\
&\hat{u}_L(\hat{E}_{L}+p_{tot,L})=\hat{u}_R(\hat{E}_{R}+p_{tot,R}),
\end{aligned}
\end{equation}
where $\hat{u}_L:=u_L-s$, $\hat{u}_R:=u_R-s$, $\hat{E}_{L}:=\rho_Le_{tot,L}+\frac{1}{2}\rho_L \hat{u}^2_L$, $\hat{E}_{R}:=\rho_Re_{tot,R}+\frac{1}{2}\rho_R \hat{u}^2_R$,
and the subscript $L$ and $R$ denote the left and right (limited) states of  the physical variables
across the discontinuity.

Based on the above generalized Riemann invariants and Rankine-Hugoniot relations,
 the following theorem can be established.
\begin{theorem}\label{thm:judge}
 If $\gamma \in { (1, 14.39]}$,
then the following conclusions hold:

{\tt (i)} The left and right states across the 1-wave satisfy
\begin{equation}
\begin{array}{cc}
p_{tot,L}<p_{tot,R},\quad u_L>u_R, & \mbox{ for shock}, \\
p_{tot,L}>p_{tot,R},\quad u_L<u_R, & \mbox{ for rarefaction}.
\end{array}
\end{equation}
{\tt (ii)} For the contact discontinuity, one has
\begin{equation}\label{eq:contact} p_{tot,L}=p_{tot,R},\quad u_L=u_R.\end{equation}

{\tt (iii)} The left and right states across the 3-wave satisfy
\begin{equation}
\begin{array}{cc}
p_{tot,L}>p_{tot,R},\quad u_L>u_R, & \mbox{ for shock}, \\
p_{tot,L}<p_{tot,R},\quad u_L<u_R, & \mbox{ for rarefaction}.
\end{array}
\end{equation}

\end{theorem}

\begin{proof}
(i) For the 2-wave, i.e. contact discontinuity,
because  $p_{tot}$ and $u$ are two $2$-Riemann invariants, see \eqref{eq:Riemanninvar},
two identities in \eqref{eq:contact} hold.

(ii) If assuming that the 1-wave is the rarefaction wave, then the values of $\lambda_1$
at the left- and right-hand sides of the rarefaction wave satisfy
\[
 u_R-c_R>u_L-c_L,
\]
and the second 1-Riemann invariant in \eqref{eq:Riemanninvar}
satisfies
\begin{equation}\label{120}
u_L-u_R+\int_{\rho_R}^{\rho_L}\frac{c}{\rho}d\rho=0.
\end{equation}
Combing them gives
\begin{equation}\label{121}
c_L-c_R+\int_{\rho_R}^{\rho_L}\frac{c}{\rho}d\rho > 0.
\end{equation}

On the other hand,  within the rarefaction wave,
one has $S=S_L$ due to  the first 1-Riemann invariant $S$  in \eqref{eq:Riemanninvar},
so that the total pressure $p_{tot}$ can
be considered as  a function of a single variable $\rho$  within this rarefaction wave,
i.e.  $p_{tot}=p_{tot}(\rho)$.
 Thus within the first rarefaction wave,  $c(\rho, p_{tot})$ can also be expressed
 as a function of a single variable $\rho$, i.e. $c=c(\rho)$, and one has
\[
c_L-c_R=\int_{\rho_R}^{\rho_L}c'(\rho)d\rho.
\]
Substituting the above identity into \eqref{121} gives
\begin{equation}\label{123}
\int^{\rho_L}_{\rho_R}\left(c'(\rho)+\frac{c}{\rho}\right)d\rho > 0.
\end{equation}
Because of the proof of Lemma \ref{lem:cha},
the inequality
\[
    c'(\rho)+\frac{c}{\rho}=\frac{\partial \rho c(\rho,p_{tot})}{\rho\partial \rho}+\frac{\partial c(\rho,p_{tot})}{\partial p_{tot}}\frac{dp_{tot}}{d\rho}=-\nabla_{\vec{W}}\lambda_{1}(\vec{W})\cdot\vec{\tilde{R}}_{1}(\vec{W})>0,
    \]
    holds for   any $\rho > 0$ and $\gamma\in (1,14.39]\subset (1,15.95)$.
    Combing the above inequality with  \eqref{123} gives
    $$\rho_{L}>\rho_{R}.$$
    Due to \eqref{eq:c}, $p_{tot}'(\rho)>0$, which implies
\[
p_{tot,L} > p_{tot,R}.
\]
Moreover, the inequality $\rho_L > \rho_R$  and \eqref{120}
imply
\[
u_L<u_R.
\]

(iii) For the $1$-shock wave, using the Rankine-Hugoniot  conditions \eqref{rh} gives
 \begin{equation}\label{hc}
    \frac{1}{2}\hat{u}_{L}^2+h_{L}=\frac{1}{2}\hat{u}_{R}^2+h_{R},
    \end{equation}
    and
     \begin{equation}\label{ec}
    e_{tot,L}-e_{tot,R}=\frac{1}{2}(p_{tot,L}+p_{tot,R})\frac{\rho_L-\rho_R}{\rho_L \rho_R}.
    \end{equation}
On the other hand, the Lax shock wave conditions
\begin{equation}\label{EQ.2.17}
u_L-c_L>s>u_R-c_R,\quad s<u_R,
\end{equation}
gives that $\hat{u}_L>c_L>0$ and $0<\hat{u}_R<c_R$.

The proof by contradiction is considered here.
Assume  $p_{tot,L}\geq p_{tot,R}$.
Using such assumption and the first and second identities in \eqref{rh} gives
\[
\hat{u}_L \leq \hat{u}_R.
\]
Combing it with \eqref{EQ.2.17}
yields
\[
u_L \leq u_R,\quad 0<c_L<\hat{u}_L \leq \hat{u}_R<c_R.
\]
Again using the first identity in \eqref{rh} gives
\begin{equation}\label{33111111111}
\rho_L \geq \rho_R.
\end{equation}
Substituting it into \eqref{ec}
derives
\begin{equation}\label{331}
e_{tot,L} \geq e_{tot,R},\quad \rho_Le_{tot,L} \geq \rho_Re_{tot,R}.
\end{equation}

The following  attempts to find the contradiction to \eqref{331} in   two cases.

 {\tt Case (i):  $\gamma_{1}>\frac{1}{3}$}.
 Because $0<c_L<\hat{u}_L \leq \hat{u}_R<c_R$, and \eqref{hc}, one has
    \[
    \frac{1}{2}c_{L}^2+h_{L}<\frac{1}{2}c_{R}^2+h_{R}.
    \]
Consider $\frac{c^2}{2}+h$ as a function of $\rho$ and $e_{tot}$,
and then one has
    \[
    \frac{\partial (\frac{c^2}{2}+h)}{\partial \rho}=\frac{ r_{e}e(3\gamma_{1}-1)(24\gamma_{1}r_{e}+176r_{e}^2+15\gamma_{1}+80r_{e}+6)}{18(4r_{e}+1)^3\rho}>0,
    \]
    \[
    \frac{\partial (\frac{c^2}{2}+h)}{\partial e_{tot}}=\frac{ g(r_{e},\gamma_{1})}
    {18(4r_{e}+1)^3},
    \]
where
    \[
    g(r_{e},\gamma_{1}):=9(-8r_{e}+1)\gamma_{1}^2+(528r_{e}^2+312r_{e}+27)\gamma_{1}+1792r_{e}^3+1168r_{e}^2+240r_{e}+18.
    \]

If $r_{e}<\frac{1}{8}$, then $g(r_{e},\gamma_{1})$ is always positive; otherwise,
because
        \[
    g(r_{e},0)>0, \quad \frac{\partial g}{\partial \gamma_{1}}=18(-8r_{e}+1)\gamma_{1}+528r_{e}^2+312r_{e}+27,
    \]
   the function $g(r_{e},\gamma_{1})$
   is first   increasing and then decreasing monotonically as
 $\gamma_{1}$ increases.
 Thanks to the fact that $g(r_{e},0)>0$,
 $g(r_{e},\gamma_{1})$   has only one positive zero with respect to $\gamma_{1}$.
If there exists $\tilde{\gamma}_{1}$ and $r_{e}$  such that $g(r_{e},\tilde{\gamma}_{1})=0$,
then for all $\gamma_{1}>\tilde{\gamma}_{1}$ and such $r_{e}$, $g(r_{e},\gamma_{1})<0$.
Hence, one has
    \[
    \mathcal{\tilde{R}}:=\{\gamma_{1}|g(r_{e},\gamma_{1})>0,\forall r_{e}\}=\{\gamma_{1}|\frac{1}{3}<\gamma_{1}<\bar{\gamma}_{1}\},\quad \bar{\gamma}_{1}:={\inf}\{\gamma_{1}|\exists r_{e}, ~\mbox{\rm s.t. } g(r_{e},\gamma_{1})<0\}.
    \]
Let us observe the maximum real
zero of $g(r_{e},\gamma_{1})$ with respect to the independent variable $r_{e}$, and denote it by $r_{e}^{\max}$.
If $r_{e}^{\max}\geq0$, then  there exist some $r_{e} =r_{e}^{\max}\geq 0$ and
$\tilde{\gamma}_{1}$ such that $g(r_{e},\tilde{\gamma}_{1})=0$. It implies
 $\gamma_{1}\notin\mathcal{\tilde{R}}$.
Fig.~\ref{fig:2.2}  shows the maximum real zero of $g(r_{e},\gamma_{1})$ with respect to {$r_e$}.
It is seen that if $\bar{\gamma}_{1}\in(13.39,13.4)$, then
 $\frac{c^2}{2}+h$ is monotone increasing with respect to $\rho$ and $e_{tot}$ when $\frac{1}{3}<\gamma_{1}<\bar{\gamma}_{1}$, thus one has
    \[
    e_{tot,L}<e_{tot,R} \quad\mbox{ or}\quad \rho_L <\rho_R,
    \]
    which conflicts with {\eqref{33111111111} or \eqref{331}}.

\begin{figure}[htb]
  \centering
  \includegraphics[scale=0.45]{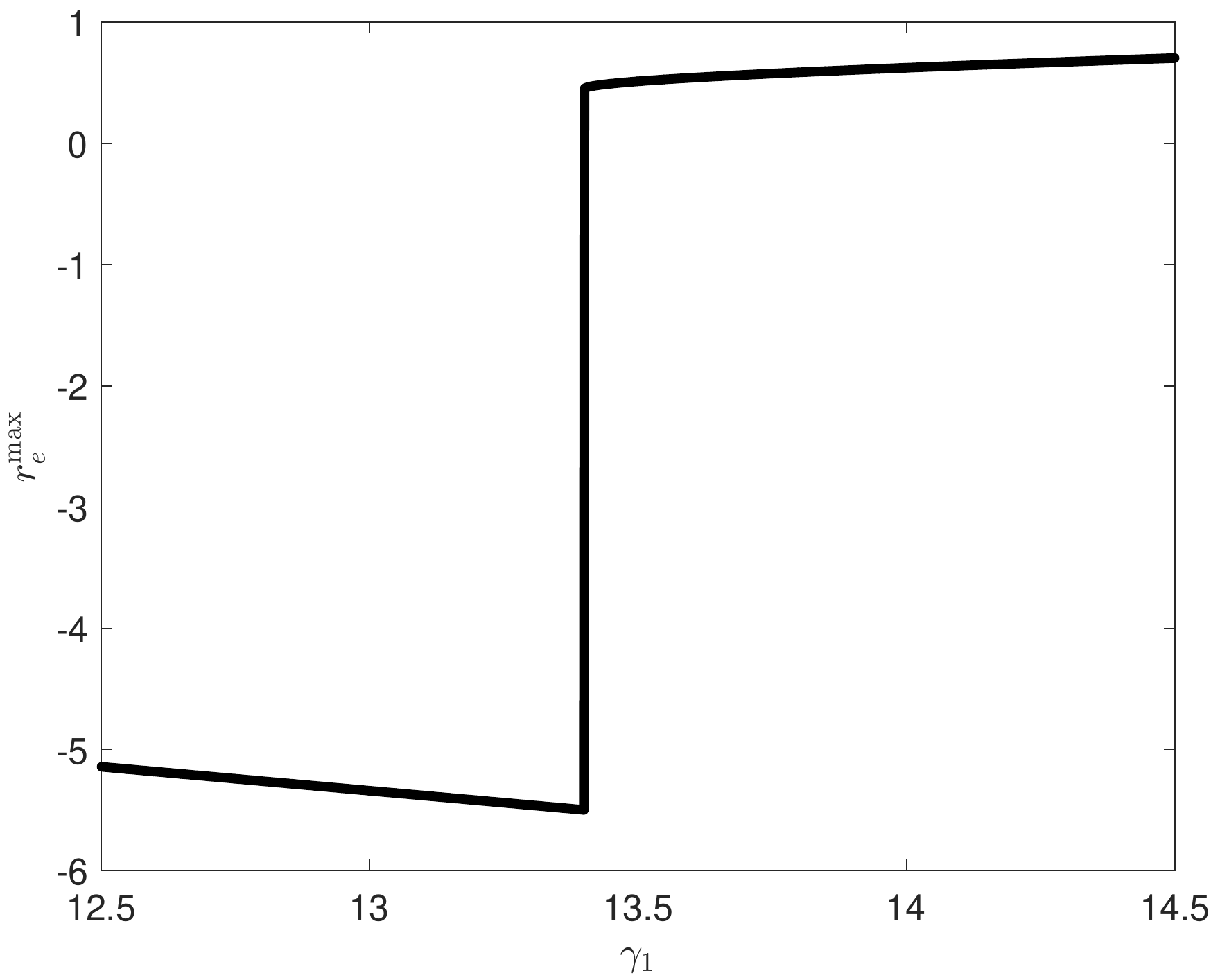}
  \caption{The maximum real zero of $g(r_{e},\gamma_{1})$ with respect to $r_{e}$.}
  \label{fig:2.2}
\end{figure}

{\tt Case (ii): $0<\gamma_{1}\leq\frac{1}{3}$}.
The first equation of \eqref{rh} and Lax conditions \eqref{EQ.2.17} gives
    \[
    {\rho c_{L}<\rho c_{R}}.
    \]
If $\rho c$ is considered as a function of $\rho$ and $e_{tot}$, then one has
    \[
    \frac{\partial (\rho c)}{\partial \rho}=\frac{512r_{e}^4+16(135\gamma_{1}+11)r_{e}^3+120\gamma_{1}(3\gamma_{1}+11)r_{e}^2
    +3\gamma_{1}(63\gamma_{1}+83)r_{e}+18\gamma_{1}(\gamma_{1}+1) }{6(4r_{e}+1)^{\frac{5}{2}} \sqrt{e} \sqrt{9\gamma_{1}^2+60\gamma_{1}r_{e}+16r_{e}^2+9\gamma_{1}}}>0,
    \]
    and
     \[
    \frac{\partial (\rho c)}{\partial e_{tot}}=\frac{\rho  \tilde{g}(r_{e},\gamma_{1})
    }{6(4r_{e}+1)^{\frac{5}{2}}\sqrt{9\gamma_{1}^2+60\gamma_{1}r_{e}+16r_{e}^2+9\gamma_{1}}\sqrt{e}},
    \]
where
    \[
    \tilde{g}(r_{e},\gamma_{1}):=(-72r_{e}+9)
    \gamma_{1}^2+(240r_{e}^2+168r_{e}+9)\gamma_{1}+256r_{e}^3+112r_{e}^2.
    \]
 Similarly,  if $r_{e}<\frac{1}{8}$, then $\tilde{g}(r_{e},\gamma_{1})$ is positive;
 otherwise,
because
        \[
    \tilde{g}(r_{e},0)>0, \quad \frac{\partial \tilde{g}}{\partial \gamma_{1}}=18(-8r_{e}+1)\gamma_{1}+240r_{e}^2+168r_{e}+9,
    \]
    $\tilde{g}(r_{e},\gamma_{1})$ is first increasing and then decreasing monotonically with respect to $\gamma_{1}$.
  On the other hand, because
    \[
    \tilde{g}(r_{e},\frac{1}{3})=4(4r_{e}+1)^3>0,
    \]
    $\tilde{g}(r_{e},\gamma_{1})>0$ for  $0<\gamma_{1}\leq\frac{1}{3}$, and
    $\rho c$ is monotone increasing with respect to $\rho$ and $e_{tot}$.
    Thus one has
    \[
    e_{tot,L}<e_{tot,R} \quad\mbox{or }\ \rho_L <\rho_R,
    \]
    which conflicts with {\eqref{33111111111} or \eqref{331}}.

In summary, one has
\[
p_{tot,L}< p_{tot,R}.
\]
Combing it with the first and second equations in \eqref{rh} yields $u_L>u_R$.

(iv) The conclusion on the third characteristic field can be similarly verified. The proof
is completed.\qed
\end{proof}

Before ending this section, the total differentials of the generalized Riemann invariants
are given here. For smooth
flows, the  RHEs \eqref{eq:RHE} can be equivalently written in  the following form
\begin{equation}
\label{eq:Dt}
\frac{D\rho}{Dt}+\rho\frac{\partial u}{\partial x}=0,\quad
\rho\frac{Du}{Dt}+\frac{\partial p_{tot}}{\partial x}=0,\quad
\frac{DS}{Dt}=0\ \mbox{ or }\ \frac{Dp_{tot}}{Dt}+\rho c^2\frac{\partial u}{\partial x}=0,
\end{equation}
where $\frac{D}{Dt}=\frac{\partial }{\partial t}+u\frac{\partial }{\partial x}$ is the material derivative operator.

Because the third equation in \eqref{eq:Dt} shows that the entropy $S$ is constant along the trajectory of each fluid particle, and the entropy $S$ is one of the i-Riemann invariants associated with the characteristic field $\lambda_{i}$, $i = 1$ or $3$, see \eqref{eq:Riemanninvar}, 
 all thermodynamic variables can be considered as functions of $\rho$ and $S$.
With the help of such consideration,  the total differentials of the i-Riemann invariants $\psi_{i} =\psi_{i}(u,{\rho},S)$ are
\begin{equation}
\label{eq:dpsi}
d\psi_{i}=\frac{c}{\rho}d\rho+\frac{\partial \psi_{i}}{\partial S}dS+du
=du+(-1)^{(i-1)/2}\left(\frac{1}{\rho c}dp_{tot}+K(\rho,S)dS\right),
 i = 1, 3,
\end{equation}
where   $\frac{\partial \psi_{i}}{\partial S}=
\int^{\rho}\frac{1}{\omega}\frac{\partial c(\omega,S)}{\partial S}d\omega$, and
\begin{equation}
\label{eq:K}
K(\rho,S)=-\frac{1}{\rho c}\frac{\partial p_{tot}(\rho,S)}{\partial S}+
\int^{\rho}_{0}\frac{1}{\omega}\frac{\partial c(\omega,S)}{\partial S}d\omega,
\end{equation}
with
\[
\frac{\partial p_{tot}(\rho,S)}{\partial S}=\left(\frac{\partial S(\rho,e)}{\partial e}\right)^{-1}\left(\frac{\partial p_{tot}(\rho,e)}{\partial e}\right),\quad \frac{\partial c(\rho,S)}{\partial S}=\left(\frac{\partial S(\rho,e)}{\partial e}\right)^{-1}\left(\frac{\partial c(\rho,e)}{\partial e}\right),
\]
and
\[
e=\frac{\rho^{\gamma-1}}{\gamma-1}\exp\left(-\frac{1}{3}C_v^{-1}\left({\rm LambertW}\left(\frac{4\hat{a}_{R}\rho^{3\gamma-4}\exp(3S/C_{v})}{C_{v}^4(\gamma-1)^3}\right)C_v-3S\right)\right),
\]
which is obtained by \eqref{eq:S},
 here the  function
${\rm LambertW}(x)$ is  the inverse  of $x\exp(x)$, so-called the Lambert W function {\cite{LambertW}}.
\begin{remark}
The integral $\int^{\rho}_{0}\frac{1}{\omega}\frac{\partial c(\omega,S)}{\partial S}d\omega$ in \eqref{eq:K} can not be exactly evaluated, so that some  numerical quadrature has to be considered.
Our work uses the QAGS adaptive integration with singularities in the GNU Scientific Library.
\end{remark}

\section{Numerical schemes}
\label{sec:GRP}
    This section gives the direct Eulerian GRP schemes for the one- and two-dimensional  RHEs.

\subsection{The outline of the one-dimensional GRP scheme}
\label{subsec:OUTLINE}
  This section gives the outline of the one-dimensional direct Eulerian GRP scheme.
  For the sake of simplicity, the domain $\Omega$ is divided into a uniform mesh
   $\{x_j=j\Delta x, j \in \mathbb{Z}\}$ and
   $I_{j+\frac{1}{2}}=[x_j,x_{j+1}]$ {denotes} the $(j+\frac12)$th cell.
   A (nonuniform) partition of   the time interval $[0,T]$ is also given   as
 $\{t_0=0, t_{n+1}=t_n+\Delta t_n, n \geq 0\}$, where the time step size
 $\Delta t_n$ is determined by
   \[
\Delta t_n=\frac{C_{cfl} \Delta x}{\max \limits_{i,j}\{\lambda_i(\vec{U}_{j+\frac{1}{2}}^n) |\}},
   \]
here $\vec{U}_{j+\frac{1}{2}}^n$  approximates the  cell average value of
of $\vec{U}(x,t_n)$  over the cell $I_{j+\frac{1}{2}}$ at time $t_n$, and
$C_{cfl}$ denotes the CFL number.

If assuming that the (initial) data at time $t_n$ are piecewise linear functions with slopes $(\vec{U'})_{j+\frac{1}{2}}^n$, i.e.
\[
\vec{U}_{h}(x, t_n)=\vec{U}_{j+\frac{1}{2}}^n+(\vec{U'})_{j+\frac{1}{2}}^n(x-x_{j+\frac{1}{2}}),\ x\in I_{j+\frac{1}{2}},
\]
then the solution vector $\vec{U}$ at time
$t_{n+1}$ of  \eqref{eq:RHE} can be approximately evolved by using a second order accurate Godunov-type
scheme
\begin{equation}\label{eq:GRPS}
\vec{U}_{j+\frac{1}{2}}^{n+1}=\vec{U}_{j+\frac{1}{2}}^n-\frac{\Delta t_n}{\Delta x}\left(\vec{F}(\vec{U}_{j+1}^{n+1/2})-\vec{F}(\vec{U}_{j}^{n+1/2})\right),
\end{equation}
where $\vec{U}_{j}^{n+1/2}$ denotes
a second-order approximation of $\vec{U}(x_j,t_n+\frac{1}{2}\Delta t_n)$,  and is analytically derived by
the second order accurate resolution of the local generalized Riemann problem (GRP) at each point $(x_{j},t_{n})$, i.e.
\begin{equation}\label{eq:GRP}
\left\lbrace\begin{array}{ll}
\eqref{eq:RHE},& t-t_n>0, \\
\vec{U}(x,t_n)=&\left\lbrace\begin{array}{lc}
\vec{U}_{j-\frac{1}{2}}^n+(\vec{U}')_{j-\frac{1}{2}}^n(x-x_{j-\frac{1}{2}}), & x-x_j<0, \\
\vec{U}_{j+\frac{1}{2}}^n+(\vec{U}')_{j+\frac{1}{2}}^n(x-x_{j+\frac{1}{2}}), & x-x_j>0.
\end{array}\right.
\end{array}\right.
\end{equation}
More specifically, $\vec{U}_{j}^{n+1/2}$ is usually calculated by
\begin{equation}\label{eq:jnp12}
\vec{U}_{j}^{n+\frac{1}{2}}=\vec{U}_{j}^{RP,n}+\frac{\Delta t_n}{2}\left(\frac{\partial \vec{U}}{\partial t}\right)_j^n,
\end{equation}
where the calculations of $\left(\frac{\partial \vec{U}}{\partial t}\right)_j^n$
is one of the key elements in the GRP scheme,
and will be completed by resolving the problem \eqref{eq:GRP},
see Sections \ref{subsubsec:RARE}-\ref{subsubsec:aco},
and $\vec{U}_{j}^{RP,n}:=\vec{\omega}(0;\vec{U}_{j,L}^{n},\vec{U}_{j,R}^{n})$, here
$\vec{\omega}(\frac{x-x_{j}}{t-t_n};\vec{U}_{j,L}^{n},\vec{U}_{j,R}^{n})$ is the exact solution of the (classical)
Riemann problem for \eqref{eq:RHE} centered at $(x_j,t_n)$, i.e. the Cauchy problem
\begin{equation}\label{eq:RP}
\left\lbrace\begin{array}{ll}
\eqref{eq:RHE},& t-t_n>0, \\
\vec{U}(x,t_n)=&\left\lbrace\begin{array}{lc}
\vec{U}_{j,L}^n, & x-x_j<0, \\
\vec{U}_{j,R}^n, & x-x_j>0,
\end{array}\right.  \\
\end{array}\right.
\end{equation}
and  $\vec{U}_{j,L}^n$ and $\vec{U}_{j,R}^n$ are the left and right limiting values of $\vec{U}_{h}(x)$ at $(x_j,t_n)$. The exact solution of Riemann problem \eqref{eq:RP} will be discussed in Section \ref{subsubsec:RP}.

After getting $\vec{U}_{j+\frac{1}{2}}^{n+1}$, the approximate slope
$\vec{U}'$ at time $t_{n+1}$ is evolved by
\begin{equation}\label{eq:slope}
(\vec{U}')^{n+1}_{j+\frac{1}{2}}=\vec{R}_{j+\frac{1}{2}}\mbox{minmod}\left(\frac{\theta}{\Delta x}\vec{R}_{j+\frac{1}{2}}^{-1}\left(\vec{U}_{j+\frac{1}{2}}^{n+1}-\vec{U}_{j-\frac{1}{2}}^{n+1}\right),
\vec{R}_{j+\frac{1}{2}}^{-1}(\vec{U}')^{n+1,-}_{j+\frac{1}{2}}, \frac{\theta}{\Delta x}\vec{R}_{j+\frac{1}{2}}^{-1}\left(\vec{U}_{j+\frac{3}{2}}^{n+1}-\vec{U}_{j+\frac{1}{2}}^{n+1}\right)\right),
\end{equation}
where $\vec{R}_{j+\frac{1}{2}}:=\vec{R}(\vec{U}_{j+\frac{1}{2}}^{n+1})$, the parameter  $\theta \in [1,2)$, and
\begin{equation*}
(\vec{U}')^{n+1,-}_{j+\frac{1}{2}}=\frac{1}{\Delta x}(\vec{U}_{j+1}^{n+1,-}-\vec{U}_{j}^{n+1,-}), \
\vec{U}_{j}^{n+1,-}:=\vec{U}_{j}^{RP,n}+\Delta t_n\left(\frac{\partial \vec{U}}{\partial t}\right)_{j}^{n+1}
=2\vec{U}_{j}^{n+\frac{1}{2}}-\vec{U}_{j}^{RP,n}.
\end{equation*}
Here the limiter function is used to suppress possible numerical oscillations near the discontinuities, and
$\vec{R}(\vec{U})$ is the right eigenvector matrix of the Jacobian matrix $\frac{\partial \vec{F}(\vec{U})}{\partial \vec{U}}$ as follows
\[
\vec{R}(\vec{U})=\begin{pmatrix}
1&1&1\\
u-c&u&u+c\\
H-uc&\frac{u^2}{2}-(4\gamma-\frac{16}{3})\frac{\hat{a}_{R}e^5}{C_{v}^{4}\hat{K}} &H+uc
\end{pmatrix},
\]
where $H:=\frac{E+p_{tot}}{\rho}$ denotes the total enthalpy.

\subsection{Resolution of classical and generalized Riemann problems}\label{subsec:GRP}
This section derives  $\vec{U}_{j}^{RP,n}$ and $\left(\frac{\partial \vec{U}}{\partial t}\right)_{j}^{n}$ in \eqref{eq:jnp12} by respectively resolving the local classical Riemann problem \eqref{eq:RP} and GRP \eqref{eq:GRP},
which can be changed into
\begin{equation}\label{eq:RP1}
\left\lbrace\begin{array}{ll}
\eqref{eq:RHE},& t>0, \\
\vec{U}(x,0)=&\left\lbrace\begin{array}{lc}
\vec{U}_{L}, & x<0, \\
\vec{U}_{R}, & x>0,
\end{array}\right.
\end{array}\right.
\end{equation}
and
\begin{equation}\label{eq:GRP1}
\left\lbrace\begin{array}{ll}
\eqref{eq:RHE},& t>0, \\
\vec{U}(x,0)=&\left\lbrace\begin{array}{lc}
\vec{U}_{L}+x\vec{U}'_L, & x<0, \\
\vec{U}_{R}+x\vec{U}'_R, & x>0,
\end{array}\right.
\end{array}\right.
\end{equation}
by a coordinate transformation and ignoring the subscript $j$ in the initial data,
where $\vec{U}_{L}, \vec{U}_{R}, \vec{U}'_L$ and $\vec{U}'_R$ are corresponding
constant vectors.

The initial structure of the solution $\vec{U}^{GRP}(x, t)$ of \eqref{eq:GRP1} is determined by the solution $\vec{\omega}(x/t;\vec{U}_L,\vec{U}_R)$
 of the Riemann problem  \eqref{eq:RP1},
 and
\[
\lim_{t\rightarrow0}\vec{U}^{GRP}(t\lambda, t)= \vec{\omega}(0;\vec{U}_L,\vec{U}_R)=:\vec{U}^{RP},\quad x=t\lambda.
\]
The local wave configuration around the singularity point $(x, t) = (0, 0)$ of the GRP \eqref{eq:GRP1} is usually piecewise smooth and consists
of the elementary waves such as rarefaction wave, the shock wave and the contact discontinuity, as the schematic descriptions shown in Figs. \ref{Fig:3.3} and \ref{Fig:3.4}.

The flow is isentropic for the rarefaction waves as a part of the solution $\vec{\omega}(x/t;\vec{U}_L,\vec{U}_R)$ of the  Riemann problem \eqref{eq:RP1}, so the generalized Riemann invariants $\psi_{i}$ and $S$ are constant and their derivatives vanish within the $i$th-rarefaction wave, $i = 1, 3$. Unfortunately, those  properties  do not hold for the GRP \eqref{eq:GRP1} generally, because the general (curved) rarefaction waves have to be considered in the solution of the GRP \eqref{eq:GRP1}. However,  the solution of \eqref{eq:GRP1} can be regarded as a perturbation of the solution of   \eqref{eq:RP1} when $0\leq t\ll1$, so that it can still be expected that $\psi_{i}$ and $S$ are regular within the $i$th-rarefaction wave as a part of the solution $\vec{U}^{GRP}(x, t)$ of \eqref{eq:GRP1} around the singularity point $(x, t) = (0, 0)$,  $i = 1, 3$, and the generalized Riemann invariants are used to resolve the rarefaction waves around the singularity point $(x, t) = (0, 0)$.

From now on,  our attention is restricted to the local wave configuration displayed in Figs. \ref{Fig:3.3} and \ref{Fig:3.4}, where a rarefaction wave moves to the
left and a shock moves to the right, the intermediate region is separated by a contact discontinuity,
and the intermediate states in the two subregions are denoted by $\vec{U}_1$ and $\vec{U}_2$, respectively. Other local wave configurations can be similarly considered.
Finally, we denote by $\vec{U}_{*}$, $\left(\frac{\partial \vec{U}}{\partial t}\right)_{*}$, and $\left(\frac{D \vec{U}}{D t}\right)_{*}$ the limiting states at
$x = 0$, as $t\rightarrow0+$. The major feature of the direct Eulerian GRP scheme is to form a linear algebraic system
\begin{equation}\label{eq:DutDpt}
\left\lbrace \begin{aligned}
a_L\left(\frac{D u}{D t}\right)_*+b_L\left(\frac{D p_{tot}}{D t}\right)_*=d_L, \\
a_R\left(\frac{D u}{D t}\right)_*+b_R\left(\frac{D p_{tot}}{D t}\right)_*=d_R,
\end{aligned} \right.
\end{equation}
by resolving respectively the left and  right waves  in Fig. \ref{Fig:3.3}. Solving this system gives the values of the
quantities $\left(\frac{D u}{D t}\right)_*$ and $\left(\frac{D p_{tot}}{D t}\right)_*$, and  then  gets
$\left(\frac{\partial \rho}{\partial t}\right)_*$, $\left(\frac{\partial u}{\partial t}\right)_*$, $\left(\frac{\partial p_{tot}}{\partial t}\right)_*$ and closes the calculation in \eqref{eq:jnp12} and the scheme \eqref{eq:GRPS}.

\begin{figure}[htbp]
  \centering
  \includegraphics[scale=0.2]{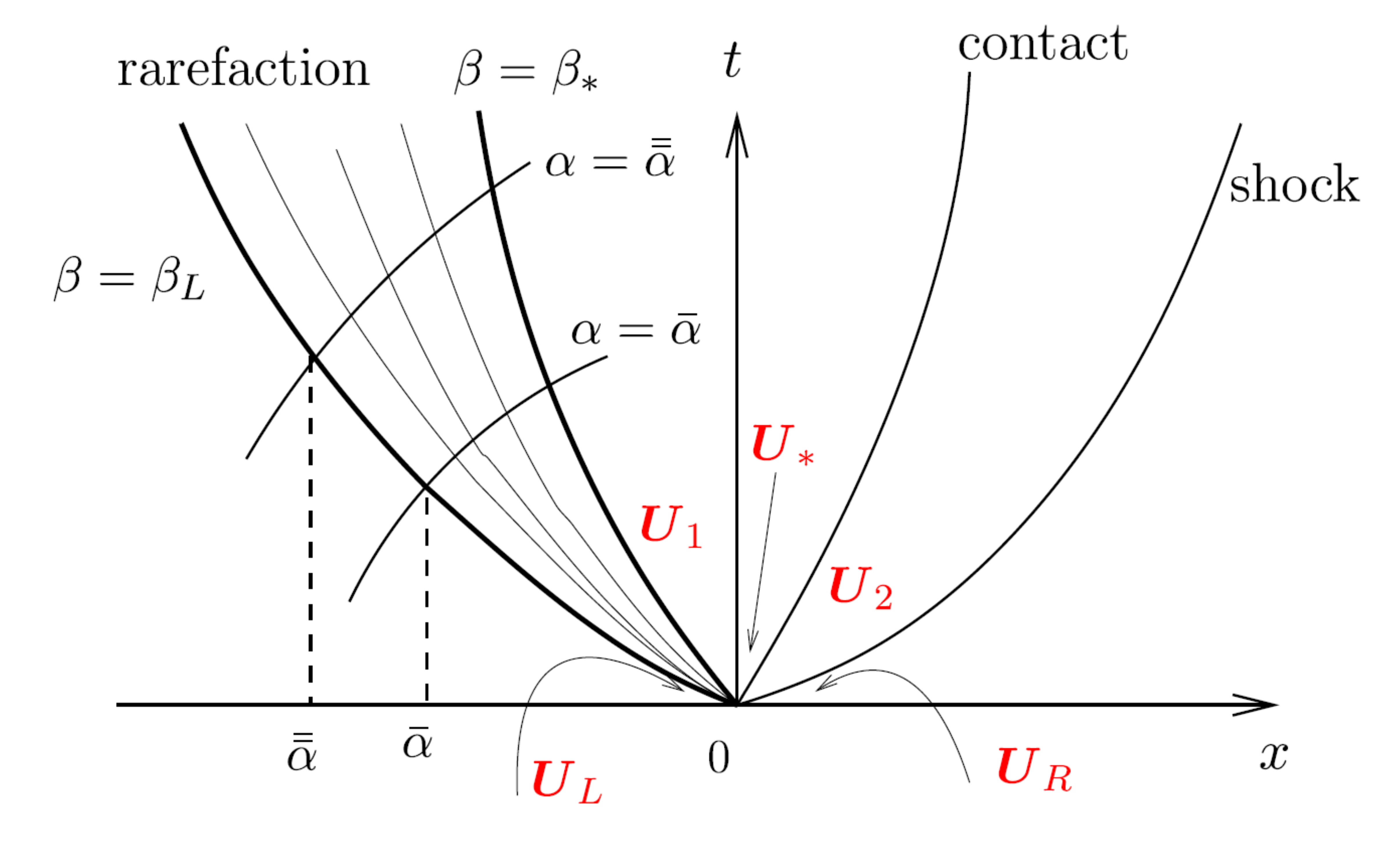}
  \caption{The schematic description of a local wave configuration for the GRP \eqref{eq:GRP1}
  with $0\leq t\ll 1$.}
  \label{Fig:3.3}
\end{figure}
\begin{figure}[htbp]
  \centering
  \includegraphics[scale=0.2]{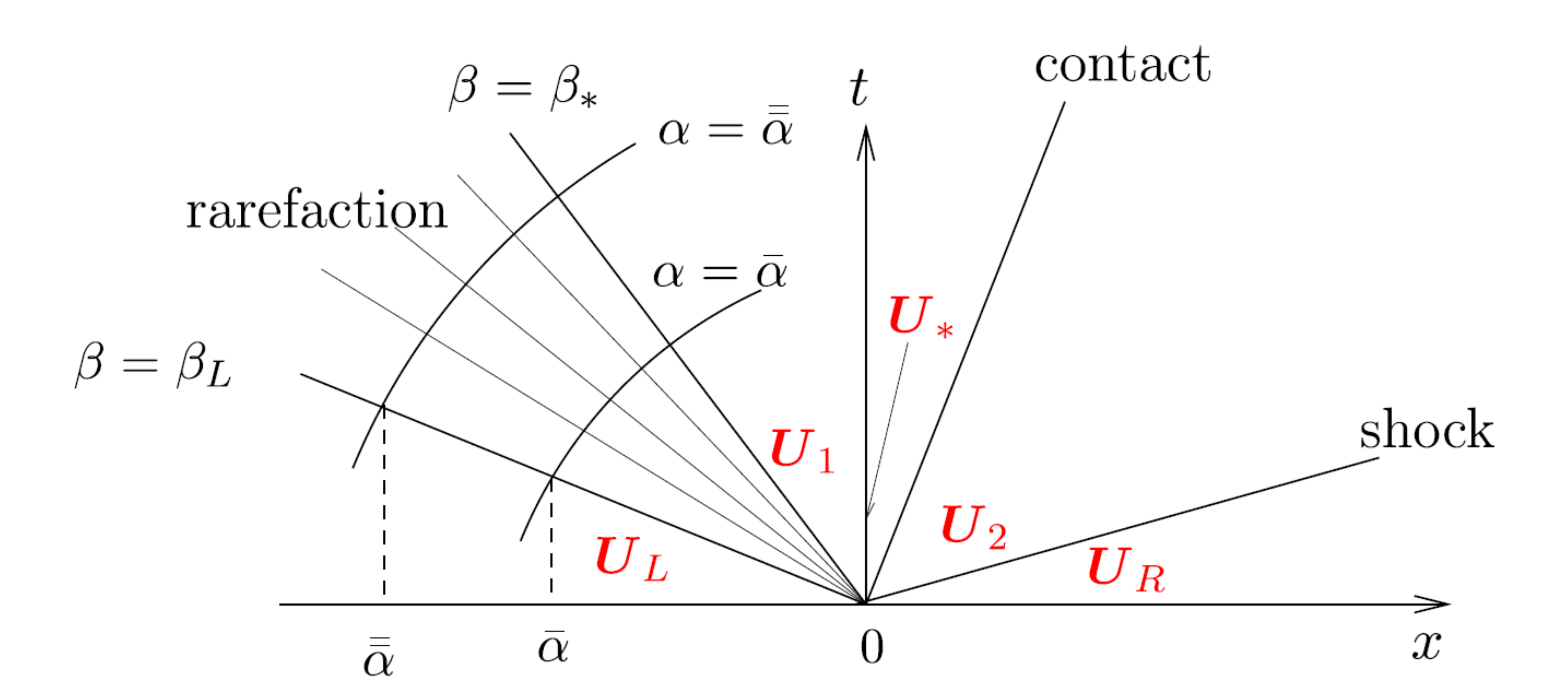}
  \caption{The schematic description of a local wave configuration for the (classical) RP \eqref{eq:RP1}, associated with Fig. \ref{Fig:3.3}.}\label{Fig:3.4}
\end{figure}

\subsubsection{Exact solution of classical Riemann problem}
\label{subsubsec:RP}
This section discusses the exact solution of the Riemann problem \eqref{eq:RP1} in order to get $\vec{U}_*$.
For the sake of convenience, $(\rho_{1*},u_{*},p_{tot,*})$ and $(\rho_{2*},u_{*},p_{tot,*})$ are used to denote the left and right states of the contact discontinuity  in Fig. \ref{Fig:3.4}.

For the rarefaction and shock waves  in Fig. \ref{Fig:3.4},  $u_{*}$
can be  represented by $p_{tot,*}$ as follows
\begin{equation}
\label{eq:RPrare}
u_L(p_{tot,*})=u_L+\int^{\rho_{L}}_{\rho_{1*}}\frac{c(\omega,S)}{\omega}d\omega, \quad S=S_{L}, \quad p_{tot,*}<p_{tot,L},
\end{equation}
and
\begin{equation}
\label{eq:RPShock}
\begin{aligned}
u_R(p_{tot,*})&=u_R+\sqrt{\frac{(p_{tot,*}-p_{tot,R})(\rho_{2*}-\rho_R)}{\rho_R \rho_{2*}}}, \\
e_{tot,2*}-e_{tot,R}&=\frac{1}{2}(p_{tot,R}+p_{tot,R})\frac{(\rho_{2*}-\rho_R)}{\rho_R \rho_{2*}}, \quad  p_{tot,*}>p_{tot,R},
\end{aligned}
\end{equation}
which are deduced from \eqref{eq:Riemanninvar} and \eqref{rh}. The quantity $\rho_{1*}$  in the first equation of \eqref{eq:RPrare} can be represented in terms of $p_{tot,*}$ through the second equation  in \eqref{eq:RPrare}.
Similarly,  $\rho_{2*}$ in the first equation of \eqref{eq:RPShock} can also be represented in terms of $p_{tot,*}$ through  the second equation in \eqref{eq:RPShock}. In practice, both the second equations
with  a fixed $p_{tot,*}$
in \eqref{eq:RPrare} and \eqref{eq:RPShock} are numerically solved by the bisection method with a proper interval.

\begin{remark}
If the left wave is the  shock wave, then \eqref{eq:RPrare} is replaced with
\[
\begin{aligned}
u_L(p_{tot,*})&=u_L-\sqrt{\frac{(p_{tot,*}-p_{tot,L})(\rho_{1*}-\rho_L)}{\rho_L \rho_{1*}}}, \\
e_{tot,1*}-e_{tot,L}&=\frac{1}{2}(p_{tot,L}+p_{tot,L})\frac{(\rho_{1*}-\rho_L)}{\rho_L \rho_{1*}}, \quad  p_{tot,*}>p_{tot,L}.
\end{aligned}
\]
If the right wave is  the  rarefaction wave, then \eqref{eq:RPShock}  is replaced with
\[
u_R(p_{tot,*})=u_R-\int^{\rho_{R}}_{\rho_{2*}}\frac{c(\omega,S)}{\omega}d\omega, \quad S=S_{R}, \quad p_{tot,*}<p_{tot,R}.
\]

\end{remark}

Combining \eqref{eq:RPrare} with \eqref{eq:RPShock} gives
a nonlinear equation with respect to $p_{tot,*}$ as follows
\begin{equation}\label{eq:RPptot}
u_L(p_{tot,*})-u_R(p_{tot,*})=0.
\end{equation}
 The following theorem tells us when \eqref{eq:RPptot} has a solution.
\begin{theorem}
If $\gamma \in { (1, 14.39]}$ and the initial data in \eqref{eq:RP1} satisfy the positive pressure condition
\begin{equation}
\label{eq:con}
\int^{\rho_{L}}_{0}\frac{c(\omega,S_{L})}{\omega}d\omega+\int^{\rho_{R}}_{0}\frac{c(\omega,S_R)}{\omega}d\omega>u_{R}-u_{L},
\end{equation}
then \eqref{eq:RPptot} has an unique solution.
\end{theorem}
\begin{proof}
Taking the partial derivative of \eqref{eq:RPrare} and \eqref{eq:RPShock} with respect to $p_{tot,*}$ gives
\[
\frac{\partial u_L(p_{tot,*})}{\partial p_{tot,*}}=
-\frac{c_{1*}}{\rho_{1*}}\frac{d \rho_{*}(p_{tot,*})}{d p_{tot,*}}=-\frac{1}{\rho_{1*}c_{1*}}<0,  \quad p_{tot,*}<p_{tot,L},
\]
and
\[
\frac{\partial u_R(p_{tot,*})}{\partial p_{tot,*}}=\frac{(\frac{\rho_{2*}-\rho_R}{\rho_R \rho_{2*}}+\frac{p_{tot,*}-p_{tot,R}}{\rho_{2*}^2}\frac{d \rho_{*}(p_{tot,*})}{d p_{tot,*}})}{2\sqrt{\frac{(p_{tot,*}-p_{tot,R})(\rho_{2*}-\rho_R)}{\rho_R \rho_{2*}}}}, \quad  p_{tot,*}>p_{tot,R}.
\]
 Thanks to Theorem \ref{thm:judge}, $\frac{d \rho_{*}(p_{tot,*})}{d p_{tot,*}}>0$ and $\rho_{2*}>\rho_{R}$ for the right shock wave, thus it holds
\[
\frac{\partial u_R(p_{tot,*})}{\partial p_{tot,*}}>0.
\]
For the right rarefaction wave and the left shock wave, the similar conclusions can be obtained.
Hence $u_L(p_{tot,*})-u_R(p_{tot,*})$ is monotone decreasing.

It is not difficult to show that
\begin{equation}\label{EQ:123456AB}
u_{L}(0)-u_{R}(0)=u_{L}-u_{R}+\int^{\rho_{L}}_{0}\frac{c(\omega,S_{L})}{\omega}d\omega+\int^{\rho_{R}}_{0}\frac{c(\omega,S_R)}{\omega}d\omega>0,
\end{equation}
and  the inequality
\begin{align*}
&u_{L}(p_{tot})-u_{R}(p_{tot})\\
=&u_{L}-u_{R}-\sqrt{(p_{tot}-p_{tot,L})\left(\rho_{L}^{-1}-\rho_{1,*}^{-1}(p_{tot})\right)}
-\sqrt{(p_{tot}-p_{tot,R})\left(\rho_{R}^{-1}-\rho_{2,*}^{-1}(p_{tot})\right)}\\
<&u_{L}-u_{R}-\sqrt{(p_{tot}-p_{tot,L})\left(\rho_{L}^{-1}-\rho_{1,*}^{-1}(p_{tot,\max})\right)}
-\sqrt{(p_{tot}-p_{tot,R})\left(\rho_{R}^{-1}-\rho_{2,*}^{-1}(p_{tot,\max})\right)},
\end{align*}
holds for $p_{tot}>p_{tot,\max}:=\max\{p_{tot,L},p_{tot,R}\}$.
The last inequality implies
\begin{equation}\label{EQ:123456}
\lim\limits_{p_{tot}\rightarrow +\infty}\left(u_{L}(p_{tot})-u_{R}(p_{tot})\right)=-\infty.
\end{equation}

Thus, if $u_L(p_{tot,\min})-u_R(p_{tot,\min})>0$ and $u_L(p_{tot,\max})-u_R(p_{tot,\max})<0$,
where $p_{tot,\min}:=\min\{p_{tot,L},p_{tot,R}\}$,
then there exists an unique solution $p_{tot,*}\in (p_{tot,\min},p_{tot,\max})$;
if $u_L(p_{tot,\max})-u_R(p_{tot,\max})>0$, then
there exists an unique solution $p_{tot,*}\in (p_{tot,\max},+\infty)$ due to \eqref{EQ:123456};
if $u_L(p_{tot,\min})-u_R(p_{tot,\min})<0$, then
there exists an unique solution $p_{tot,*}\in (0,p_{tot,\min})$ due to \eqref{EQ:123456AB}.
The proof is completed.
\qed\end{proof}

In practice,  \eqref{eq:RPptot} is solved by using the Newton's iteration,   and
then  $\rho_{1*}$, $\rho_{2*}$ and $u_{*}$ can be gotten with the help of the {generalized} Riemann invariants and the Rankine-Hugoniot conditions. Finally, $\vec{U}_{*}=\vec{U}_{1}$ if $u_{*}>0$; otherwise $\vec{U}_{*}=\vec{U}_{2}$.

For the sonic case, i.e. the $t$-axis is within the left rarefaction region. As $t\rightarrow0+$ along
$t$-axis, $u_{*}=c_{*}$. 
 The limiting value $\vec{U}_{*}$ can be simply gotten as follows. The equation
\begin{equation}\label{eq:RPsonic}
    c(\rho_{*},S_{L})-\int_{\rho_{*}}^{\rho_L}\frac{c(\omega,S_L)}{\omega}d\omega=u_L,
\end{equation}
is first solved to obtain $\rho_{*}$ by the bisection method with the interval $[\rho_{1*},\rho_{L}]$
and  then calculate $u_{*}=c(\rho_{*},S_{L})$ and $p_{tot,*}=p_{tot}(\rho_{*},S_{L})$.

Based on the above discussion, the exact solution of the Riemann problem \eqref{eq:RP1}
can be completely obtained by mimicking the procedure for the exact solution of the Riemann problem of
the Euler equations in gas dynamics in \cite[Chapter 4]{Toro}.

\subsubsection{Resolution of rarefaction waves}
\label{subsubsec:RARE}
This section resolves the left rarefaction wave  in the solution of local GRP \eqref{eq:GRP1} as shown in Fig. \ref{Fig:3.3} and gives the first equation in \eqref{eq:DutDpt}. Before that,  some similar  notations
are introduced   to those in \cite{Ben-Li-Warnecke2006,Wu-Tang2016,Wu-Yang-Tang2014A,Wu-Yang-Tang2014B,Yang-He-Tang2011,Yang-Tang2012}.

The region of the left rarefaction wave for \eqref{eq:GRP1} can be described by the set
  $\mathcal{R}=\{(\alpha (x,t),\beta (x,t))|\beta \in [\beta_L,\beta_*],-\infty <\alpha <0\}$, where $\beta_L=\lambda_1(\vec{U}_L)$, $\beta_*=\lambda_1(\vec{U}_*)$, and $\beta =\beta (x,t)$ and $\alpha = \alpha (x,t)$ are the integral curves of the following equations
\begin{equation}
\frac{d x}{d t}=\lambda_1,\quad \frac{d x}{d t}=\lambda_3,
\end{equation}
respectively.
Here $\beta$ and $\alpha$ have been denoted as follows: $\beta$ is the initial value of the slope  $\lambda_1$ at the singularity point $(0,0)$,
and $\alpha$ for the transversal characteristic curves $\frac{d x}{d t}=\lambda_3$
is the $x$-coordinate of the intersection point with the leading
$\beta$-curve, see Fig. \ref{Fig:3.3}.
Similarly, a local coordinate transformation is also introduced within the region of the left rarefaction
wave in Fig. \ref{Fig:3.4}. To avoid confusion, it  is denoted as $x_{\rm ass} = x({\alpha,\beta})$
 and $t_{\rm ass} = t({\alpha,\beta})$.

 The major result of this section is given in the following lemma.
 \begin{lemma}\label{lem:rare}
The limiting values of $ \frac{D u}{D t} $ and $ \frac{D p_{tot}}{D t} $ satisfy
\begin{equation}\label{eq:rare}
a_L\left(\frac{D u}{D t}\right)_*+b_L\left(\frac{D p_{tot}}{D t}\right)_*=d_L,
\end{equation}
with
\begin{equation}\label{eq:aLbL}
a_L=1,\quad b_L=\frac{1}{\rho_{1*}c_{1*}},
\end{equation}
\begin{equation}
d_L(\beta_{*})=\frac{1}{2}A_L(0,\beta_{*})+\frac{1}{2}\left(\frac{\partial t}{\partial \alpha}\right)^{-1}B_L(0,\beta_{*}),
\label{eq:dL}
\end{equation}
where the expressions of $A_L(0,\beta)$, $\frac{\partial t}{\partial \alpha}(0,\beta)$, and  $B_L(0,\beta)$ are given as follows
\begin{equation}\label{eq:AL}
A_{L}(0,\beta)=K(\rho(0,\beta),S_{L})c_{L}S'_{L}\frac{\rho(0,\beta)c(0,\beta)}{\rho_{L}c_{L}},
\end{equation}
\begin{equation}\label{eq:TL}
\frac{\partial t}{\partial \alpha}(0,\beta)=\frac{1}{\beta_{L}}\left(\frac{\rho(0,\beta)c(0,\beta)}{\rho_{L}c_{L}}\right)^{-\frac{1}{2}},
\end{equation}
and
\begin{equation}\label{eq:BL}
\begin{aligned}
B_{L}(0,\beta)=&\frac{1}{\beta_{L}}\left(A_{L}(0,\beta_L)-2c_L\psi_{1,L}'\right)\\
&-\int_{c_L}^{c(0,\beta)}\frac{1}{2\omega}\frac{\partial t}{\partial \alpha}(0,\beta(\omega))A_L(0,\beta(\omega))\left(1+\frac{\omega}{\rho(0,\beta(\omega))}\left(\frac{\partial c(\rho,S)}{\partial \rho}\right)^{-1}(0,\beta(\omega))\right)d\omega,
\end{aligned}
\end{equation}
here $S'_{L}=\left(\frac{\partial S}{\partial x}\right)_{L}$,
$\psi_{1,L}'=-\frac{1}{\rho_{L}c_{L}}\left(\frac{\partial p_{tot}}{\partial x}\right)_{L}+\left(\frac{\partial u}{\partial x}\right)_{L}
+K(\rho_{L},S_{L})S'_L$, and $\beta=\lambda_{1}$.
\end{lemma}
\begin{proof}
Following the proof of Lemma 3.1 in \cite{Ben-Li-Warnecke2006}, one can obtain \eqref{eq:aLbL}, \eqref{eq:dL},
and
\[
A_{L}=A_{L}(0,\beta)=K(\rho(0,\beta),S_{L})c_{L}S'_{L}\exp\left(-\int_{\beta_L}^{\beta}\frac{1}{c(0,\omega)}d\omega\right),
\]
\[
\frac{\partial t}{\partial \alpha}(0,\beta)=\frac{1}{\beta_{L}}\exp\left(\int_{\beta_L}^{\beta}\frac{1}{2c(0,\omega)}d\omega\right),
\]
\[
B_{L}(0,\beta)=\frac{1}{\beta_{L}}\left(A_{L}(0,\beta_L)-2c_L\psi_{1,L}'\right)
+\int_{\beta_L}^{\beta}\frac{1}{2c(0,\omega)}\frac{\partial t}{\partial \alpha}(0,\omega)A_L(0,\omega)d\omega.
\]
Using
\[
d\beta=du-dc=-\frac{c}{\rho}d\rho-dc=-\left(\frac{c}{\rho}\left(\frac{\partial c(\rho,S)}{\partial \rho}\right)^{-1}+1\right)dc,
\]
gives \eqref{eq:BL} and
\begin{align*}
\int_{\beta_L}^{\beta}\frac{1}{c(0,\omega)}d\omega=&-\int_{c_L}^{c(0,\beta)}\frac{1}{\omega}d\omega
-\int_{c_L}^{c(0,\beta)}\frac{1}{\rho(0,\beta(\omega))}\frac{\partial \rho(c,S)}{\partial c}(0,\beta(\omega))d\omega\\
=&-\ln\left(\frac{c(0,\beta)}{c_{L}}\right)-\ln\left(\frac{\rho(0,\beta)}{\rho_{L}}\right)
=-\ln\left(\frac{\rho(0,\beta)c(0,\beta)}{\rho_{L}c_{L}}\right),
\end{align*}
and then derives  \eqref{eq:AL} and \eqref{eq:TL}.
The proof is completed.
\qed\end{proof}
\begin{remark}
For the right rarefaction wave, corresponding coefficients $a_{R}$, $b_{R}$, and $d_{R}$ in \eqref{eq:DutDpt} can be similarly
given by
\begin{equation}\label{eq:aRbR}
a_R=1,\quad b_R=-\frac{1}{\rho_{2*}c_{2*}},
\end{equation}
\begin{equation}
d_R(\beta_{*})=\frac{1}{2}A_R(0,\beta_{*})+\frac{1}{2}\left(\frac{\partial t}{\partial \alpha}\right)^{-1}B_R(0,\beta_{*}),
\label{eq:dR}
\end{equation}
where
\begin{equation}\label{eq:AR}
A_{R}(0,\beta)=K(\rho(0,\beta),S_{R})c_{R}S'_{R}\frac{\rho(0,\beta)c(0,\beta)}{\rho_{R}c_{R}},
\end{equation}
\begin{equation}\label{eq:TR}
\frac{\partial t}{\partial \alpha}(0,\beta)=\frac{1}{\beta_{R}}\left(\frac{\rho(0,\beta)c(0,\beta)}{\rho_{R}c_{R}}\right)^{-\frac{1}{2}},
\end{equation}
and
\begin{equation}\label{eq:BR}
\begin{aligned}
B_{R}(0,\beta)=&\frac{1}{\beta_{R}}\left(A_{R}(0,\beta_R)+2c_R\psi_{3,R}'\right)\\
&-\int_{c_R}^{c(0,\beta)}\frac{1}{2\omega}\frac{\partial t}{\partial \alpha}(0,\beta(\omega))A_R(0,\beta(\omega))\left(1+\frac{\omega}{\rho(0,\beta(\omega))}\left(\frac{\partial c(\rho,S)}{\partial \rho}\right)^{-1}(0,\beta(\omega))\right)d\omega,
\end{aligned}
\end{equation}
here $S'_{R}=\left(\frac{\partial S}{\partial x}\right)_{R}$,
$\psi_{3,R}'=-\frac{1}{\rho_{R}c_{R}}\left(\frac{\partial p_{tot}}{\partial x}\right)_{R}+\left(\frac{\partial u}{\partial x}\right)_{R}
-K(\rho_{R},S_{R})S'_R$, and $\beta=\lambda_{3}$.
\end{remark}
\begin{remark}
 The integral in \eqref{eq:BL} or \eqref{eq:BR} can be exactly evaluated. In our computations,
 the three point Gauss-Lobatto quadrature is used only with an additional calculation the physical states at the internal point $\frac{1}{2}(c_{L}+c_{1,*})$ of the integral interval.
%
%
Moreover, if $\hat{a}_{R}=0$, the above corresponding coefficients $a_{L}$, $b_{L}$, and $d_{L}$ (or $a_{R}$, $b_{R}$, and $d_{R}$)
are the same as those in \cite{Ben-Li-Warnecke2006}.
\end{remark}

\subsubsection{Resolution of shock waves}
\label{subsubsec:shock}
This section resolves the right shock wave in the solution for the local GRP \eqref{eq:GRP1} as shown in Fig. \ref{Fig:3.3} and
gives the second equation in \eqref{eq:DutDpt} through differentiating the shock relation along the shock trajectory.

 Let $x=x(t)$ be the shock trajectory associated with the third characteristic field and assume that it propagates with the speed $\sigma:=x'(t)>0$ to the right, see Fig. \ref{Fig:3.3}. Denote the left and right states of the shock wave by $\vec{U}(t)$ and $\vec{\bar{U}}(t)$, respectively, i.e. $\vec{U}(t)=\vec{U}(x(t)-0,t)$ and $\vec{\bar{U}}(t)=\vec{U}(x(t)+0,t)$. Across the shock wave, using the Rankine-Hugoniot jump conditions
gives
\begin{equation}
u=\bar{u}+\Phi(p_{tot},\bar{p}_{tot},\bar{\rho}),
\end{equation}
where
\begin{equation}\label{eq:PHI}
\Phi(p_{tot},\bar{p}_{tot},\bar{\rho},\rho(p_{tot},\bar{p}_{tot},\bar{\rho}))=\sqrt{(p_{tot}-\bar{p}_{tot})\left(\frac{1}{\bar{\rho}}-\frac{1}{\rho(p_{tot},\bar{p}_{tot},\bar{\rho})}\right)},
\end{equation}
and the density $\rho$ is implicitly defined by
the  equation
 \begin{equation}
e_{tot}(\bar{\rho},\bar{p}_{tot})-e_{tot}(\rho, p_{tot})=\frac{1}{2}(\bar{p}_{tot}+p_{tot})\left(\frac{\bar{\rho}-\rho}{\bar{\rho}\rho}\right),
\end{equation}
that is to say, $\rho$ can be regarded as a function of $p_{tot},\bar{p}_{tot},\bar{\rho}$, denoted by $\rho=H(p_{tot},\bar{p}_{tot},\bar{\rho})$.

Along the shock trajectory, one  has
\begin{equation}\label{411}
\frac{D_\sigma u}{D t}=\frac{D_\sigma \bar{u}}{D t}+\frac{D_\sigma \Phi}{D t},
\end{equation}
where $D_\sigma:=\frac{\partial}{\partial t}+\sigma\frac{\partial }{\partial x}$
denotes the directional derivative operator along the shock trajectory.

The major result in this section is given as follows.
\begin{lemma}\label{lem:shock}
The limiting values of $\frac{D u}{D t}$ and $\frac{D p_{tot}}{D t}$ satisfy
\begin{equation}\label{eq:shock}
a_R\left(\frac{D u}{D t}\right)_*+b_R\left(\frac{D p_{tot}}{D t}\right)_*=d_R,
\end{equation}
with
\begin{align}
&\begin{aligned}
a_R=&1+\rho_{2,*}(\sigma_{R}-u_*)\Phi_{1}(p_{tot,*},p_{tot,R},\rho_{R},\rho_{2,*}),\\ b_R=&-\frac{\sigma_{R}-u_*}{\rho_{2,*}c_{2,*}}-\Phi_{1}(p_{tot,*},p_{tot,R},\rho_{R},\rho_{2,*}),
\end{aligned}\label{eq:sarbr} \\
&d_R=L_{p_{tot}}^R (p_{tot}')_R+L_{u}^Ru'_R+L_{\rho}^R\rho'_R,\label{eq:sdr}
\end{align}
where
\begin{equation}\label{Lpurhor}
\begin{aligned}
&L_{p_{tot}}^R=-\frac{1}{\rho_R}+(\sigma_{R}-u_R)\Phi_{2}(p_{tot,*},p_{tot,R},\rho_{R},\rho_{2,*}), \\
&L_{u}^R=\sigma_{R}-u_R-\rho_R c^2_R\Phi_{2}(p_{tot,*},p_{tot,R},\rho_{R},\rho_{2,*})-\rho_R\Phi_{3}(p_{tot,*},p_{tot,R},\rho_{R},\rho_{2,*}),\\
&L_{\rho}^R=(\sigma_{R}-u_R)\Phi_{3}(p_{tot,*},p_{tot,R},\rho_{R},\rho_{2,*}),
\end{aligned}
\end{equation}
and $\sigma_{R}=\frac{\rho_{2*}u_*-\rho_{R}u_R}{\rho_{2*}-\rho_{R}}$.
Here $\Phi_{i}$, $i=1,2,3$, are defined by
\begin{equation}\label{eq:PHI123}
\begin{aligned}
&\Phi_1(p_{tot},\bar{p}_{tot},\bar{\rho},\rho(p_{tot},\bar{p}_{tot},\bar{\rho}))=\frac{\partial \Phi}{\partial p_{tot}}+\frac{\partial \Phi}{\partial \rho}H_{1},\\
&\Phi_2(p_{tot},\bar{p}_{tot},\bar{\rho},\rho(p_{tot},\bar{p}_{tot},\bar{\rho}))=\frac{\partial \Phi}{\partial \bar{p}_{tot}}+\frac{\partial \Phi}{\partial \rho}H_{2},\\
&\Phi_3(p_{tot},\bar{p}_{tot},\bar{\rho},\rho(p_{tot},\bar{p}_{tot},\bar{\rho}))=\frac{\partial \Phi}{\partial \bar{\rho}}+\frac{\partial \Phi}{\partial \rho}H_{3},
\end{aligned}
\end{equation}
and
\begin{equation}\label{eq:phip123}
\begin{aligned}
&\frac{\partial \Phi}{\partial p_{tot}}(p_{tot},\bar{p}_{tot},\bar{\rho},\rho)=\frac{1}{2\Phi}\left(\frac{1}{\bar{\rho}}-\frac{1}{\rho}\right),\quad
\frac{\partial \Phi}{\partial \bar{p}_{tot}}(p_{tot},\bar{p}_{tot},\bar{\rho},\rho)=\frac{1}{2\Phi}\left(\frac{1}{\rho}-\frac{1}{\bar{\rho}}\right),\\
&\frac{\partial \Phi}{\partial \bar{\rho}}(p_{tot},\bar{p}_{tot},\bar{\rho},\rho)=\frac{1}{2\Phi\bar{\rho}^2}(\bar{p}_{tot}-p_{tot}),\quad
\frac{\partial \Phi}{\partial \rho}(p_{tot},\bar{p}_{tot},\bar{\rho},\rho)=\frac{1}{2\Phi\rho^2}(p_{tot}-\bar{p}_{tot}),
\end{aligned}
\end{equation}
\begin{equation}\label{eq:H123}
\begin{aligned}
&H_{1}(p_{tot},\bar{p}_{tot},\bar{\rho})=\frac{\partial H}{\partial p_{tot}}(p_{tot},\bar{p}_{tot},\bar{\rho})=-\frac{ \frac{\partial e_{tot}}{\partial p_{tot}}(\rho,p_{tot}) - \frac{\rho-\bar{\rho}}{2\rho\bar{\rho}}   }
{\frac{\partial e_{tot}}{\partial \rho}(\rho,p_{tot}) - \frac{p_{tot}+\bar{p}_{tot}}{2\rho^2}  },\\
&H_{2}(p_{tot},\bar{p}_{tot},\bar{\rho})=\frac{\partial H}{\partial \bar{p}_{tot}}(p_{tot},\bar{p}_{tot},\bar{\rho})=\frac{ \frac{\partial e_{tot}}{\partial p_{tot}}(\bar{\rho},\bar{p}_{tot}) + \frac{\rho-\bar{\rho}}{2\rho\bar{\rho}}   }
{\frac{\partial e_{tot}}{\partial \rho}(\rho,p_{tot}) - \frac{p_{tot}+\bar{p}_{tot}}{2\rho^2}  },\\
&H_{3}(p_{tot},\bar{p}_{tot},\bar{\rho})=\frac{\partial H}{\partial {\bar{\rho}}}(p_{tot},\bar{p}_{tot},\bar{\rho})=\frac{ \frac{\partial e_{tot}}{\partial \rho}(\bar{\rho},\bar{p}_{tot}) - \frac{p_{tot}+\bar{p}_{tot}}{2\bar{\rho}^2}    }
{\frac{\partial e_{tot}}{\partial \rho}(\rho,p_{tot}) - \frac{p_{tot}+\bar{p}_{tot}}{2\rho^2}  },
\end{aligned}
\end{equation}
where
\[
\frac{\partial e_{tot}}{\partial \rho}(\rho,p_{tot})=\frac{ \rho^2e^2(3-\gamma)-15(\gamma-1)\rho e\hat{a}_{R}T^4-4\hat{a}_{R}^2T^8}{3\rho^2\hat{K}},\quad
\frac{\partial e_{tot}}{\partial p_{tot}}(\rho,p_{tot})=\frac{e\tilde{K}}{\rho\hat{K}}.
\]
\end{lemma}
\begin{proof}
It is similar to the proof of  Lemma 4.1 of \cite{Ben-Li-Warnecke2006}.
\qed\end{proof}
\begin{remark}
For the left shock wave, corresponding coefficients  $a_{L}$, $b_{L}$, $d_{L}$
 in \eqref{eq:DutDpt} can be similarly derived as follows
\begin{align}
a_L=&1-\rho_{1,*}(\sigma_{L}-u_*)\Phi_{1}(p_{tot,*},p_{tot,R},\rho_{R},\rho_{1,*}),\ b_L=-\frac{\sigma_{L}-u_*}{\rho_{1,*}c_{1,*}}+\Phi_{1}(p_{tot,*},p_{tot,L},\rho_{L},\rho_{1,*}),\label{eq:salbl} \\
d_L=&L_{p_{tot}}^L(p_{tot}')_L+L_{u}^L
u'_L+L_{\rho}^L\rho'_L,\label{eq:sdl}
\end{align}
with
\begin{equation}\label{Lpurhol}
\begin{aligned}
&L_{p_{tot}}^L=-\frac{1}{\rho_L}-(\sigma_{L}-u_L)\Phi_{2}(p_{tot,*},p_{tot,L},\rho_{L},\rho_{1,*}), \\
&L_{u}^L=\sigma_{L}-u_L+\rho_L c^2_L\Phi_{2}(p_{tot,*},p_{tot,L},\rho_{L},\rho_{1,*})
+\rho_L\Phi_{3}(p_{tot,*},p_{tot,L},\rho_{L},\rho_{1,*}),\\
&L_{\rho}^L=-(\sigma_{L}-u_L)\Phi_{3}(p_{tot,*},p_{tot,L},\rho_{L},\rho_{1,*}),
\end{aligned}
\end{equation}
where $\sigma_{L}=\frac{\rho_{1*}u_*-\rho_{L}u_L}{\rho_{1*}-\rho_{L}}$.
\end{remark}

\subsubsection{Time derivatives of solutions at singularity point}\label{subsubsec:Time}

This section solves the $2\times2$ linear system \eqref{eq:DutDpt} formed by \eqref{eq:rare}
 and \eqref{eq:shock} 
to give the total derivatives $\left(\frac{D u}{D t}\right)_*$ and $\left(\frac{D p_{tot}}{D t}\right)_*$,
and further derives the time derivatives $\left(\frac{\partial u}{\partial t}\right)_*$, $\left(\frac{\partial p_{tot}}{\partial t}\right)_*$ and $\left(\frac{\partial \rho}{\partial t}\right)_*$ at the singularity point $(x,t)=(0,0)$.

{\bf Case 1:} Nonsonic case.

It is necessary to check whether such linear system has an unique solution.
\begin{theorem}\label{thm:22}
The linear system \eqref{eq:rare} and \eqref{eq:shock} has an unique solution.
\end{theorem}
\begin{proof}
In order to prove the conclusion, it needs to prove
\[
a_Lb_R-a_Rb_L \neq 0.
\]
For the left (rarefaction) wave,   \eqref{eq:aLbL} implies
\[
a_L>0,\quad b_L>0.
\]
For the right (shock) wave,
because $\rho_{2*}>\rho_R$, and $p_{tot,*}>p_{tot,R}$, one has
\[
\Phi_{1}(p_{tot,*},p_{tot,R},\rho_{R},\rho_{2,*})=
\frac{\hat{g}(\rho_{2*},p_{tot,*},\rho_{R},p_{tot,R},r_{e,2*})}{\check{g}(\rho_{2*},p_{tot,*},\rho_{R},p_{tot,R},r_{e,2*})}>0,
\]
where
\begin{align*}
\hat{g}(\rho_{2*},p_{tot,*},\rho_{R},p_{tot,R},r_{e,2*})=& 3(4\rho_{2*}r_{e,2*}+\rho_{R})(p_{tot,*}-p_{tot,R})\\
& +\big(3\gamma_{1}(e_*\rho_* r_{e,2*}+e_{2*}\rho_{2*}+p_{tot,R})+16p_{tot,R}r_{e,2*}\big)(\rho_{2*}-\rho_{R})>0,
\end{align*}
\begin{align*}
\check{g}(\rho_{2*},p_{tot,*},\rho_{R},p_{tot,R},r_{e,2*})=&\big( (28r_{e,2*}+3\gamma_{1})p_{tot,*}+6\gamma_{1}e_{2*}\rho_{2*}(1+r_{e,2*})\\
&+(4r_{e,2*}+3\gamma_{1})p_{tot,R}\big)
\sqrt{ (\rho_{2*}-\rho_{R})(p_{tot,*}-p_{tot,R})\rho_{2*}\rho_{R}}>0.
\end{align*}
Using $\sigma_{R}>u_*$ and \eqref{eq:sarbr} gives
\[
a_R>0,\quad b_R<0.
\]
In summary, one has
\begin{equation}
a_Lb_R-a_Rb_L<0,
\end{equation}
so that the linear system \eqref{eq:rare} and \eqref{eq:shock} has an unique solution.
\qed\end{proof}

The time derivatives of solutions at singularity point can be calculated as follows.
\begin{theorem}\label{thm:Time}
Assuming that the $t$ axis is not included in the rarefaction wave, then the limiting values of time derivatives
$\left(\frac{\partial \vec{W}}{\partial t}\right)_*$ can be
 calculated by
 \begin{equation}\label{eq:timeup}
\begin{split}
\left(\frac{\partial u}{\partial t}\right)_*&=\left(\frac{D u}{D t}\right)_*+\frac{u_*}{\rho_{*} c_{*}^2}\left(\frac{D p_{tot}}{D t}\right)_*,\\
\left(\frac{\partial p_{tot}}{\partial t}\right)_*&=\left(\frac{D p_{tot}}{D t}\right)_*+\rho_{*}u_*\left(\frac{D u}{D t}\right)_*,
\end{split}
\end{equation}
and
\begin{equation}\label{eq:timerhol}
\left(\frac{\partial \rho}{\partial t}\right)_*=\frac{1}{c_{*}^2}\left[\left(\frac{\partial p_{tot}}{\partial t}\right)_*+\frac{\partial p_{tot}}{\partial S}(\rho_*,S_*)\frac{u_*}{c_{*}K(\rho_*,S_*)}A_{L}(0,\beta_*)\right],
\quad  u_*>0,
\end{equation}
or
\begin{equation}\label{eq:timerhot}
\left(\frac{\partial \rho}{\partial t}\right)_*=\frac{f_Ru_{*}-g_u^R\left(\frac{D u}{D t}\right)_*-g_{p_{tot}}^R\left(\frac{D p_{tot}}{D t}\right)_*}{g^R_{\rho}},
\quad  u_*<0,
\end{equation}
 where   $A_{L}(0,\beta_*)$ and $K(\rho_*,S_*)$  are given in \eqref{eq:AL} and \eqref{eq:K}, respectively,
 and $g_u^R, g_{p_{tot}}^R$, and $g^R_{\rho}$ are constant,
 depending on $\vec{U}_R$ and $\vec{U}_*$, and
 their expressions are given by
\begin{align*}
g_u^R=&\rho_*u_*(\sigma_{R}-u_*)H_1,\quad g_{p_{tot}}^R=\frac{\sigma}{c_{*}^2}-u_*H_1,\quad g^R_{\rho}=u_*-\sigma_R,
\\
f_R=&(\sigma_{R}-u_R)H_2\left(\frac{\partial p_{tot}}{\partial x}\right)_R+(\sigma_{R}-u_R)H_3\left(\frac{\partial \rho}{\partial x}\right)_R-\rho_R(H_2c_R^2+H_3)\left(\frac{\partial u}{\partial x}\right)_R.
\end{align*}
Here $H_{i}$ is  defined in \eqref{eq:H123}, $i=1,2,3$.
\end{theorem}

{\bf Case 2:}: Sonic case.

When the $t$ axis ($x = 0$) is located in the rarefaction fan, the sonic case happens
and Theorems \ref{thm:22} and \ref{thm:Time}
are not available. However, since one of the characteristic curves is tangential to the $t$-axis, the sonic case becomes much simpler. 
\begin{theorem}\label{thm:Time2}
If the $t$ axis is within the left rarefaction wave, then one has
\begin{equation}
\left(\frac{\partial u}{\partial t}\right)_*=d_{L}(0),\quad \left(\frac{\partial p_{tot}}{\partial t}\right)_*=\rho_* u_* d_{L}(0),
\end{equation}
\[
\left(\frac{\partial \rho}{\partial t}\right)_*=\frac{1}{c_{*}^2}\left[\left(\frac{\partial p_{tot}}{\partial t}\right)_*+\frac{\partial p_{tot}}{\partial S}(\rho_*,S_*)\frac{u_*}{c_{*}K(\rho_*,S_*)}A_{L}(0,0)\right],
\]
where   $A_{L}(0,\beta)$, $d_L(\beta)$ and $K(\rho_*,S_*)$  are given in \eqref{eq:AL}, \eqref{eq:dL} and \eqref{eq:K}, respectively.
\end{theorem}

\begin{remark}
The proofs of Theorems \ref{thm:Time} and \ref{thm:Time2} are similar to those of several theorems in
\cite[Section 5]{Ben-Li-Warnecke2006}.
\end{remark}

\subsubsection{Acoustic case}
\label{subsubsec:aco}
This section turns to the acoustic case of the GRP \eqref{eq:GRP1}, i.e. $\vec{U}_L=\vec{U}_R$ and $\vec{U}'_L \neq \vec{U}'_R$.
In this case, $\vec{U}_L=\vec{U}_R=\vec{U}_*$ and only linear waves emanate from the origin (0, 0) so that
the resolution of the GRP becomes simpler than the general case discussed before.
For the sake of simplicity, in the following we only reserve the subscripts related to the state $\vec U_L$ or $\vec U_*$
or $\vec U_R$ for the derivatives of all physical variables.
\begin{theorem}\label{thm:aco}
In the acoustic case of that  $\vec{U}_L=\vec{U}_R$ and  $\vec{U}'_L \neq \vec{U}'_R$, if $u_*-c_*<0$ and $ u_*+c_*>0$, then $\left(\frac{\partial u}{\partial t}\right)_*, \left(\frac{\partial p_{tot}}{\partial t}\right)_*$ and $\left(\frac{\partial \rho}{\partial t}\right)_*$ can be obtained by
\begin{equation}\label{eq:acoup}
\begin{split}
\left(\frac{\partial u}{\partial t}\right)_*=-\frac{1}{2}\left[(u_*+c_*)\left(u'_L+\frac{p_{tot,L}'}{\rho_* c_*}\right)+(u_*-c_*)\left(u'_R-\frac{p_{tot,R}'}{\rho_* c_*}\right)\right],\\
\left(\frac{\partial p_{tot}}{\partial t}\right)_*=-\frac{\rho_* u_*}{2}\left[(u_*+c_*)\left(u'_L+\frac{p_{tot,L}'}{\rho_* c_*}\right)-(u_*-c_*)\left(u'_R-\frac{p_{tot, R}'}{\rho_* c_*}\right)\right],
\end{split}
\end{equation}
and
\begin{equation}\label{eq:acorho}
\left(\frac{\partial \rho}{\partial t}\right)_*=\left\lbrace\begin{array}{lc}
\frac{1}{c_*^2}\left[\left(\frac{\partial p_{tot}}{\partial t}\right)_*+u_*\frac{\partial p_{tot}}{\partial S}(\rho_* ,S_*)S'_L\right], &u_*>0, \\
\frac{1}{c_*^2}\left[\left(\frac{\partial p_{tot}}{\partial t}\right)_*+u_*\frac{\partial p_{tot}}{\partial S}(\rho_* ,S_*)S'_R\right], &u_*<0. \\
\end{array}
\right.
\end{equation}
\end{theorem}
\begin{proof}
Mimicking the proof of Theorem 6.1 of \cite{Ben-Li-Warnecke2006} can complete the proof.
\qed\end{proof}

\subsection{Two-dimensional extension}
\label{subsec:2d}
This section extends the previous GRP scheme to
 the two-dimensional RHEs
 \begin{equation}\label{RHE2d}
\frac{\partial \vec{U}}{\partial t}+\frac{\partial \vec{F_1(\vec{U})}}{\partial x}+\frac{\partial \vec{F_2(\vec{U})}}{\partial y}=0,
\end{equation}
with
\[\begin{split}
 \vec{U}&=(\rho,\rho u, \rho v, E)^T, \\
  \vec{F_1(\vec{U}})&=(\rho u, \rho u^2+p_{tot},\rho u v, u(E+p_{tot}))^T, \\
   \vec{F_2(\vec{U}})&=(\rho v,\rho u v, \rho v^2+p_{tot},v(E+p_{tot}))^T,
\end{split}
\]
where the physical meanings of $\rho, e_{tot}$, and $p_{tot}$
are the same as those in the one-dimensional case, $(u,v)$ denotes the velocity vector, and the total energy
$E$ is given by $E=\frac{1}{2}\rho(u^2+v^2)+\rho e_{tot}$.

The extension will be with the help of  the Strang splitting technique, see
\cite{Strang1963,Strang1964} and \cite[Chapter 7]{Ben-Falcovitz3003}. First,
the  \eqref{RHE2d} is split into two subsystems such as
\begin{equation}\label{eq:split}
\begin{array}{cc}
\left\lbrace \begin{aligned}
&\frac{\partial \rho}{\partial t}+\frac{\partial (\rho u)}{\partial x}=0,\\
& \frac{\partial (\rho u)}{\partial t}+\frac{\partial (\rho u^2+p_{tot})}{\partial x}=0,\\
& \frac{\partial (\rho v)}{\partial t}+\frac{\partial (\rho u v)}{\partial x}=0,\\
& \frac{\partial  E}{\partial t}+\frac{\partial u( E+p_{tot})}{\partial x}=0,
\end{aligned} \right.
& \ \ \
\left\lbrace \begin{aligned}
&\frac{\partial \rho}{\partial t}+\frac{\partial (\rho v)}{\partial y}=0,\\
& \frac{\partial (\rho u)}{\partial t}+\frac{\partial (\rho  u v)}{\partial y}=0,\\
& \frac{\partial (\rho v)}{\partial t}+\frac{\partial (\rho v^2+p_{tot})}{\partial y}=0,\\
& \frac{\partial  E}{\partial t}+\frac{\partial v(E+p_{tot})}{\partial y}=0.
\end{aligned} \right.
\end{array}
\end{equation}
Then the second-order accurate Strang splitting method can be given by
\begin{equation}
\vec{U}^{n+1}=\mathcal{L}_{x}\left(\frac{\Delta t}{2}\right)\mathcal{L}_{y}(\Delta t)\mathcal{L}_{x}\left(\frac{\Delta t}{2}\right)\vec{U}^{n},
\end{equation}
where $\mathcal{L}_{x}(\Delta t)$ and $\mathcal{L}_{y}(\Delta t)$ denote the one-dimensional evolution
operators for one time step $\Delta t$ of the above subsystems in \eqref{eq:split}, respectively.

The rest of the tasks only replace $\mathcal{L}_{x}(\Delta t)$ and $\mathcal{L}_{y}(\Delta t)$
with corresponding one-dimensional GRP evolution operators.
Without loss of generality,
take  the $x$-split system in \eqref{eq:split}  as an example.  In our GRP scheme,
the term $\mathcal{L}_x(\Delta t)\vec U^n$ is almost the right-hand side of
the scheme \eqref{eq:GRPS}.
In fact,  in the  $x$-split system, the third equation can be simplified as follows
\[
\frac{\partial v}{\partial t}+u\frac{\partial v}{\partial x}=0,
\]
and  decoupled from the $x$-split system,  and the other three equations
are the same as those in the one-dimensional RHEs \eqref{eq:RHE}.
Thus, both the 1- and 4-waves of $x$-split system are nonlinear,  and the 2- and 3-waves are
 contact discontinuity and shear wave with the speed of $u$, respectively.
The previous resolution of one-dimensional GRP can be used to derive the values of $\rho_*$, $u_*$, and $e_*$
 as well as the limiting values of their first order derivatives at $x=0$, as $t\rightarrow 0^+$,
 while the value of $v_*$ is calculated as follows
\[
v_*=\left\lbrace\begin{array}{cc}
                  v_L, & u_*>0, \\
                  v_R, & u_*<0,
                \end{array}\right.
\]
and the calculation of $\left(\frac{\partial v}{\partial t}\right)_*$ can be similarly done in the way used in \cite{Ben-Li-Warnecke2006}. For the case  that the 1- and 4-waves are
the left-moving rarefaction wave and right-moving shock wave, respectively,
and the $t$-axis is located in the intermediate region, see Fig. \ref{Fig:3.3},
 it  is done as follows.

\begin{theorem} For the $x$-split system of \eqref{RHE2d} in the case  that the 1- and 4-waves are
the left-moving rarefaction wave and right-moving shock wave, respectively,
and the $t$-axis is located in the intermediate region, one has:

(i) For the case of $u_*>0$,    $(\partial v /\partial t)_*$  is calculated  as follows
\begin{equation}\label{eq:vl}
\left(\frac{\partial v}{\partial t}\right)_*=-\frac{\rho_* u_*}{\rho_L}v'_L;
\end{equation}

(ii) For the case of $u_*<0$,    $(\partial v /\partial t)_*$ is calculated as follows
\begin{equation}\label{eq:vr}
\left(\frac{\partial v}{\partial t}\right)_*=-\frac{\rho_* u_*(\sigma_{R}-u_R)}{\sigma_{R} -u_*}v'_R.
\end{equation}
\end{theorem}
\begin{proof}
Using Theorem 7.1 of \cite{Ben-Li-Warnecke2006} can complete the proof.
\qed\end{proof}

It is worth noting that  the previous results are still available if the $t$ axis is within the rarefaction wave. Specially, if the $t$ axis is within the left rarefaction wave,
then $\left(\frac{\partial v}{\partial t}\right)_*$ can be obtained  by  \eqref{eq:vl}.

\section{Numerical experiments}
\label{sec:NE}
This section will solve several initial value or initial-boundary-value problems of the one- and two-dimensional RHEs \eqref{eq:RHE} and \eqref{RHE2d} to verify the accuracy and the  discontinuity resolving  capability of the proposed second order accurate GRP schemes. Unless specifically stated,
the adiabatic index $\gamma$ and radiation coefficient $\hat{a}_R$ are chosen as $5/3$ and $1$, respectively, and the CFL number $C_{cfl}$ and the limiter parameter $\theta$ in \eqref{eq:slope} are taken as $0.45$ and   $1.5$, respectively.  

\subsection{One-dimensional case}
Five one-dimensional examples are considered here. The first is used to check the accuracy of the one-dimensional GRP scheme and the others  are the Riemann problems considered in \cite{Dai-Woodward1998,Tang2000Kinetic} and used to validate the performance of the GRP scheme  in resolving the discontinuities. The numerical solutions obtained by the GRP scheme will be drawn in the symbols ``$\circ$'' and compared to the solutions of
 the MUSCL-Hancock scheme \cite[Chapter 14.4]{Toro} (in the symbol ``+'') and the exact or reference solutions in the solid lines. The computational domain $\Omega$ is chosen as $[0,1]$.

\begin{example}[Accuracy test]
\label{ex:1Dacc}
The first example is used to check the accuracy of the one-dimensional GRP scheme  for
 the smooth solution of  RHEs
\[
\left(\rho,u,p_{tot}\right)(x,t)=\left(1+0.2\sin(2\pi(x-ut)),0.2,1\right),
\]
which describes a sine wave propagating periodically within the domain $\Omega=[0,1]$.

The computational domain $\Omega$ is divided into $N$ uniform cells and the periodic boundary conditions are specified at $x=0,1$. Table \ref{tab:1Dacc1} gives the $l^{p}$ errors in the density at $t=0.5$ and corresponding convergence rates of the GRP scheme, where $p=1,2,\infty$. The data show that the convergence rates of second-order can be almost obtained in $l^1$ and $l^2$ norms, but  the  $l^\infty$ convergence rates is lower than 1.5 due to the nonlinear limiter \eqref{eq:slope}.

%

\begin{table}[htbp]
\caption{Example \ref{ex:1Dacc}:
The $l^p$ errors in density at $t=0.5$ and corresponding convergence rates  of the one-dimensional GRP scheme, where $p=1,2,\infty$.}\label{tab:1Dacc1}
\centering
\begin{tabular}{|c|c|c|c|c|c|c|}
  \hline
 N &$l^1$-error &$l^1$-order    &$l^2$-error &$l^2$-order & $l^{\infty}$-error &$l^{\infty}$-order \\
   \hline
10 & 7.91e-4 & $-$  & 2.85e-3 & $-$ &1.64e-2  &$-$ \\
20 & 2.23e-4 & 1.83  & 8.83e-4 & 1.69     &6.01e-3  &1.45  \\
40 & 5.93e-5 & 1.91 & 2.72e-4 & 1.70 &2.42e-3 &1.31 \\
80 & 1.40e-5 & 2.09 & 8.16e-5 & 1.73 &9.57e-4 &1.34 \\
160 & 3.37e-6 & 2.05 & 2.43e-5 & 1.75&3.70e-4 &1.37 \\
320 & 8.14e-7 & 2.05 & 7.18e-6 & 1.76&1.41e-4 &1.39\\
\hline
\end{tabular}
\end{table}
\end{example}

\begin{example}[Riemann problem 1]\label{ex:RP1}
The initial data of this problem are given as follows
\[
(\rho,u,T)(x,0)=\left\{ \begin{array}{cc}
(1,50,0.5), & x<0.6, \\ (2,-40,1), & x>0.6.
\end{array}
\right. \]
The exact solutions at  time $t>0$ involve two shock waves and a contact discontinuity.
Fig. \ref{fig:RP1} gives the  numerical solutions at $t=0.04$
 obtained by using the GRP and  MUSCL-Hancock schemes with 200 uniform cells with $\theta=1$.
 It is seen that   the numerical solutions are  in good agreement with the exact
  and the GRP scheme  captures the shock waves and contact discontinuity better than the MUSCL-Hancock scheme.
\end{example}
\begin{figure}
  \centering
  \subfigure[Temperature $T$]
  {
  \includegraphics[width=0.4\textwidth]{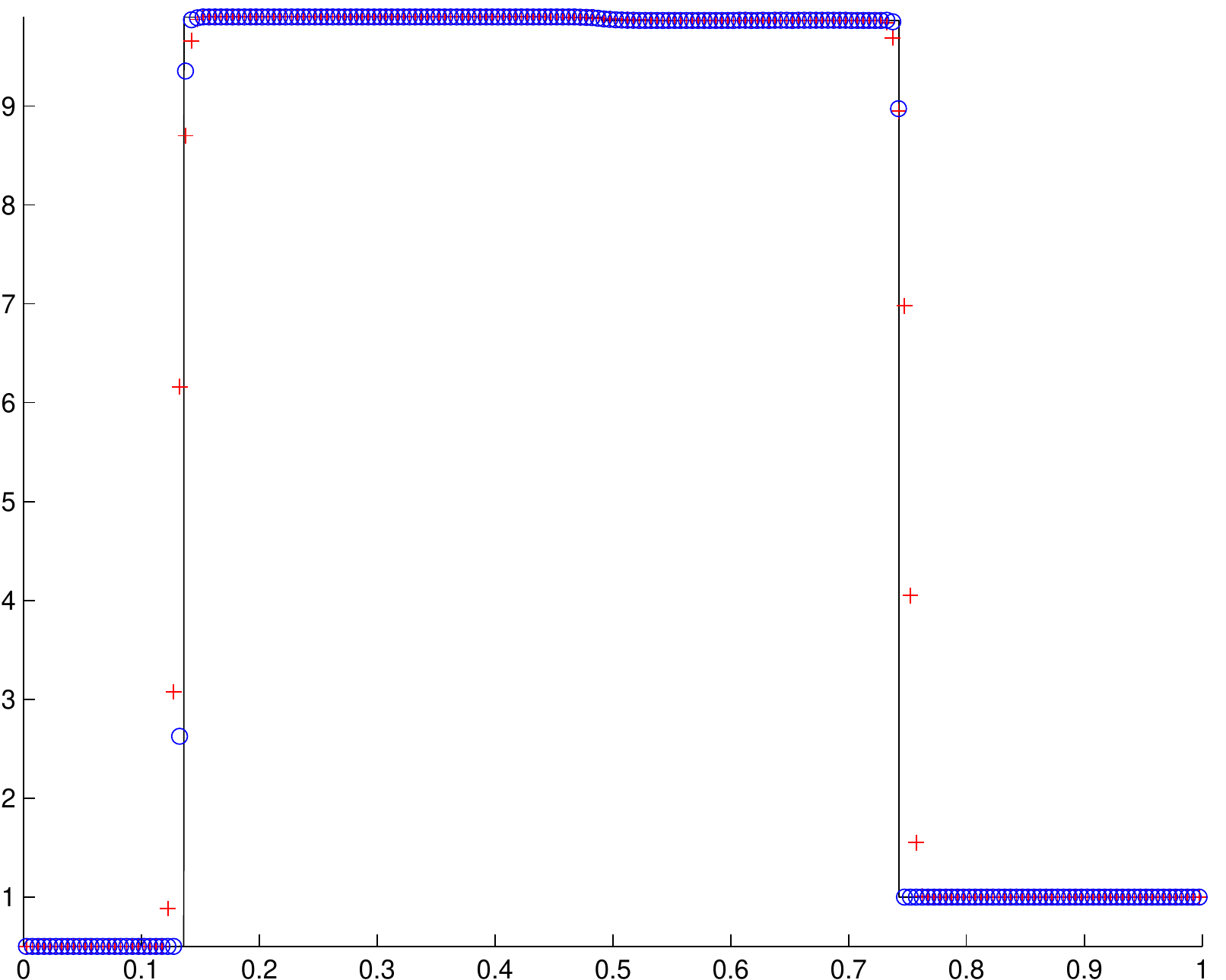}
  }
  \subfigure[Density $\rho$]
  {
  \includegraphics[width=0.4\textwidth]{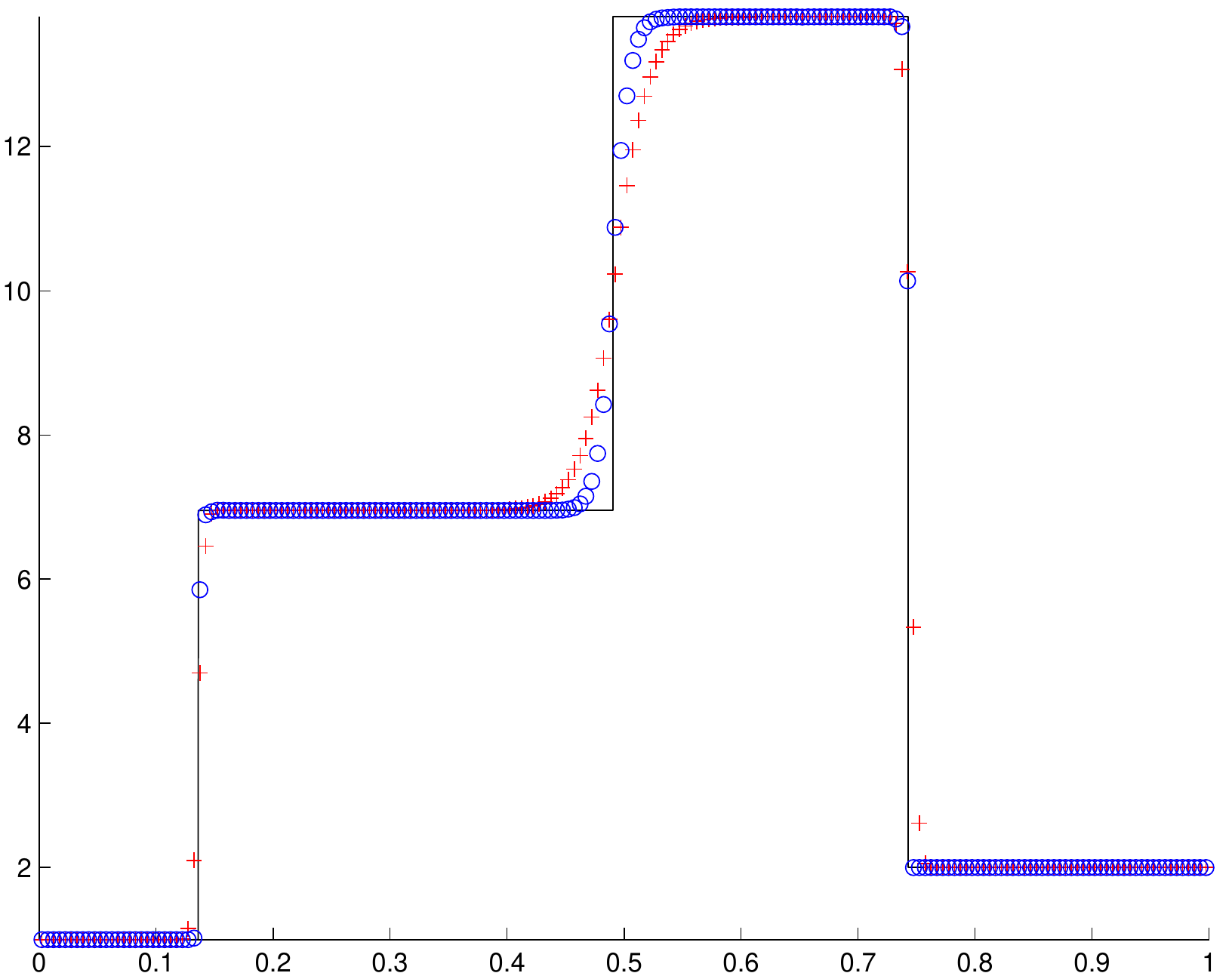}
   }

     \centering
  \subfigure[Total pressure $p_{tot}$]
  {
  \includegraphics[width=0.4\textwidth]{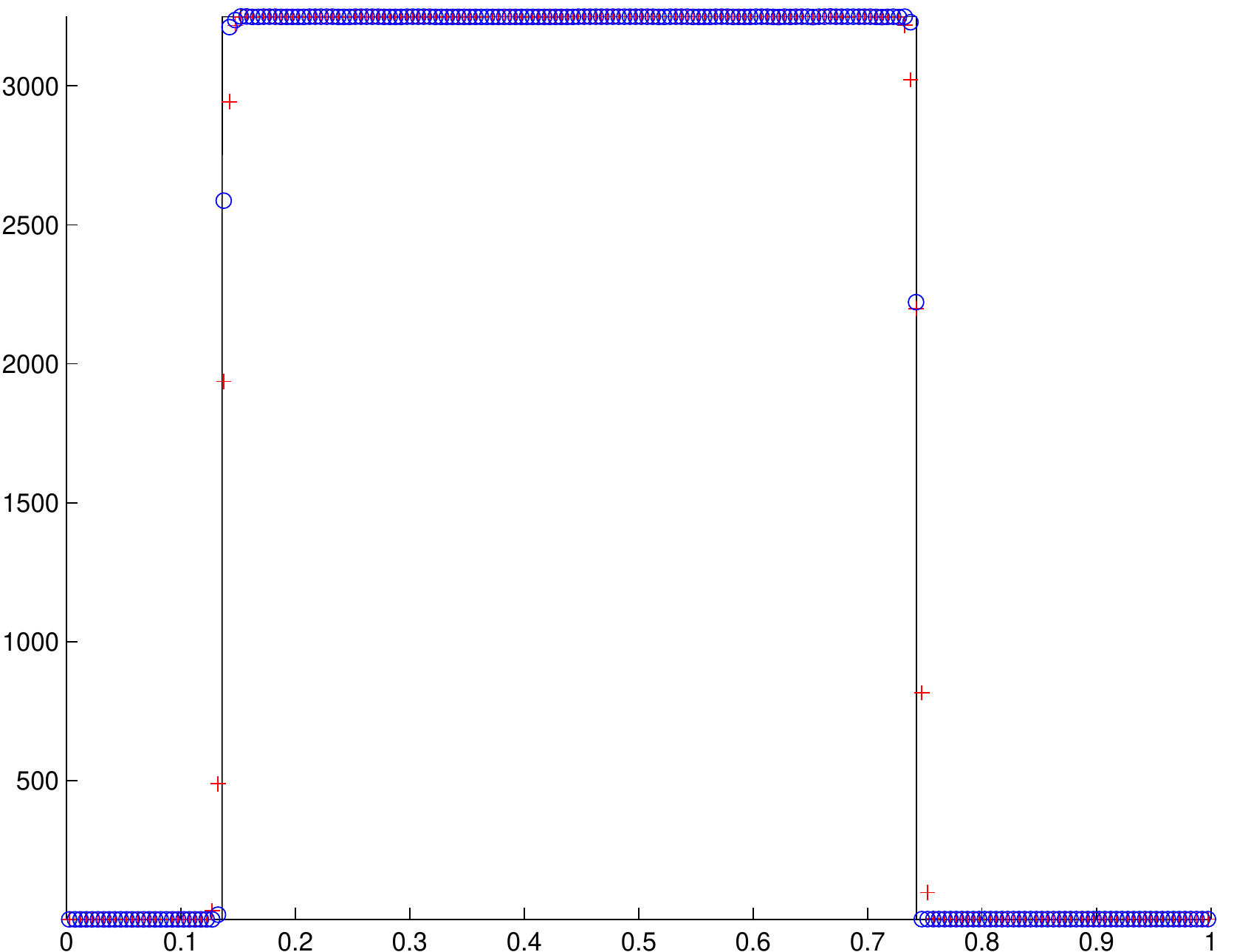}\
  }
  \centering
   \subfigure[Velocity $u$]
  {
  \includegraphics[width=0.4\textwidth]{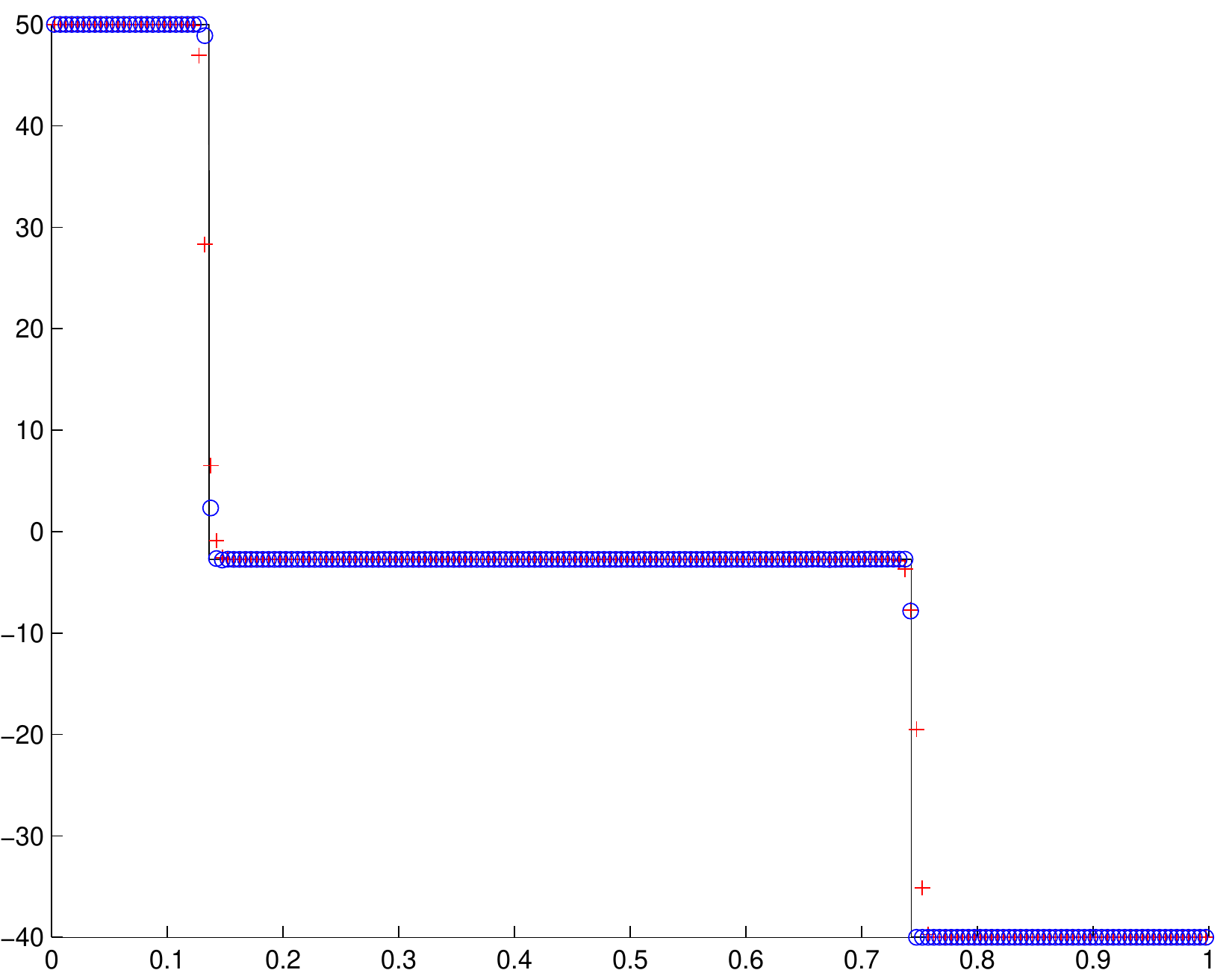}
  }

    \centering
  \subfigure[Sound speed $c$]
  {
  \includegraphics[width=0.4\textwidth]{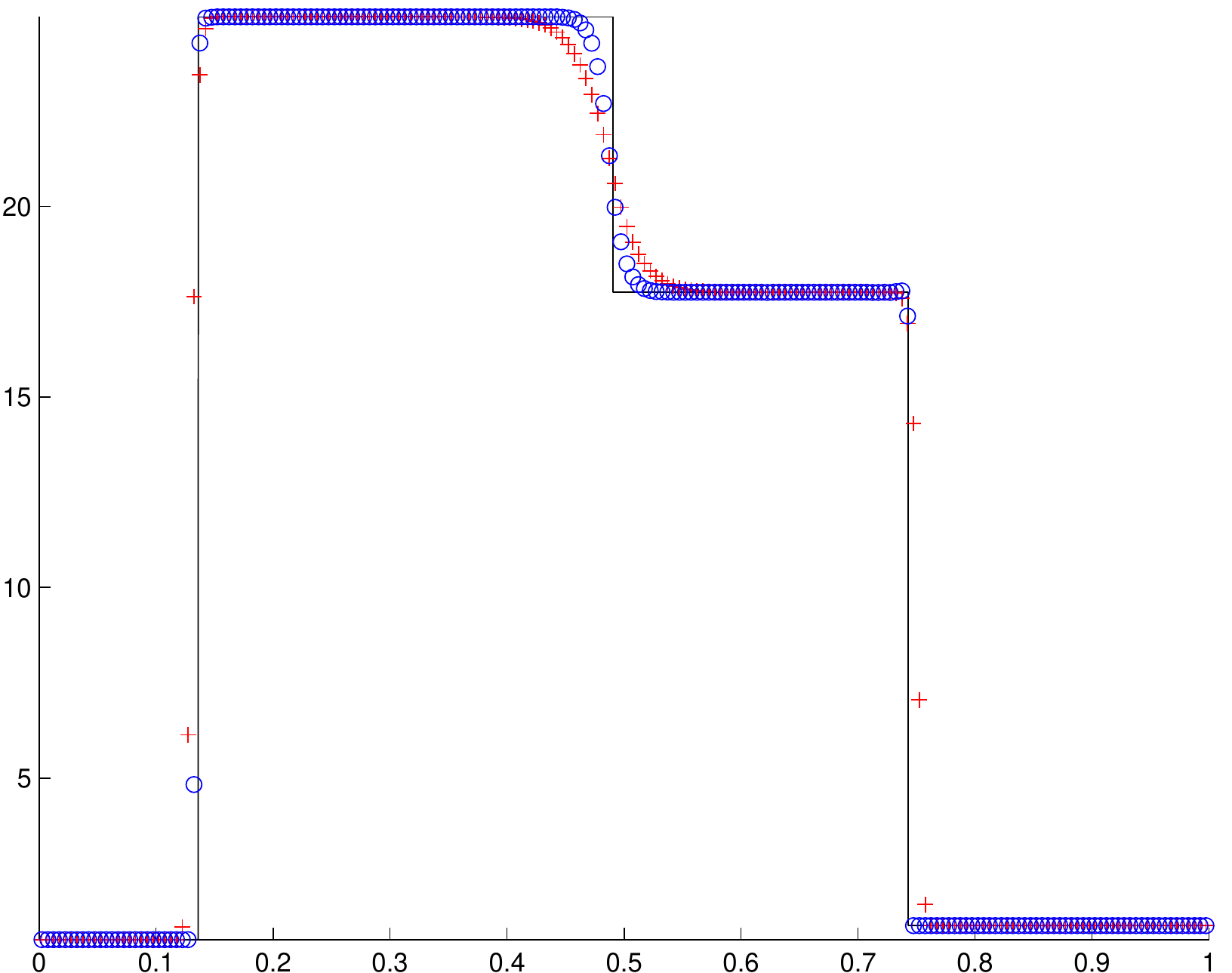}
  }
  \subfigure[Radiative pressure $p_{r}:=p_{tot}-p$]
  {
  \includegraphics[width=0.4\textwidth]{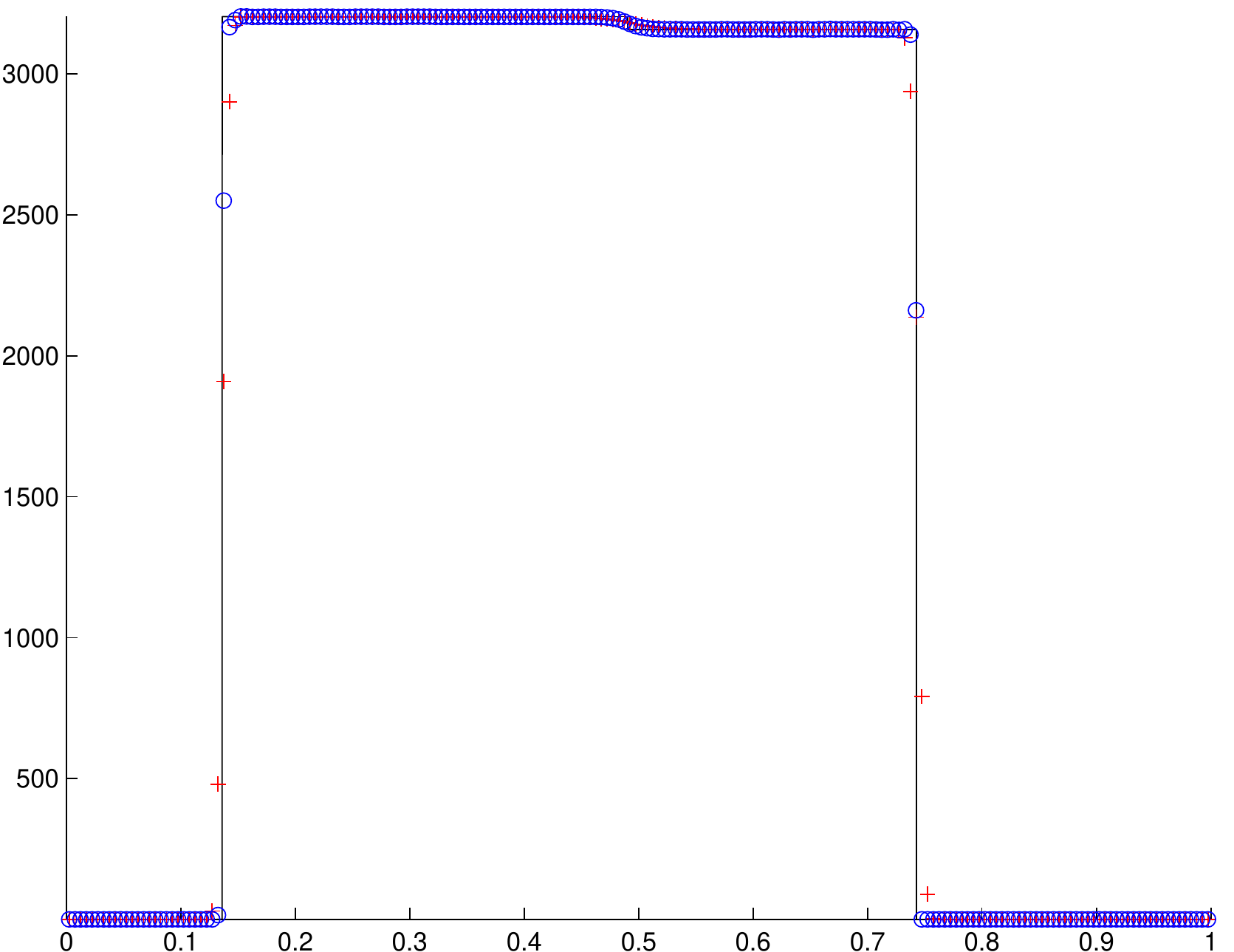}
  }
  \caption{Example \ref{ex:RP1}: The solutions at $t=0.04$ with $200$ uniform cells.}
  \label{fig:RP1}
\end{figure}

\begin{example}[Riemann problem 2]\label{ex:RP2}
It is similar to the first Riemann problem, but with much stronger shock waves
and larger initial velocity jump.
 The initial data are given by
\[
(\rho,u,T)(x,0)=\left\{ \begin{array}{cc}
(1,150,0.5), & x<0.6, \\ (2,-100,1), & x>0.6.
\end{array}
\right. \]
Fig. \ref{fig:RP2} gives the  numerical solutions at $t=0.018$
 obtained by using the GRP and MUSCL-Hancock schemes with 200 uniform cells {with $\theta=1$}.
It is seen that the numerical solutions are in good agreement with the exact solutions,
 and the GRP scheme  does work well for very strong shock waves and captures them and the contact discontinuity better than the MUSCL-Hancock scheme.
\end{example}
\begin{figure}
  \centering
  \subfigure[Temperature  $T$]
  {
  \includegraphics[width=0.4\textwidth]{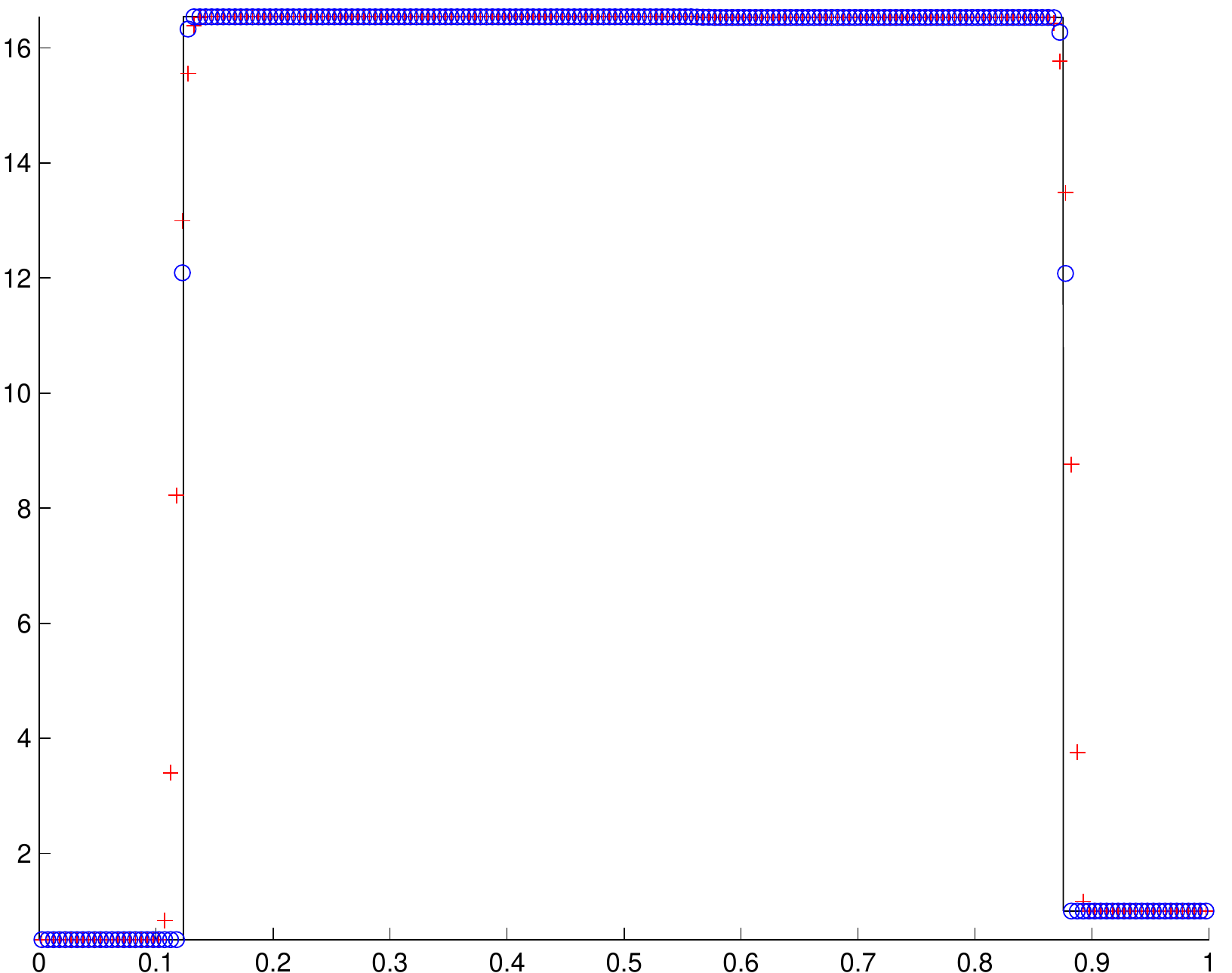}
  }
  \subfigure[Density $\rho$]
  {
  \includegraphics[width=0.4\textwidth]{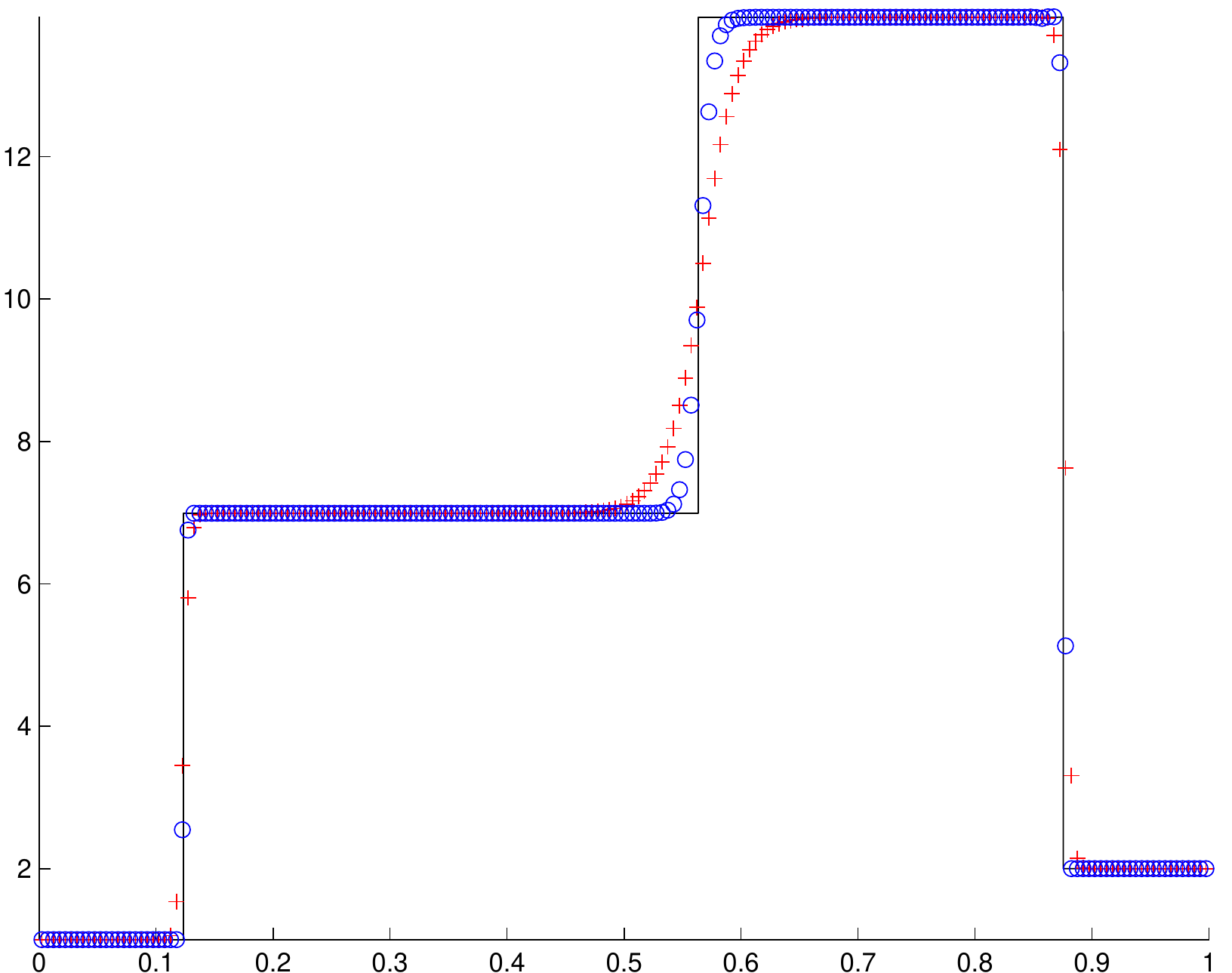}
   }

     \centering
  \subfigure[Total pressure $p_{tot}$]
  {
  \includegraphics[width=0.4\textwidth]{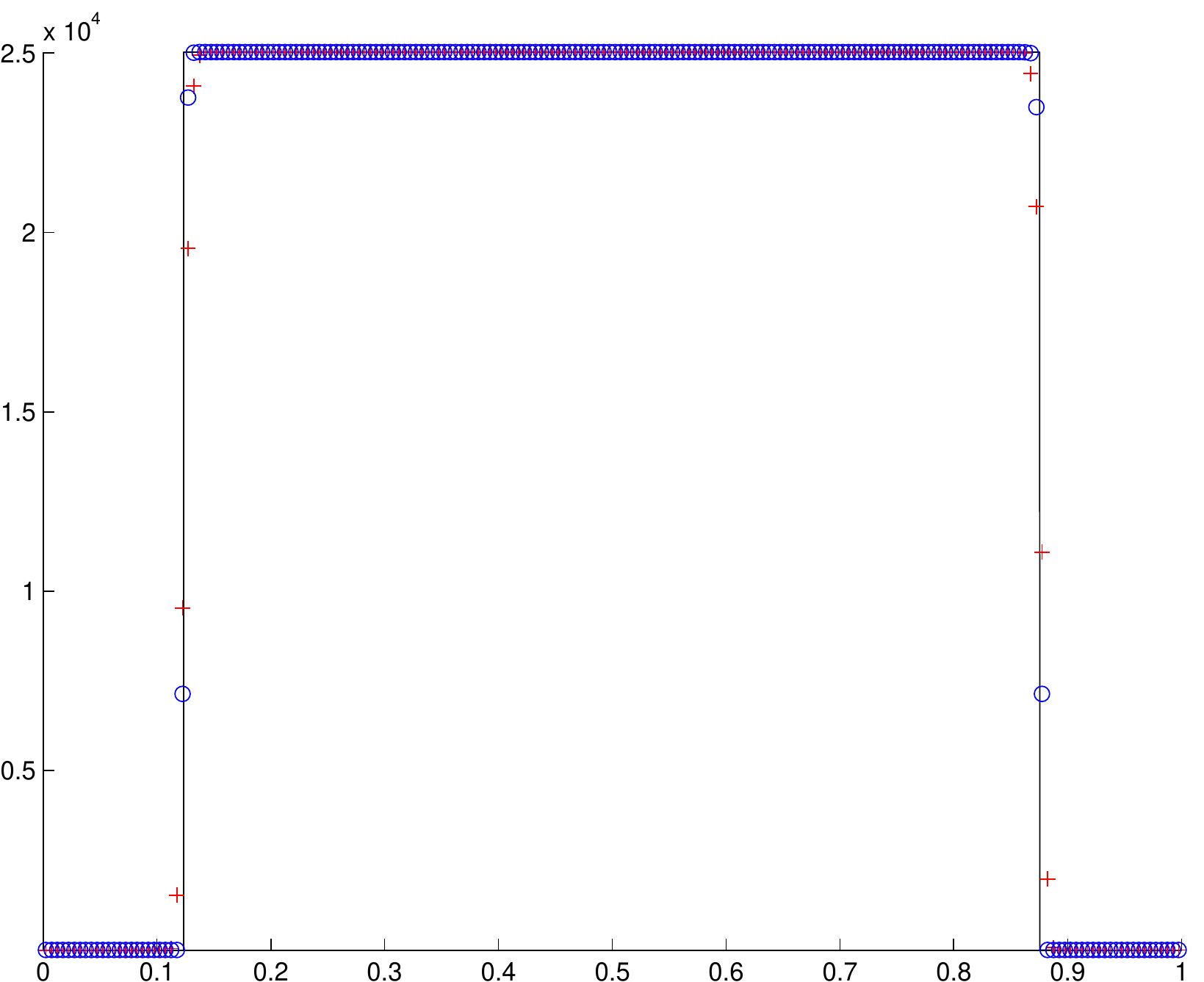}\
  }
  \centering
   \subfigure[Velocity $u$]
  {
  \includegraphics[width=0.4\textwidth]{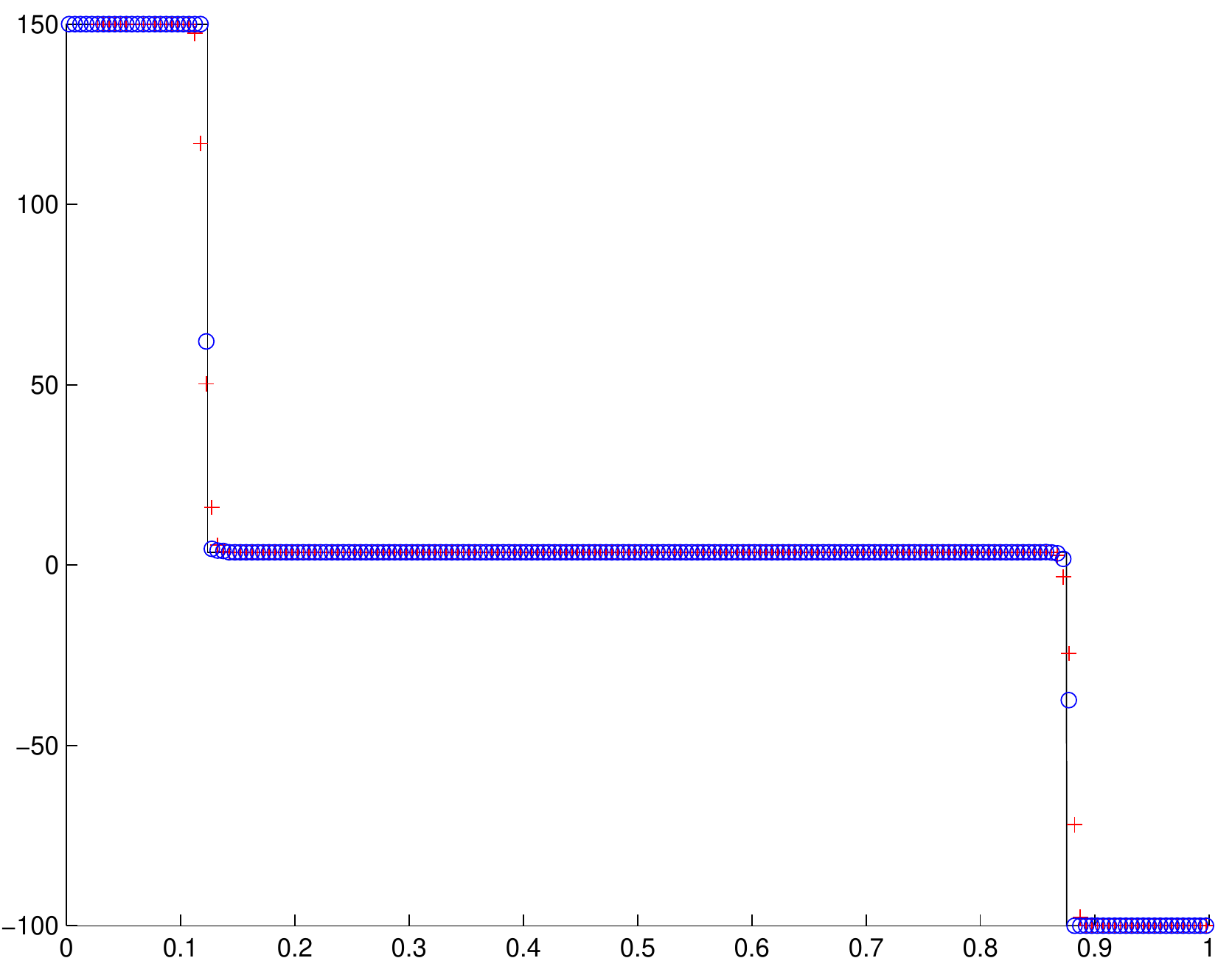}
  }

    \centering
  \subfigure[Sound speed $c$]
  {
  \includegraphics[width=0.4\textwidth]{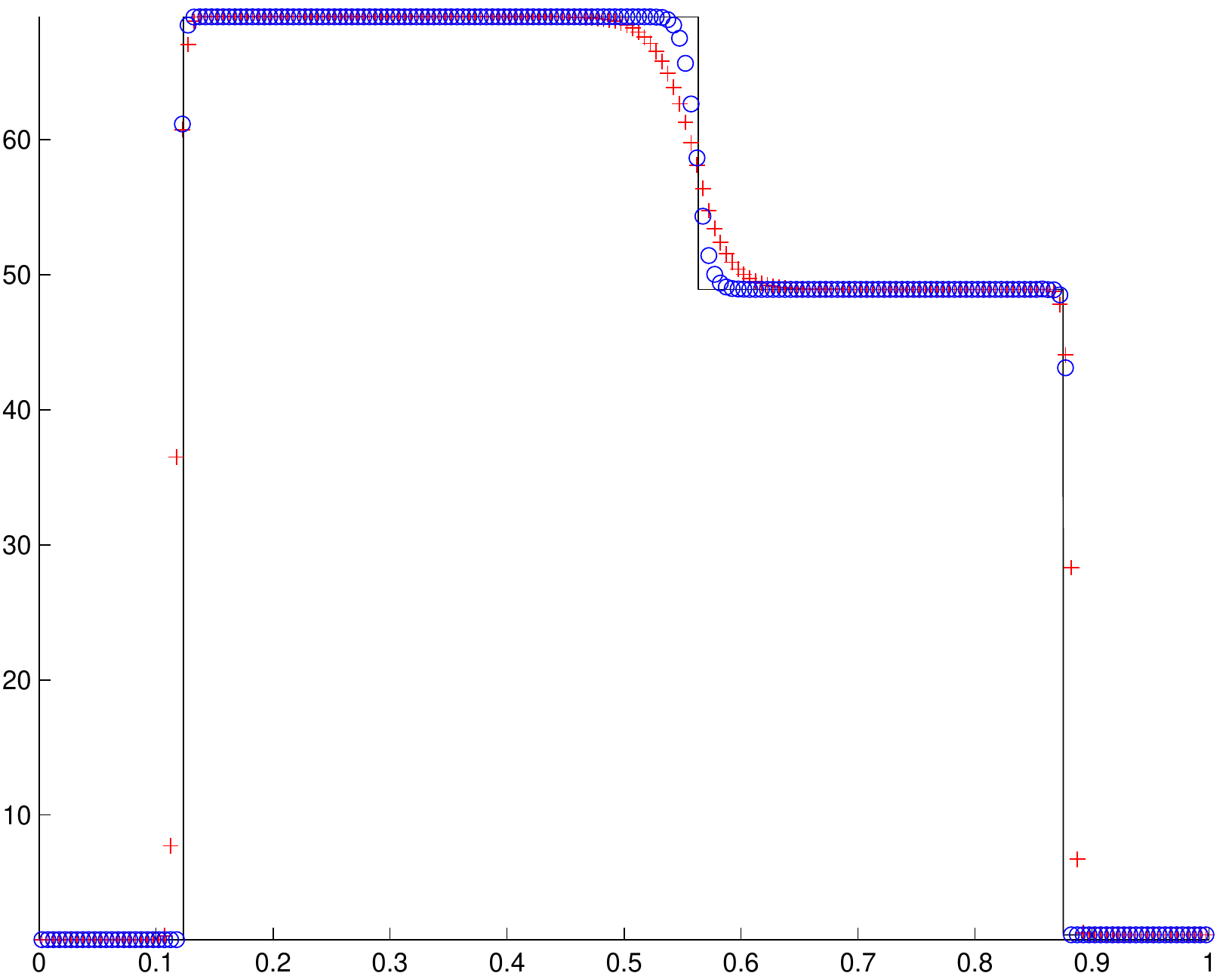}
  }
  \subfigure[Radiative pressure $p_{r}$]
  {
  \includegraphics[width=0.4\textwidth]{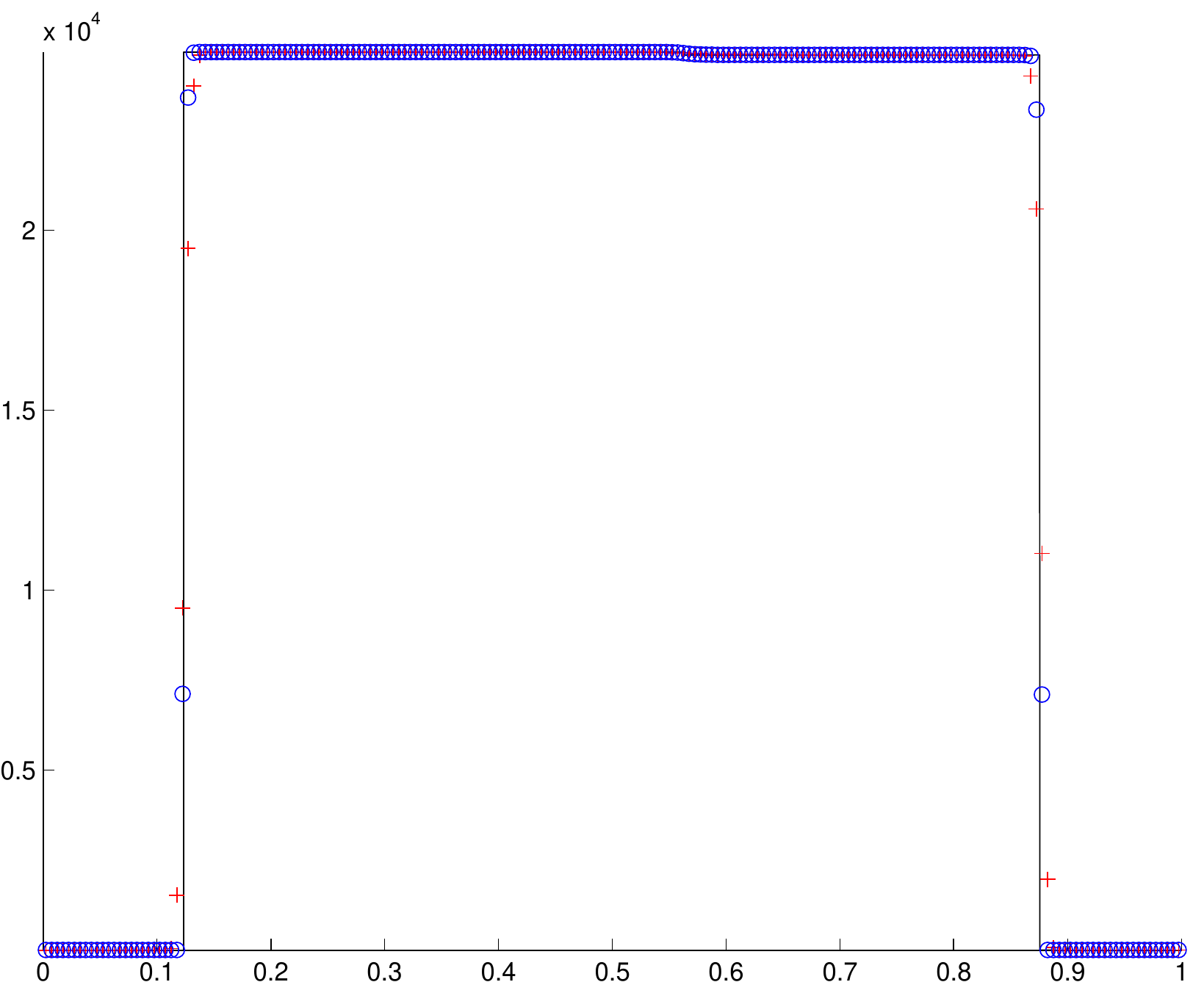}
  }
  \caption{Example \ref{ex:RP2}: The solutions at $t=0.018$ with $200$ uniform cells.}\label{fig:RP2}
\end{figure}

\begin{example}[Riemann problem 3]\rm\label{ex:RP3}
The initial data of the third one-dimensional Riemann problem are taken as
\[
(\rho,u,T)(x,0)=\left\{ \begin{array}{ll}
   (1,-1,1), & x<0.5,
\\ (1,1,1), & x>0.5.
\end{array}
\right.
\]
Fig. \ref{fig:RP3} gives the numerical solutions at $t=0.2$ obtained by using the  GRP and MUSCL-Hancock schemes.  It is seen that the solutions contain the left- and right-moving rarefaction waves, there is a undershoot (resp. overshoot) ``bubble'' in the computed density (resp. temperature) due to the numerical wall-heating, and  the GRP scheme resolve those rarefaction waves much better than the  MUSCL-Hancock scheme and  produces the less serious wall-heating phenomenon at $x=0.5$.
\end{example}
\begin{figure}
  \centering
  \subfigure[Temperature $T$]
  {
  \includegraphics[width=0.4\textwidth]{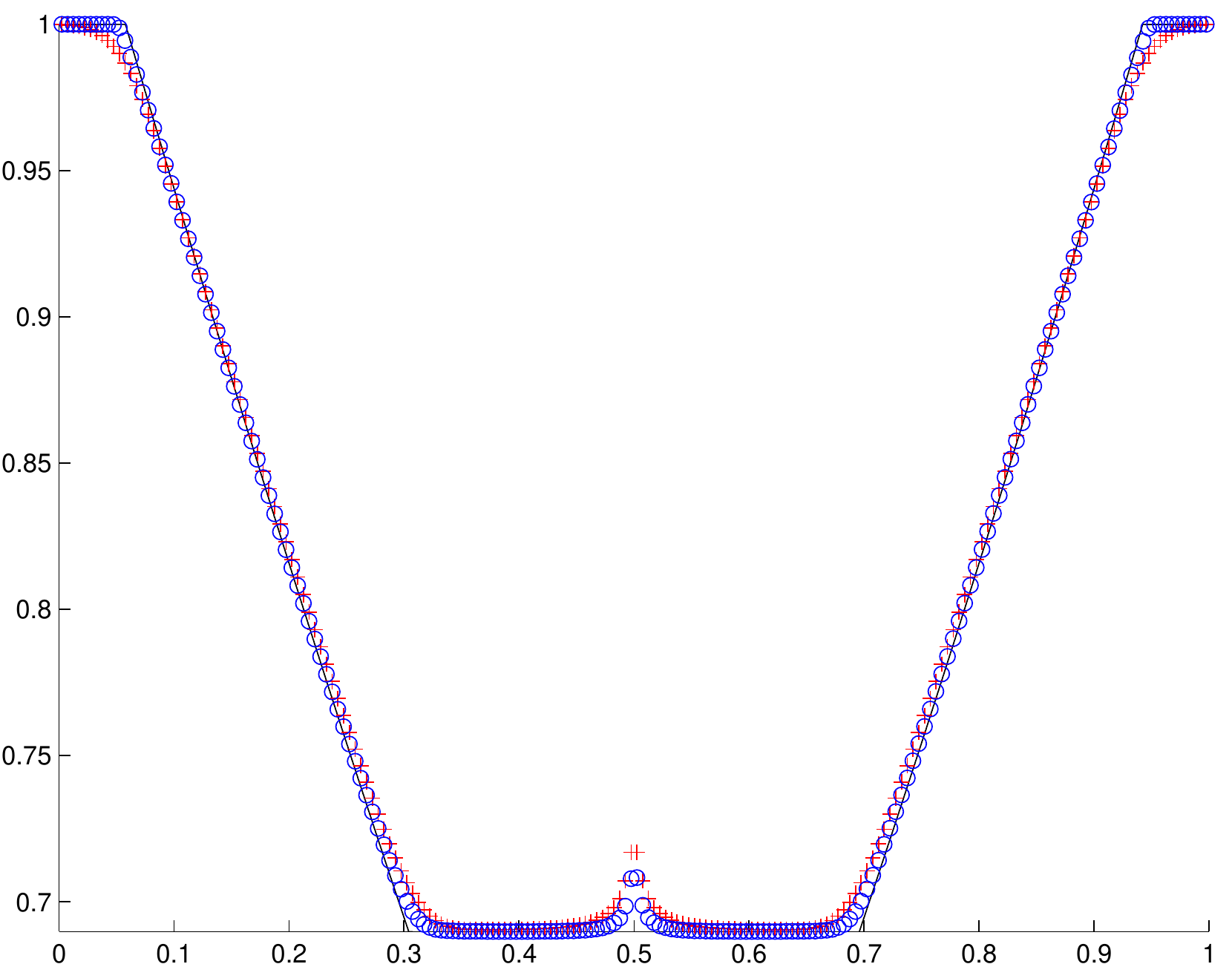}
  }
  \subfigure[Density $\rho$]
  {
  \includegraphics[width=0.4\textwidth]{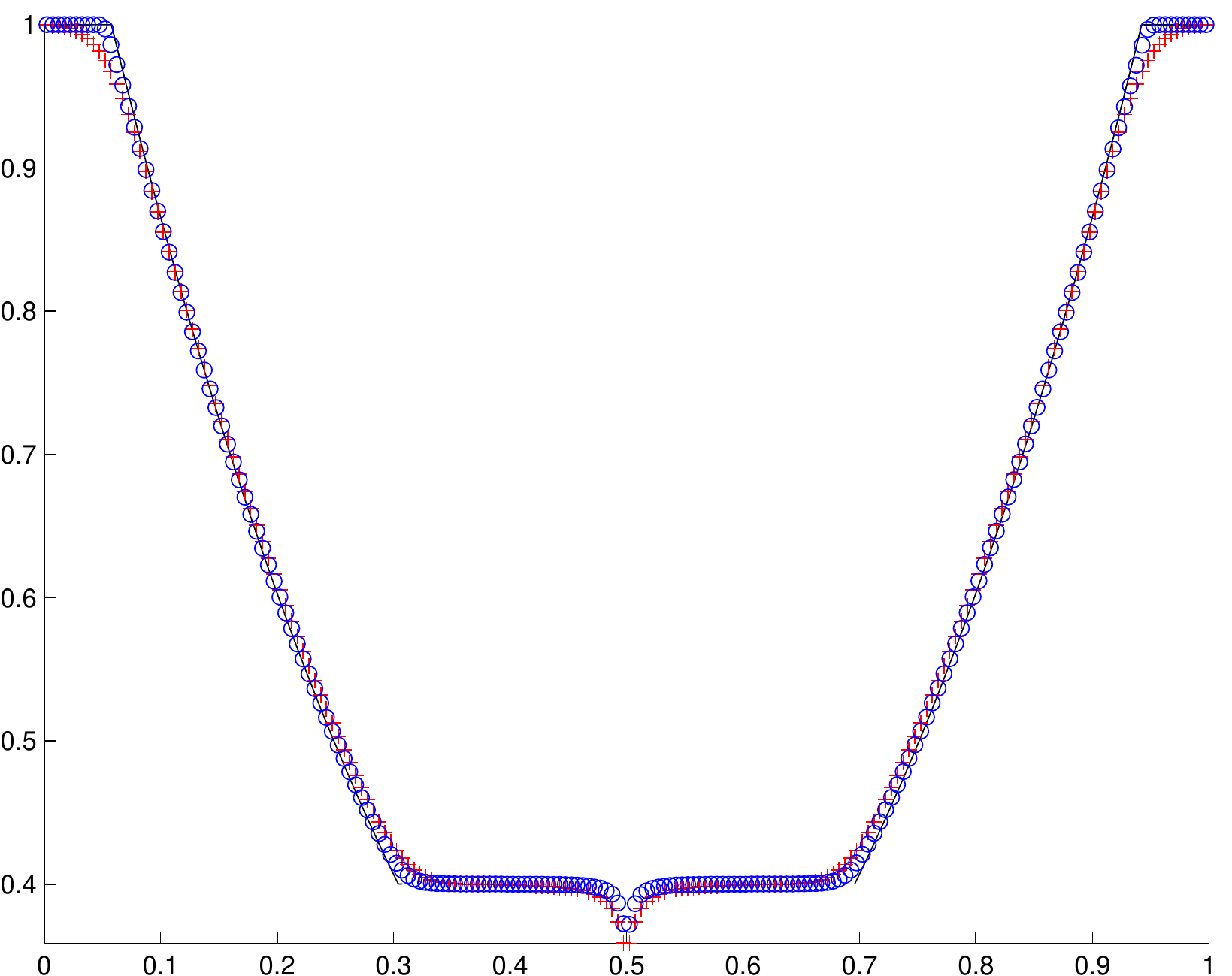}
   }

     \centering
  \subfigure[Total pressure $p_{tot}$]
  {
  \includegraphics[width=0.4\textwidth]{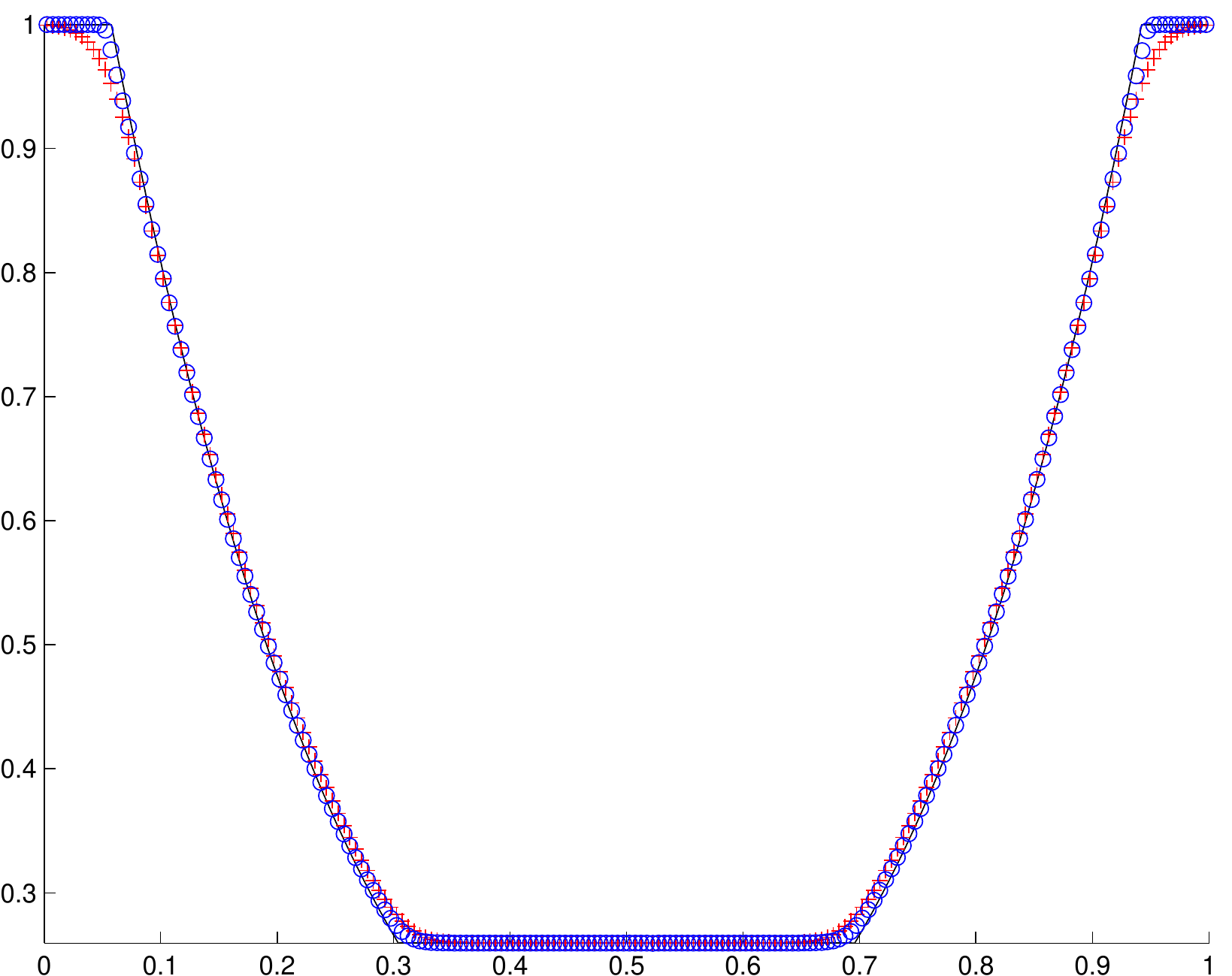}\
  }
  \centering
   \subfigure[Velocity $u$]
  {
  \includegraphics[width=0.4\textwidth]{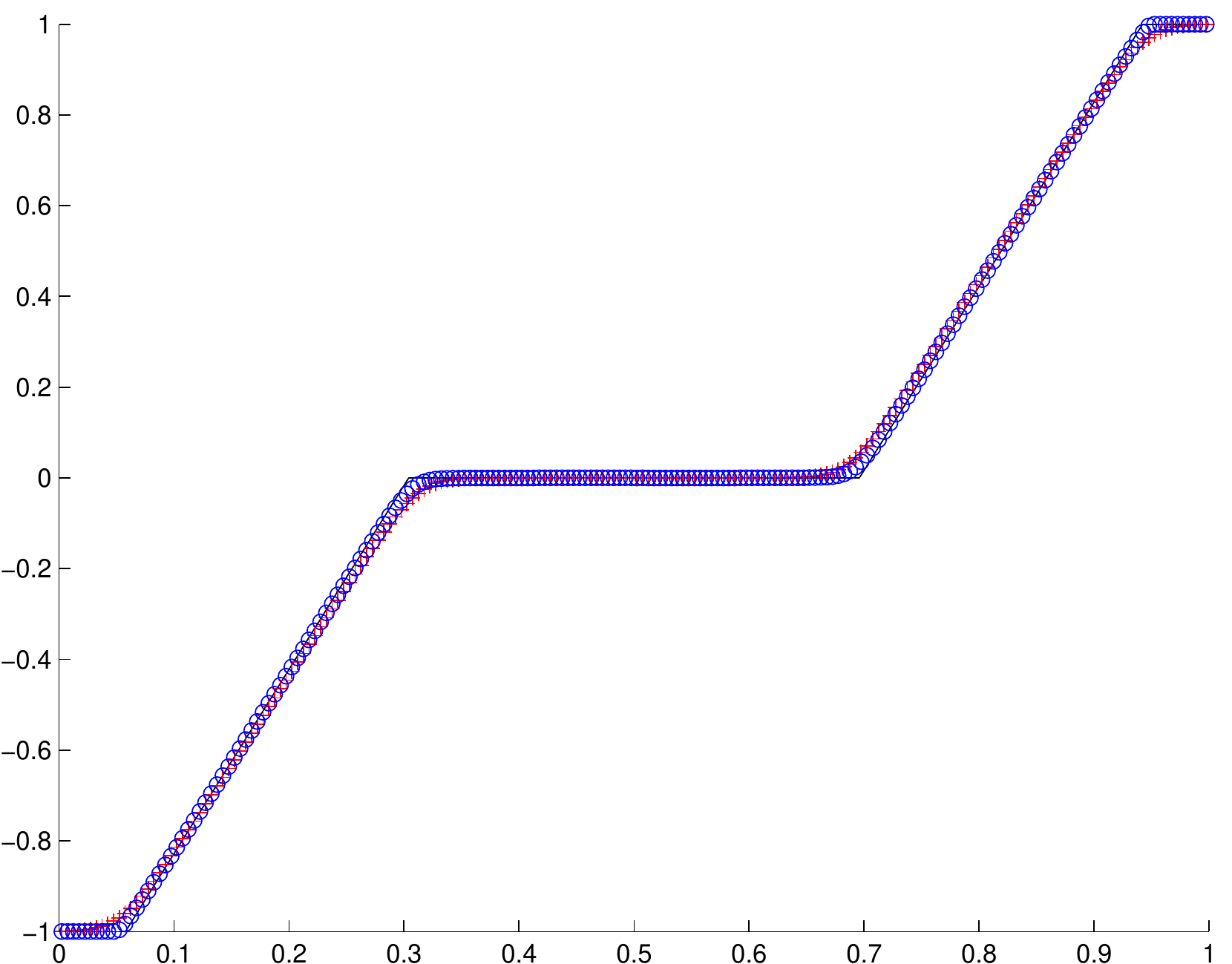}
  }

    \centering
  \subfigure[Sound speed $c$]
  {
  \includegraphics[width=0.4\textwidth]{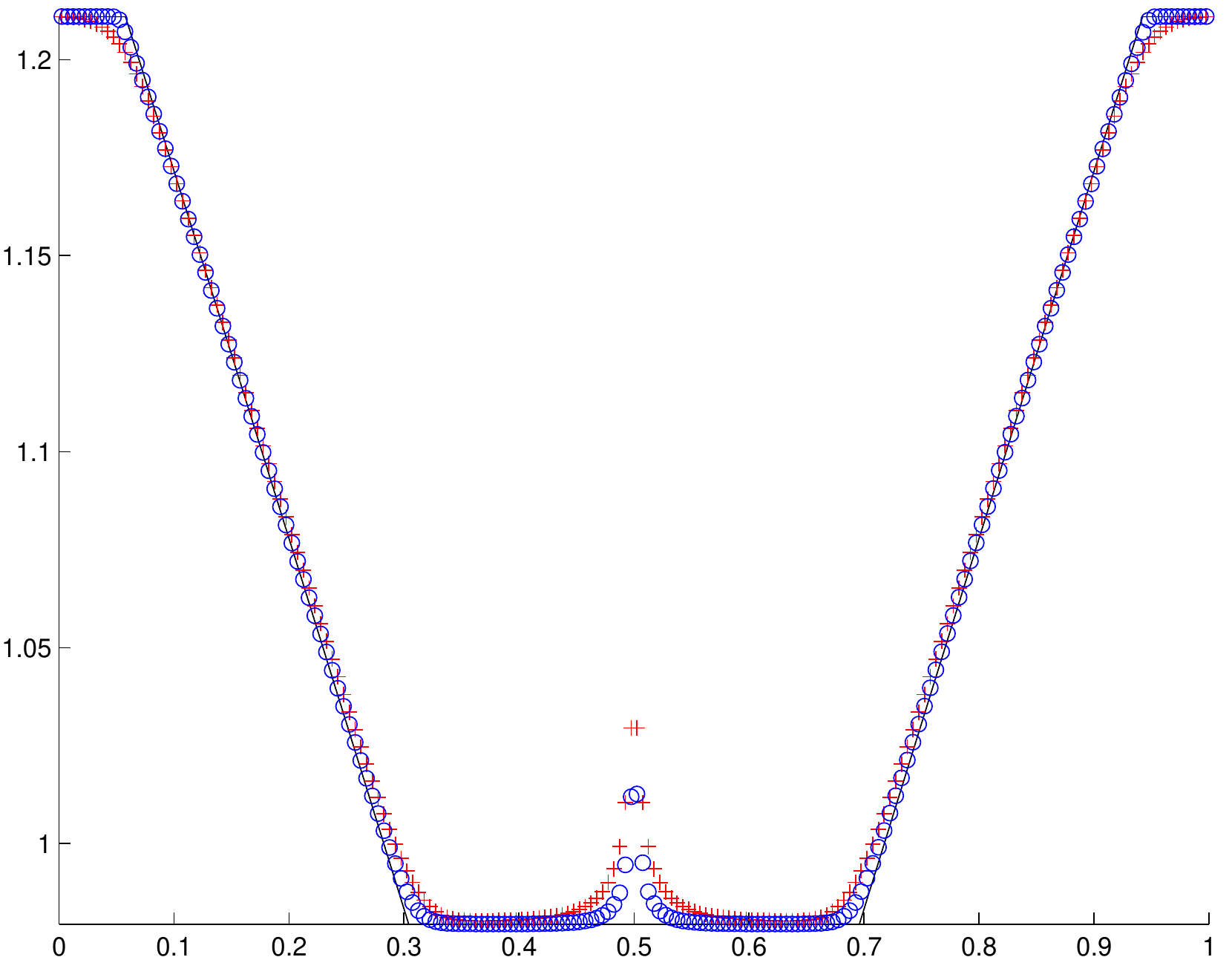}
  }
  \subfigure[Radiative pressure  $p_{r}$]
  {
  \includegraphics[width=0.4\textwidth]{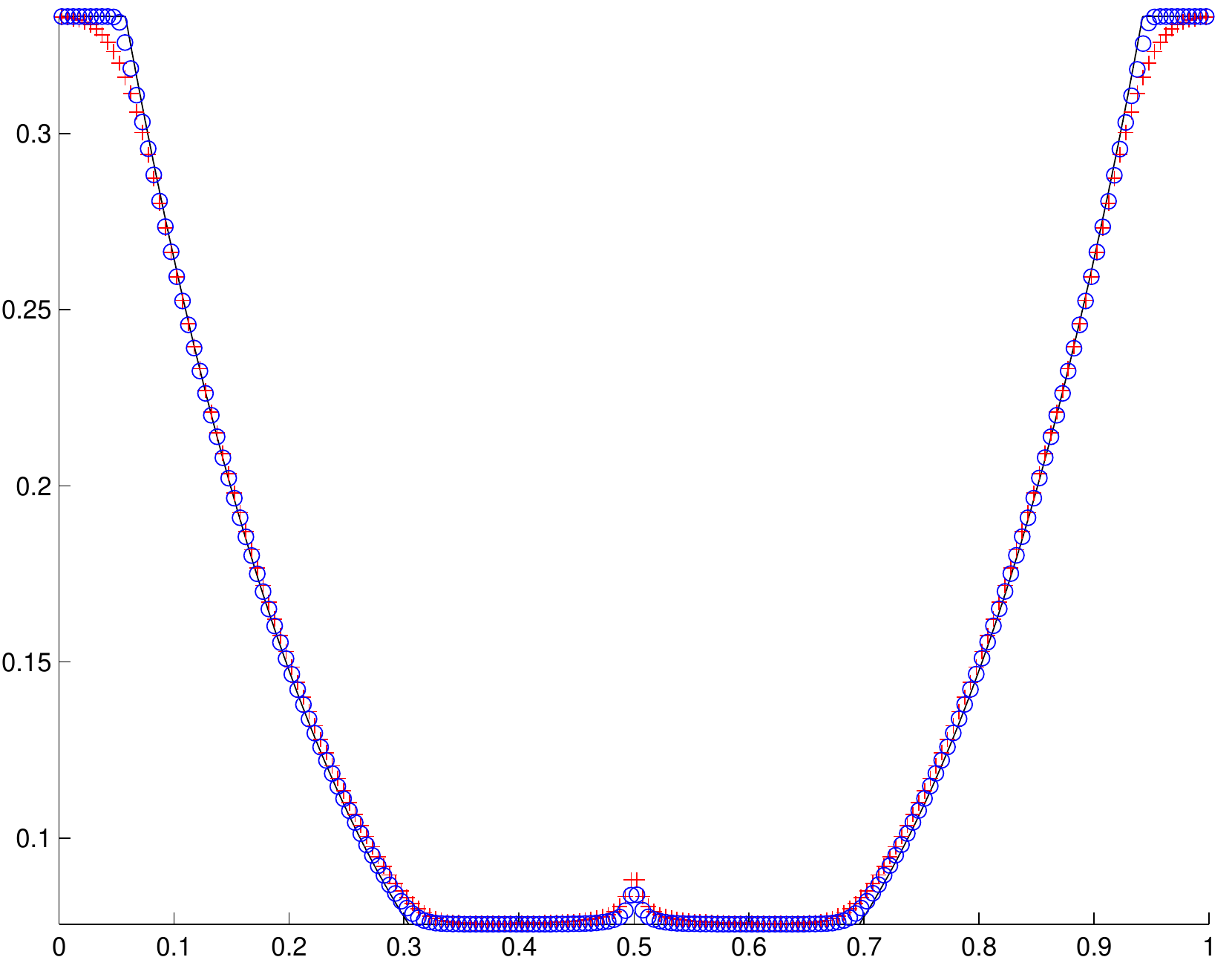}
  }
  \caption{Example \ref{ex:RP3}: The solutions at $t=0.2$ with $200$ uniform cells.  }\label{fig:RP3}
\end{figure}

\begin{example}[Riemann problem 4]\rm\label{ex:RP4}
The initial data of the last one-dimensional problem are
\[
(\rho,u,T)(x,0)=\left\{ \begin{array}{cc}
   (1,0,6), & x<0.5,
\\ (1,0,1.5), & x>0.5.
\end{array}
\right. \]
 Fig. \ref{fig:RP4} displays the solutions at $t=0.01$ obtained by using the second order GRP and MUSCL-Hancock schemes with 200 uniform cells. It is shown that the solutions involve a left-moving rarefaction wave, a contact discontinuity, and  a right-moving shock wave,  the GRP scheme resolves the narrow region between the contact discontinuity and shock wave much better than the MUSCL-Hancock scheme, the obvious overshoot is observed near the tail of the numerical rarefaction waves.
\end{example}
\begin{figure}
  \centering
  \subfigure[Temperature  $T$]
  {
  \includegraphics[width=0.4\textwidth]{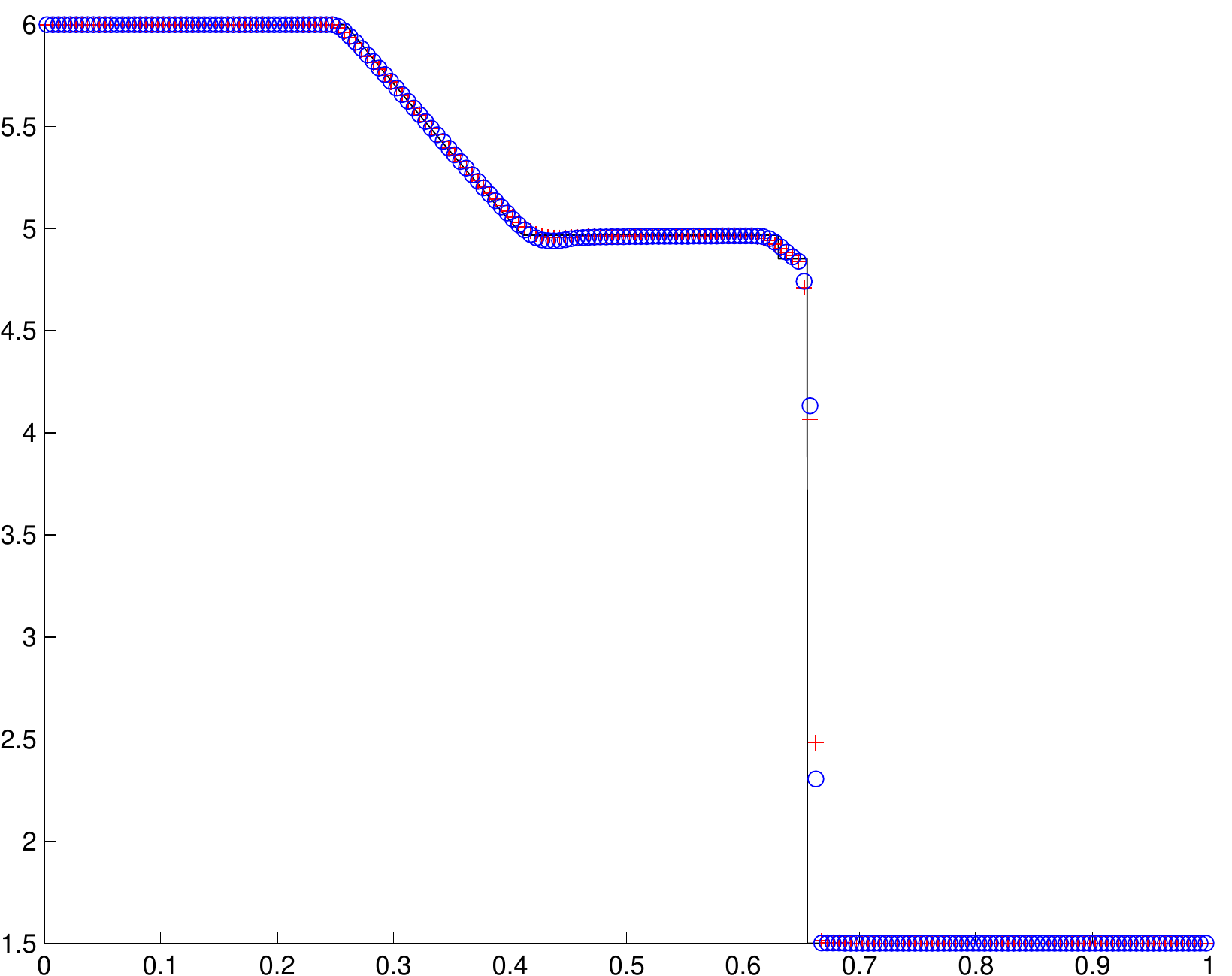}
  }
  \subfigure[Density $\rho$]
  {
  \includegraphics[width=0.4\textwidth]{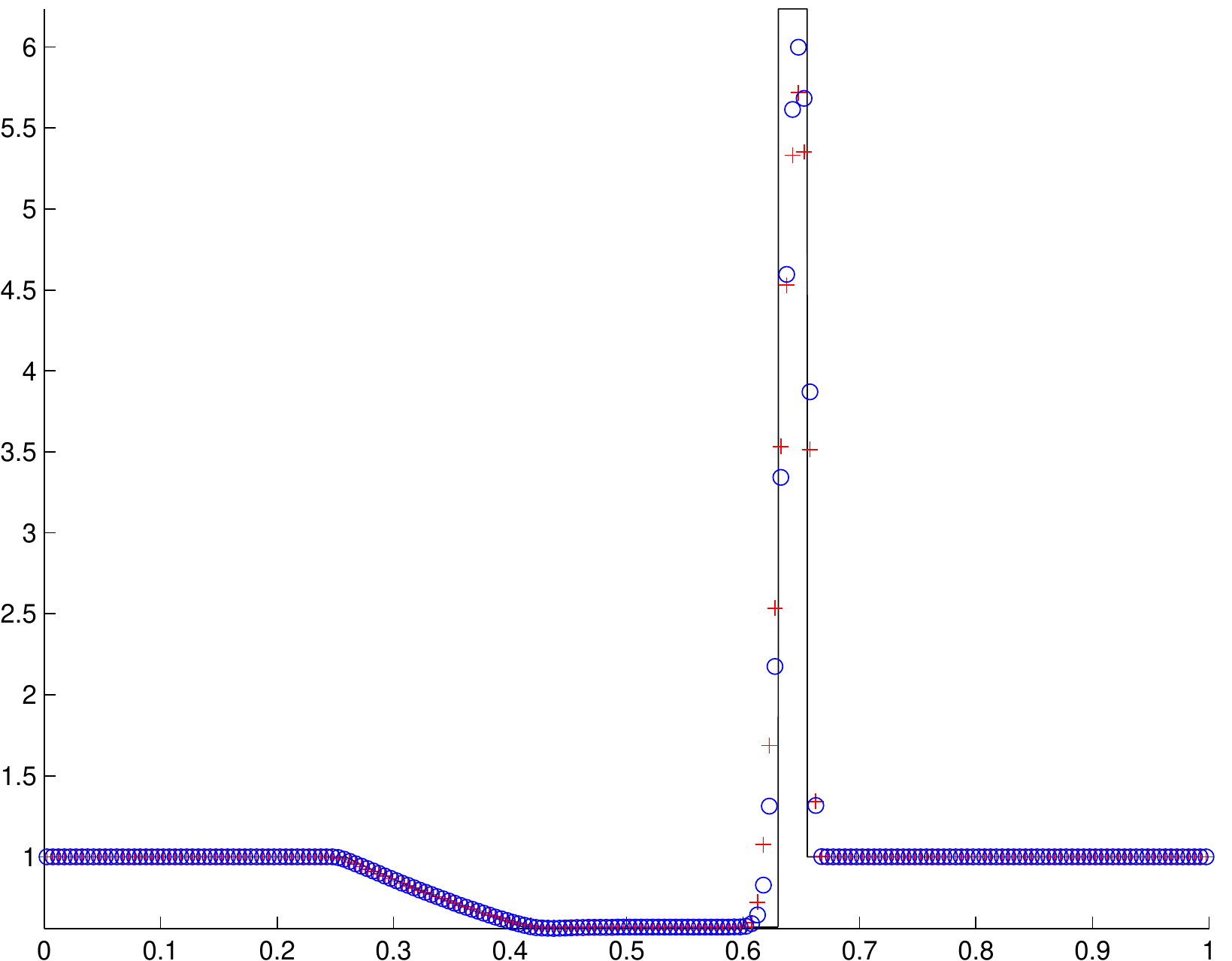}
   }

     \centering
  \subfigure[Total pressure $p_{tot}$]
  {
  \includegraphics[width=0.4\textwidth]{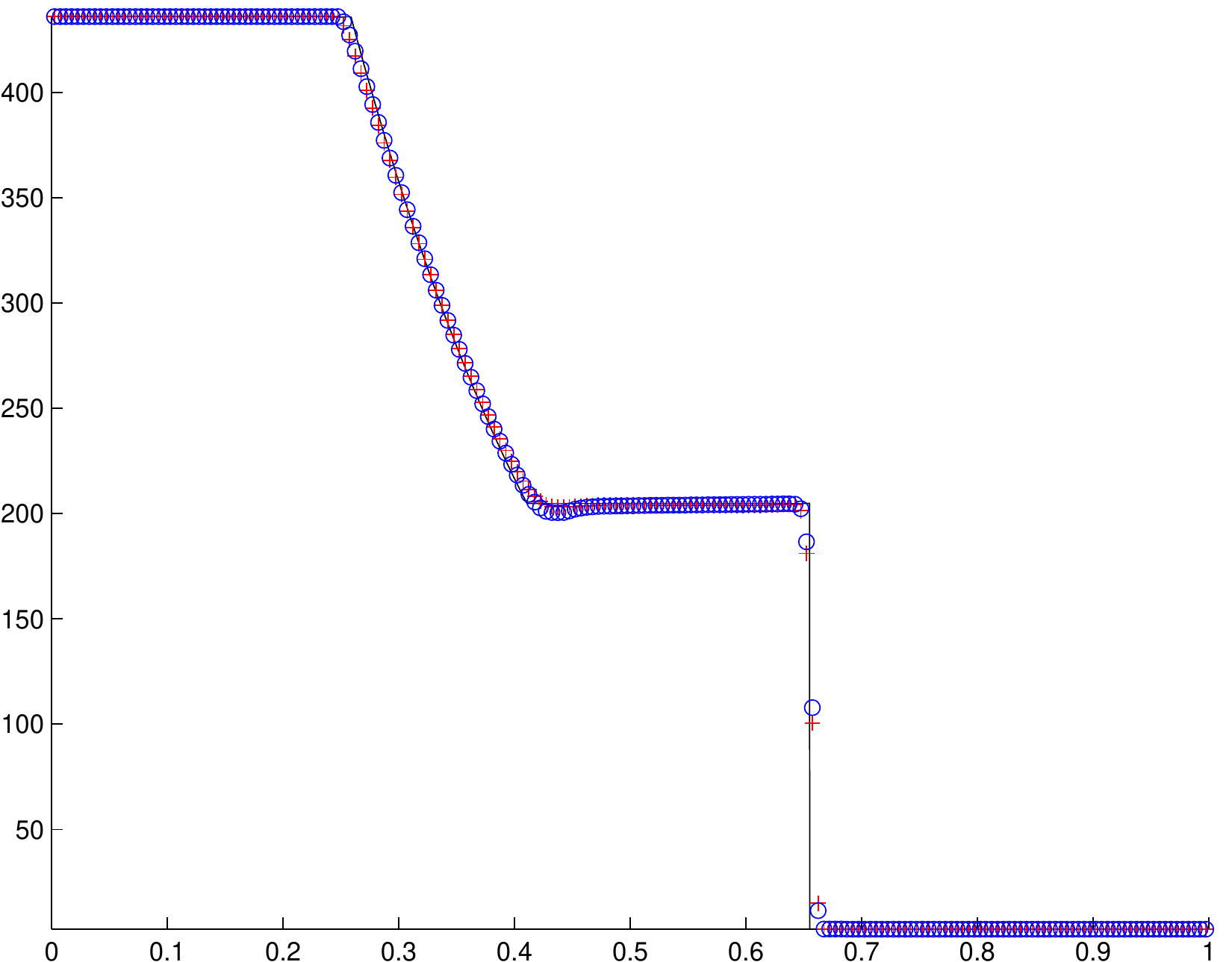}\
  }
  \centering
   \subfigure[Velocity $u$]
  {
  \includegraphics[width=0.4\textwidth]{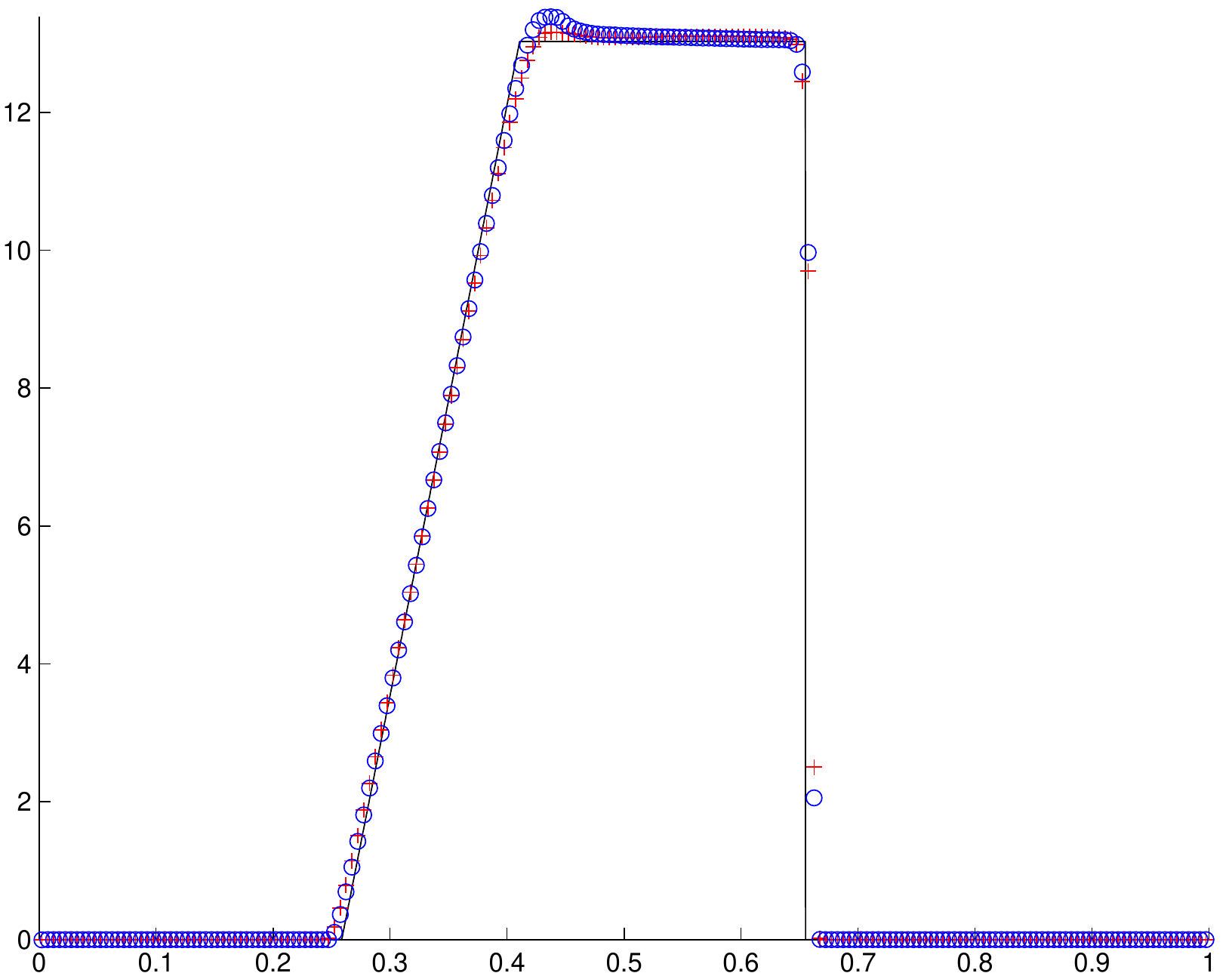}
  }

    \centering
  \subfigure[ Sound speed  $c$]
  {
  \includegraphics[width=0.4\textwidth]{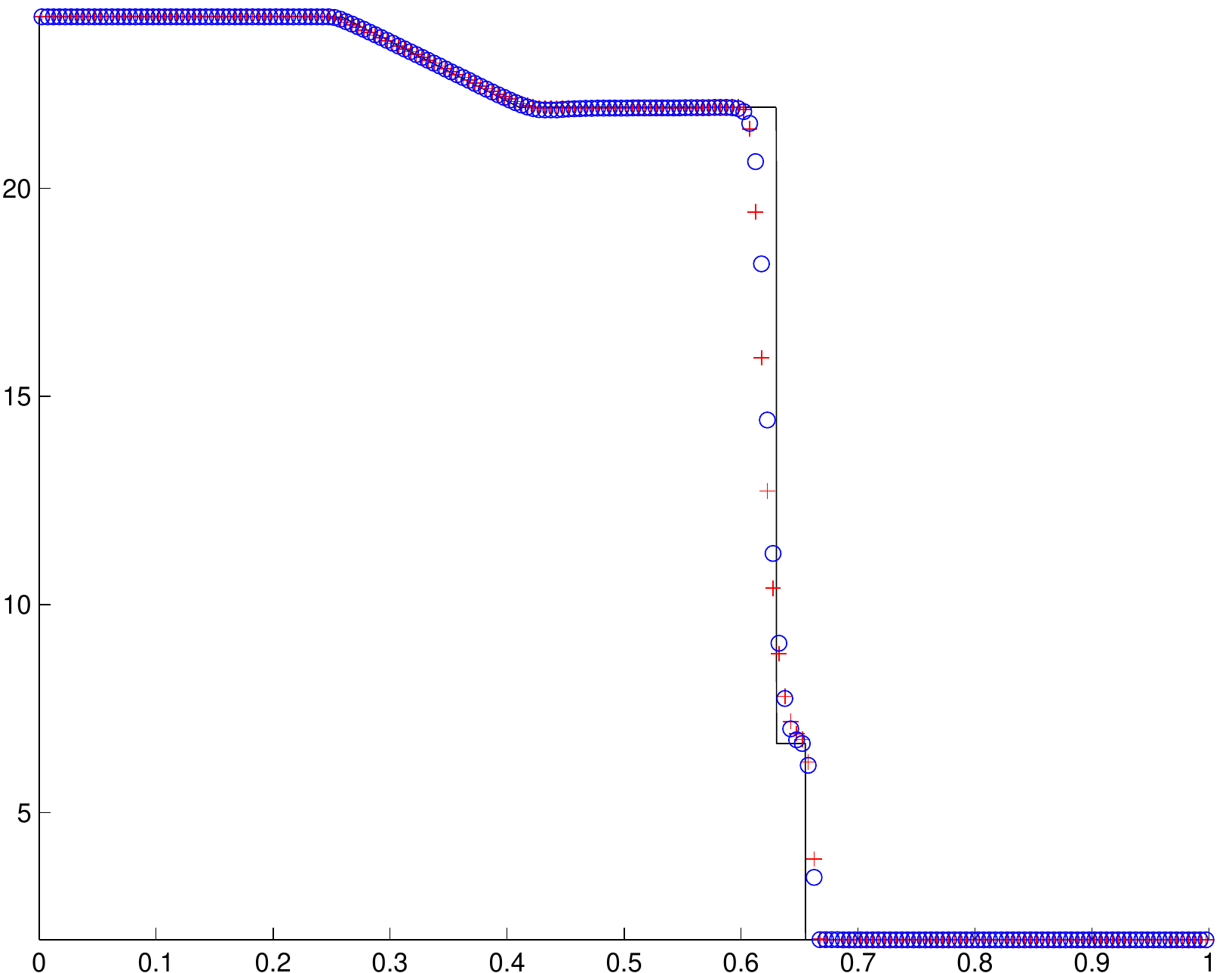}
  }
  \subfigure[Radiative pressure $p_{r}$]
  {
  \includegraphics[width=0.4\textwidth]{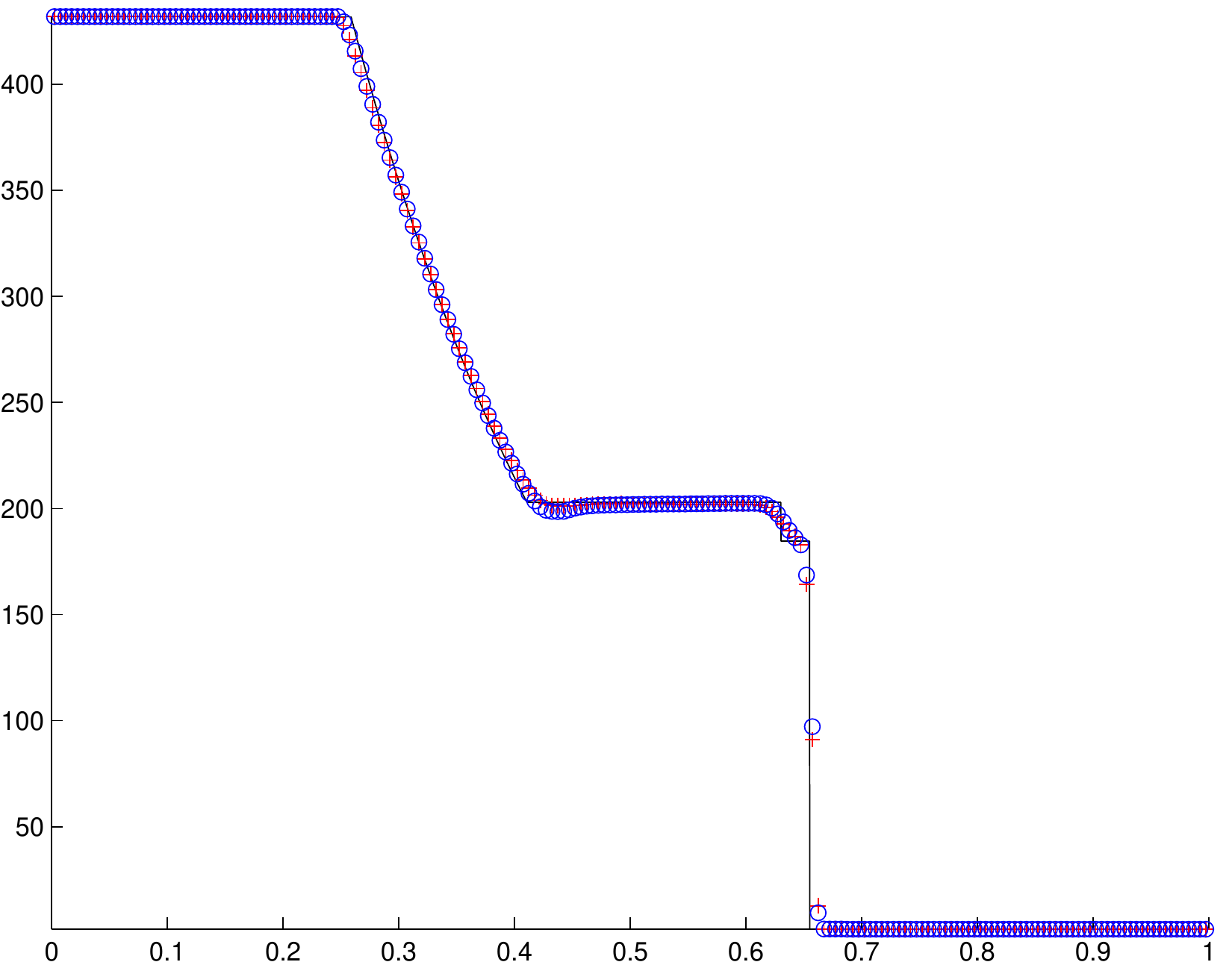}
  }
  \caption{Example \ref{ex:RP4}: The solutions at $t=0.01$ with $200$ uniform cells. }\label{fig:RP4}
\end{figure}

\subsection{Two-dimensional case}
Three two-dimensional examples are considered here. The first is used to check the accuracy of the two-dimensional
GRP scheme and the others are the shock and cloud interaction problem and
the wind and cloud interaction problem considered in \cite{Dai-Woodward1998,Tang2000Kinetic} and used
to validate the performance of the GRP scheme in resolving the discontinuities and complex
wave structure.

\begin{example}[Accuracy test]\label{ex:2Dacc}
It is used to check the accuracy of our GRP scheme for the smooth solution
of two-dimensional RHEs \eqref{RHE2d}
\[
(\rho,u,v,p_{tot})(x,y,t)=\left(1+0.2\sin(2\pi(x-ut,y-vt)),0.2,0.2,1\right),
\]
which describes a sine wave propagating periodically at an angle of $\frac{\pi}{4}$ relative to the $x$-axis in the domain $\Omega=[0,1]\times[0,1]$ with the periodic boundary conditions.

Table \ref{tab:2Dacc1} lists the $l^p$ errors in the density $\rho$ at $t=0.5$ and  convergence rates of the GRP scheme, where $p=1,2,\infty$, and  $\Omega$ is divided into $N\times N$ uniform rectangular cells. It is shown that the $l^1$ and $l^2$ convergence rates of GRP scheme are close to $2$, but the  $l^\infty$ convergence rates is lower than 1.5 due to the nonlinear limiter \eqref{eq:slope}.
\begin{table}[htbp]
\caption{Example \ref{ex:2Dacc}:
The $l^p$ errors in density at $t=0.5$ and convergence rates  of the two-dimensional GRP scheme, where $p=1,2,\infty$.}\label{tab:2Dacc1}
\centering
\begin{tabular}{|c|c|c|c|c|c|c|}
  \hline
 N &$l^1$-error &$l^1$-order    &$l^2$-error &$l^2$-order & $l^{\infty}$-error &$l^{\infty}$-order \\
   \hline
10 & 1.09e-2 & $-$  & 1.35e-2 & $-$ &2.29e-2  &$-$ \\
20 & 3.68e-3 & 1.57  & 4.74e-3 & 1.51  &1.02e-2  &1.16  \\
40 & 1.08e-3 & 1.77 & 1.46e-3 & 1.70 &3.93e-3 &1.38 \\
80 & 2.67e-4 & 2.01 & 4.38e-4 & 1.74 &1.50e-3 &1.39 \\
160 & 6.35e-5 & 2.07 & 1.30e-4 & 1.75&5.67e-4 &1.41 \\
320 & 1.54e-5 & 2.04 & 3.82e-5 & 1.76&2.12e-4 &1.42\\
\hline
\end{tabular}
\end{table}
\end{example}

\begin{example}[Shock and cloud interaction]\rm \label{ex:RP2D1}
This  problem is about the interaction between a shock wave and a denser
cloud in the domain $\Omega=[0,2] \times [0,1]$ with the boundary conditions of zero gradient for the flow variables.
Initially, there is a strong planer shock wave with the Mach number $256.7$ at $x_0=1.64$, which propagates in the $x$-direction. Its pre- and post-states are given by 
\[
(\rho,T,u,v)=\left\lbrace \begin{array}{ll}
(6.57615,20,0,0), & x<x_0, \\
(1,0.01,-22.9472,0), & x>x_0.
\end{array}\right.
\]
The circular cloud  is 100 denser than the pre-shock state, and its radius $R$ and center  are $0.15$ and $(x_0+0.18, 0.5)$, respectively. Moreover, $\hat{a}_R=0.01$.

Figs. \ref{fig:RP2Drho} and \ref{fig:RP2DT} display the contour plots of  density and temperature at $t=0.07$ obtained by using the present GRP scheme with $256 \times 128$ uniform cells, and Fig. \ref{fig:RP2Drhoc1}  draws the density and temperature at $y=0.5$, in comparison with those obtained by
using the MUSCL-Hancock scheme.
It is seen that they are obviously better than those obtained by using the MUSCL-Hancock scheme and the second-order
accurate high resolution KFVS scheme \cite{Tang2000Kinetic} and second order BGK scheme  \cite{Jiang2007}.
\end{example}
\begin{figure}[htbp]
   \centering
   \includegraphics[width=12cm,height=6cm]{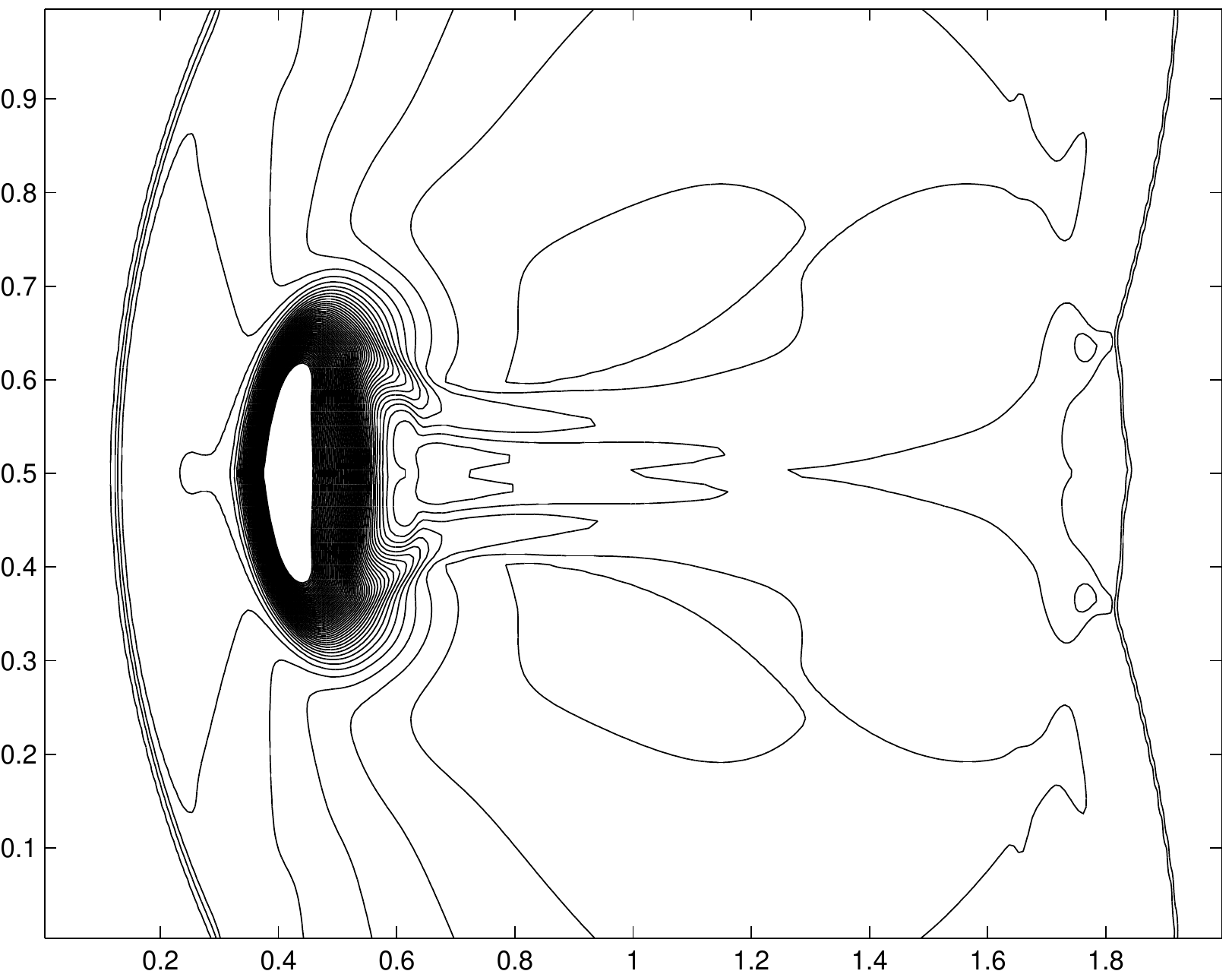} 
   \caption{Example \ref{ex:RP2D1}: The contour plots of density at $t=0.07$ (60 equally spaced contour lines). }\label{fig:RP2Drho}
   \end{figure}

\begin{figure}[htbp]
   \centering
   \includegraphics[width=12cm,height=6cm]{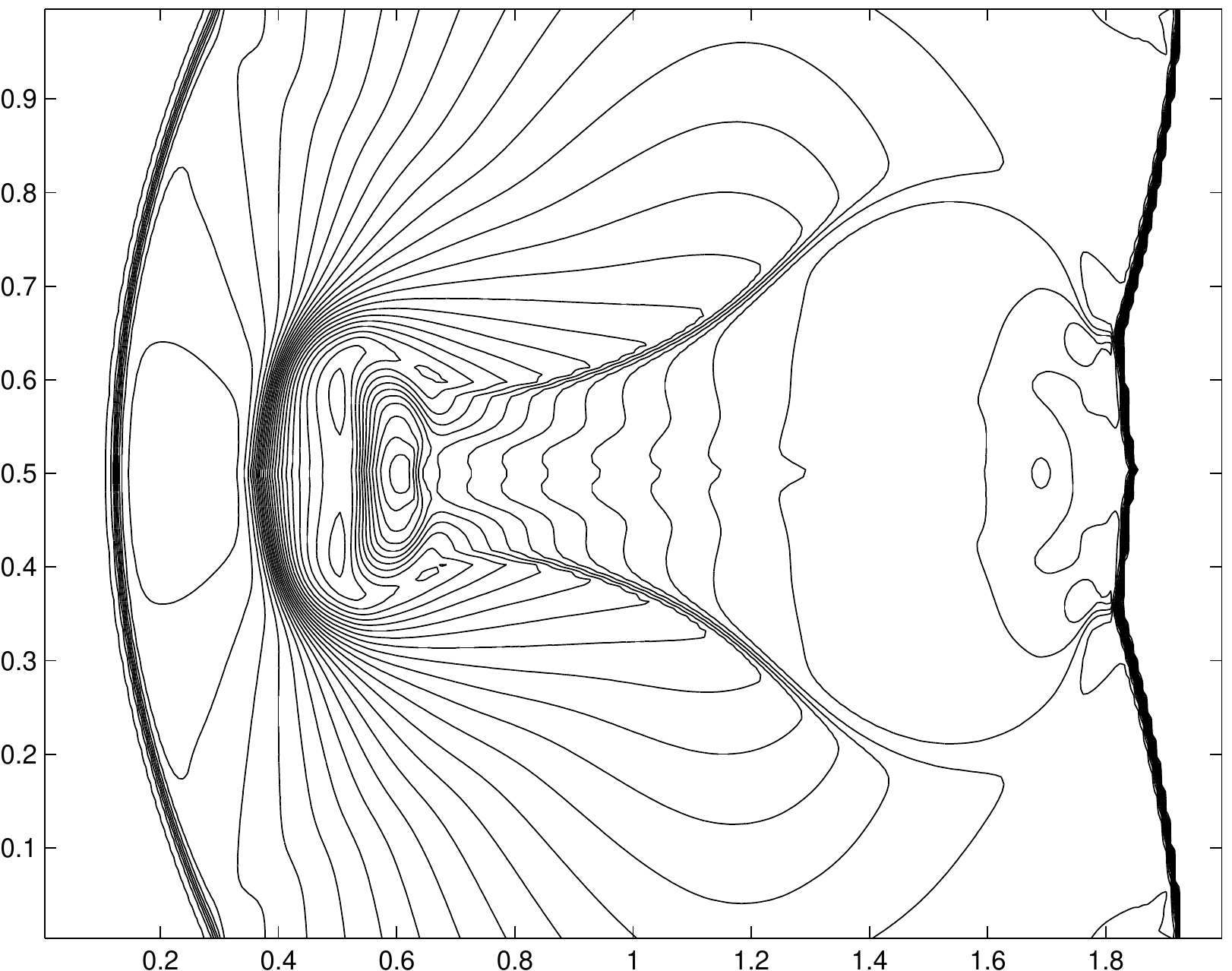} 
   \caption{Same as Fig. \ref{fig:RP2Drho} except for the temperature and 30  equally spaced contour lines. }\label{fig:RP2DT}
   \end{figure}

%

\begin{figure}[htbp]
   \centering
   \includegraphics[width=6cm,height=4.cm]{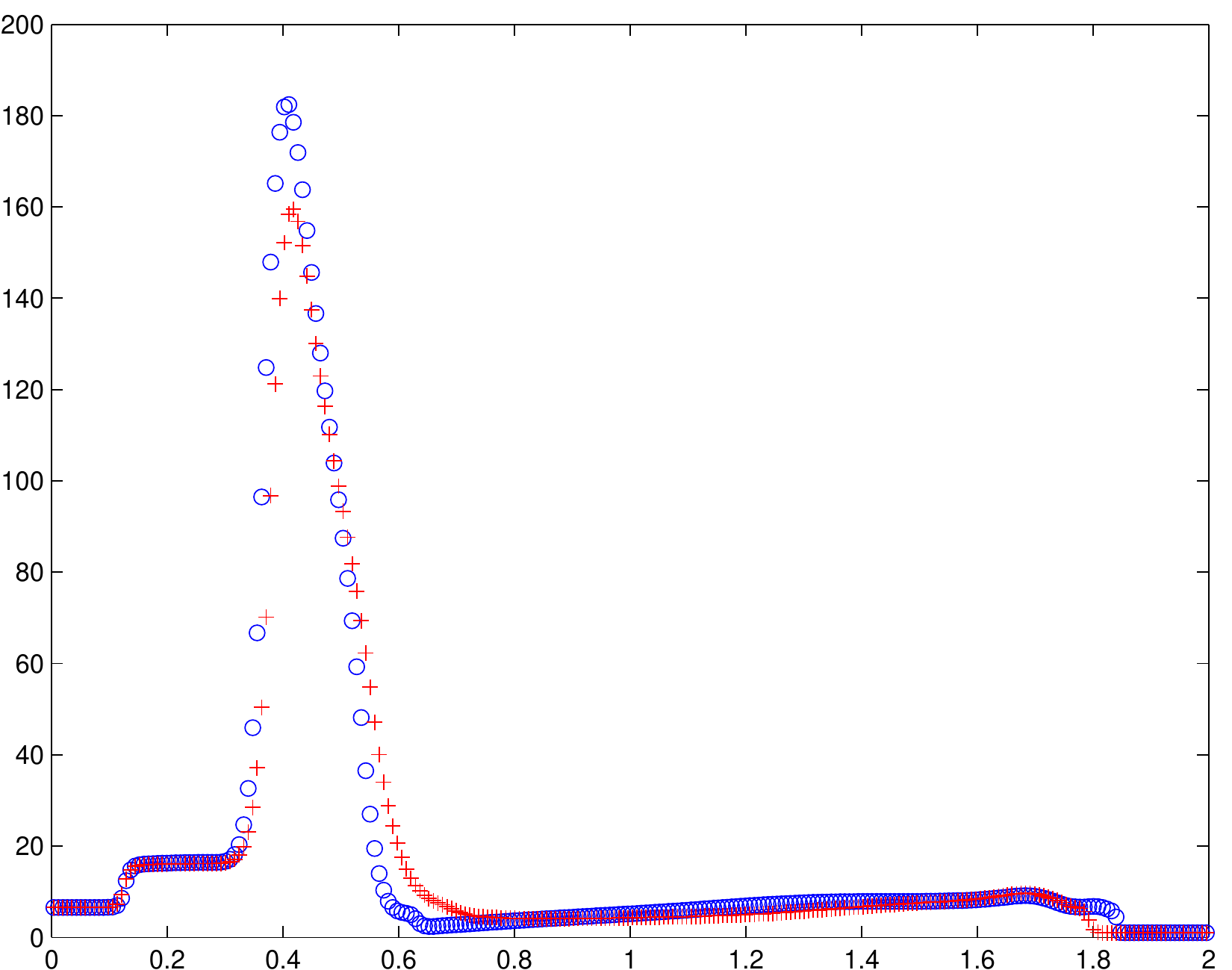} 
      \includegraphics[width=6cm,height=4.cm]{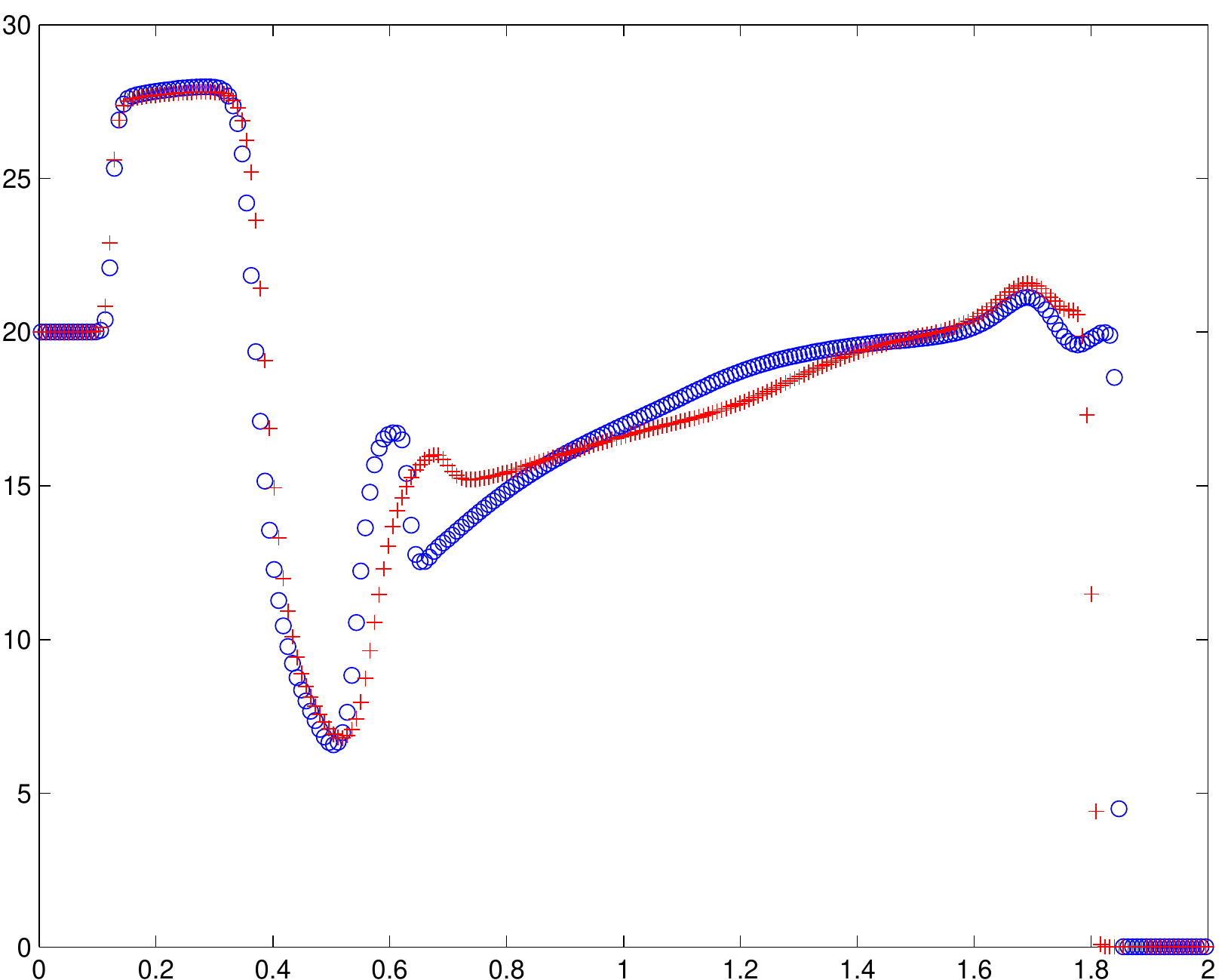} 
   \caption{Example \ref{ex:RP2D1}: The density (left) and temperature (right) at $t=0.07$ along $y=0.5$. }\label{fig:RP2Drhoc1}
   \end{figure}


\begin{example}[Wind and cloud interaction]\rm \label{ex:RP2D2}
It is about the interaction between a wind and a denser
cloud within the computational domain $\Omega=[0,2]\times[0,1]$.
Initially, there is a circular cloud of radius $R=0.15$
in   $\Omega$, whose center is located at $(x_0,y_0)=(0.3, 0.5)$.
The cloud is 25 times denser than the ambient gas with the state $(\rho,T,u,v)=(1, 0.09, 0, 0)$.
The wind is introduced through the left boundary, at which  the state $(\rho,T,u,v)(0,y,t)=\big(1,0.09,6(1-e^{-10t}),0\big)$ is always assigned. The right, upper and lower boundary conditions are zero gradient for the flow variables.
Specially, the initial data in $\Omega$ are given by
   \[
(\rho,T,u,v)(x,y,0)=\left\lbrace \begin{array}{ll}
(25,0.09,0,0), & {(x-x_0)^2+(y-y_0)^2}<R^2, \\
(1,0.09,0,0), & {(x-x_0)^2+(y-y_0)^2}>R^2.
\end{array}\right.
\]

Figs. \ref{fig:RP2D2rho} and  \ref{fig:RP2D2T} show the contour plots
of density and temperature obtained by using the GRP scheme with
 $256 \times 128$ uniform rectangular cells at $t=0.6$, respectively, and Fig. \ref{fig:RP2Drhoc}  plots the density and temperature at $y=0.5$, in comparison with those obtained by using the MUSCL-Hancock scheme. The interaction between the wind and denser cloud  is  captured much better than  those resolved by the MUSCL-Hancock scheme and the  high resolution KFVS   and second order BGK schemes in the literature.
\end{example}
\begin{figure}[htbp]
   \centering
   \includegraphics[width=15cm,height=7cm]{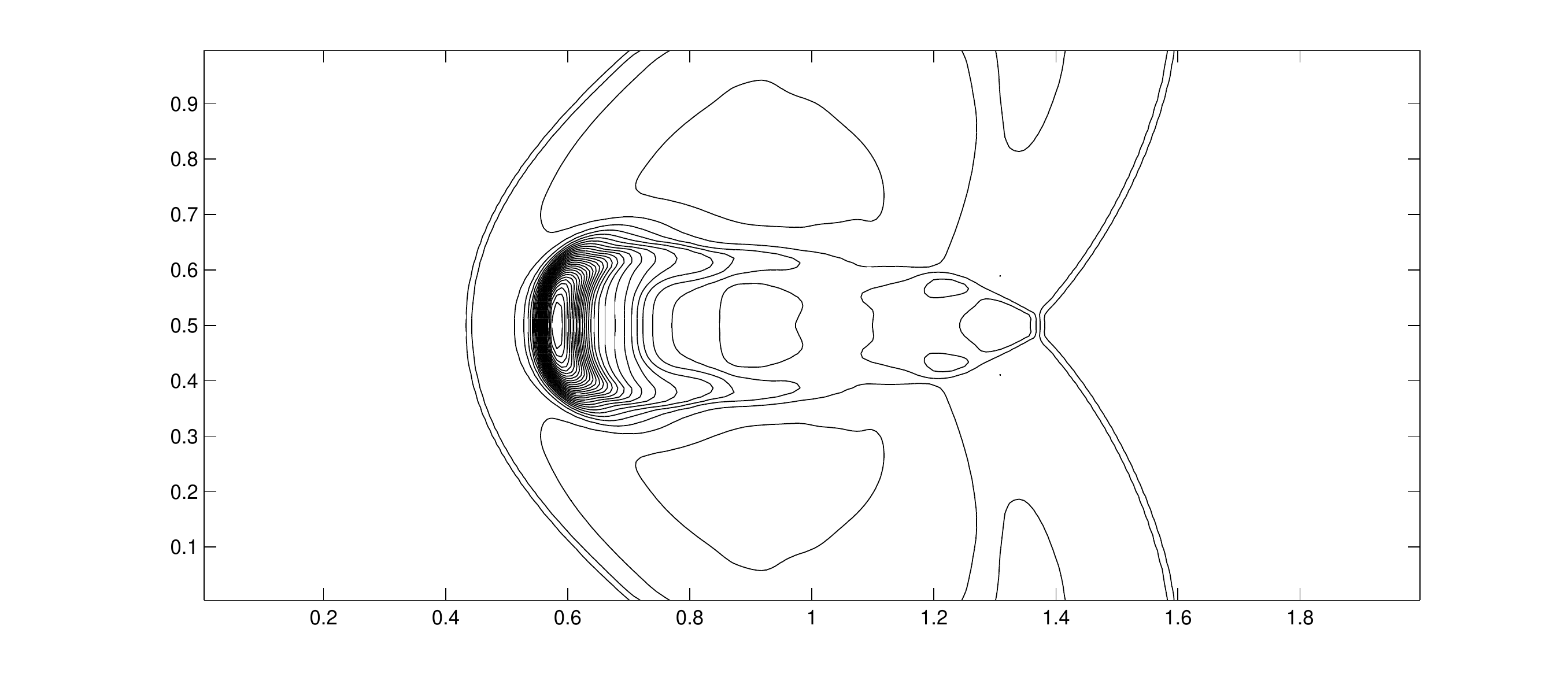} 
   \caption{Example \ref{ex:RP2D2}: The contour plots of density at $t=0.6$ (30 equally spaced contour lines). }\label{fig:RP2D2rho}
   \end{figure}

\begin{figure}[htbp]
   \centering
   \includegraphics[width=15cm,height=7cm]{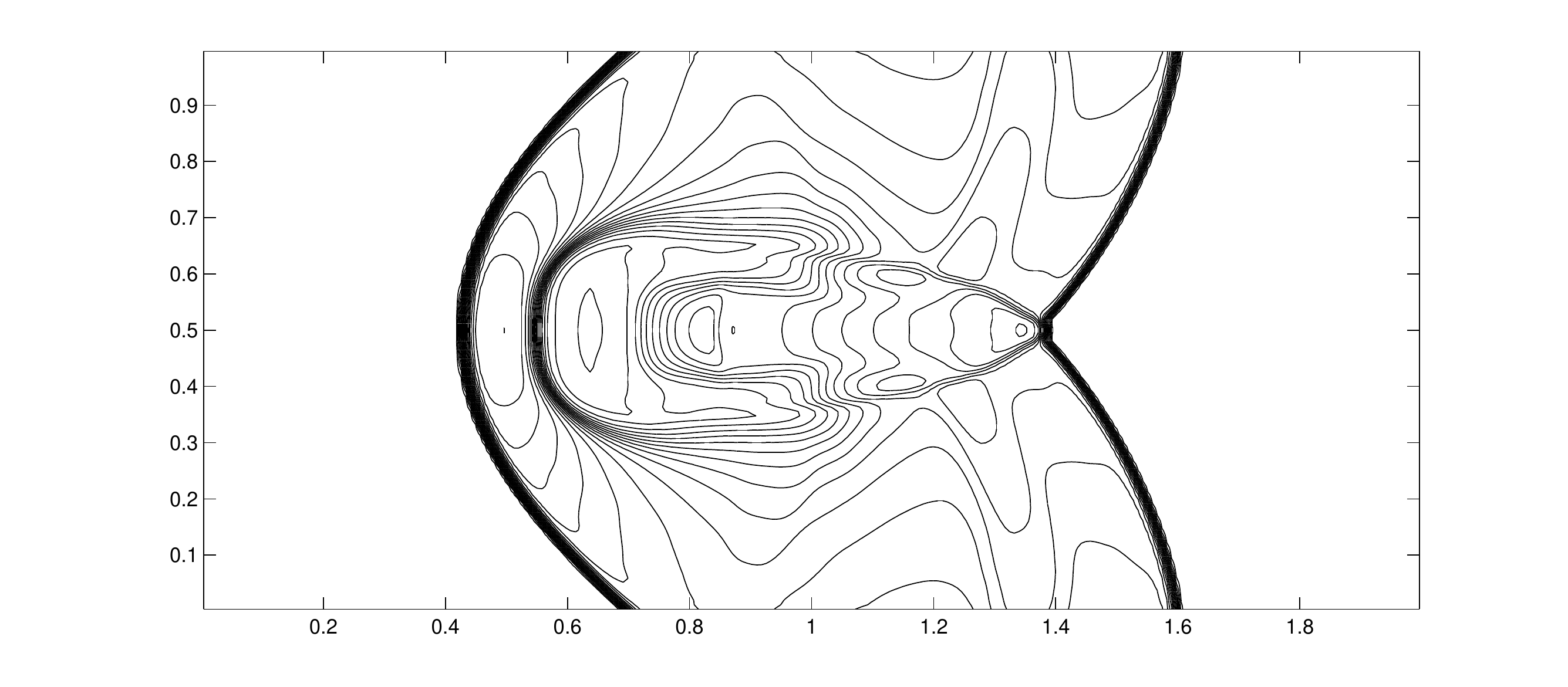} 
   \caption{Same as Fig. \ref{fig:RP2D2rho} except for the temperature and 20  equally spaced contour lines. }\label{fig:RP2D2T}
   \end{figure}

%

\begin{figure}[htbp]
   \centering
   \includegraphics[width=6cm,height=4.cm]{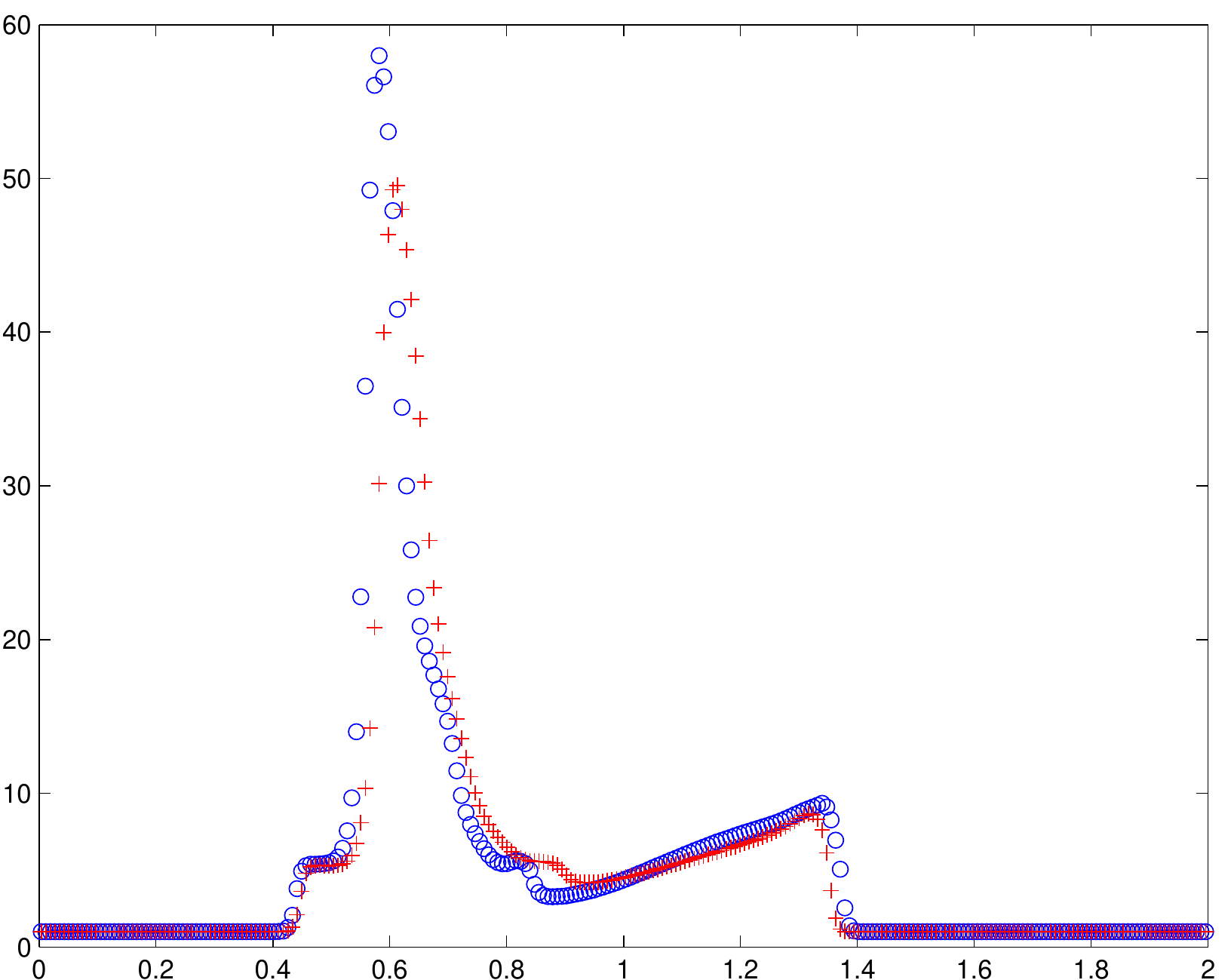} 
      \includegraphics[width=6cm,height=4.cm]{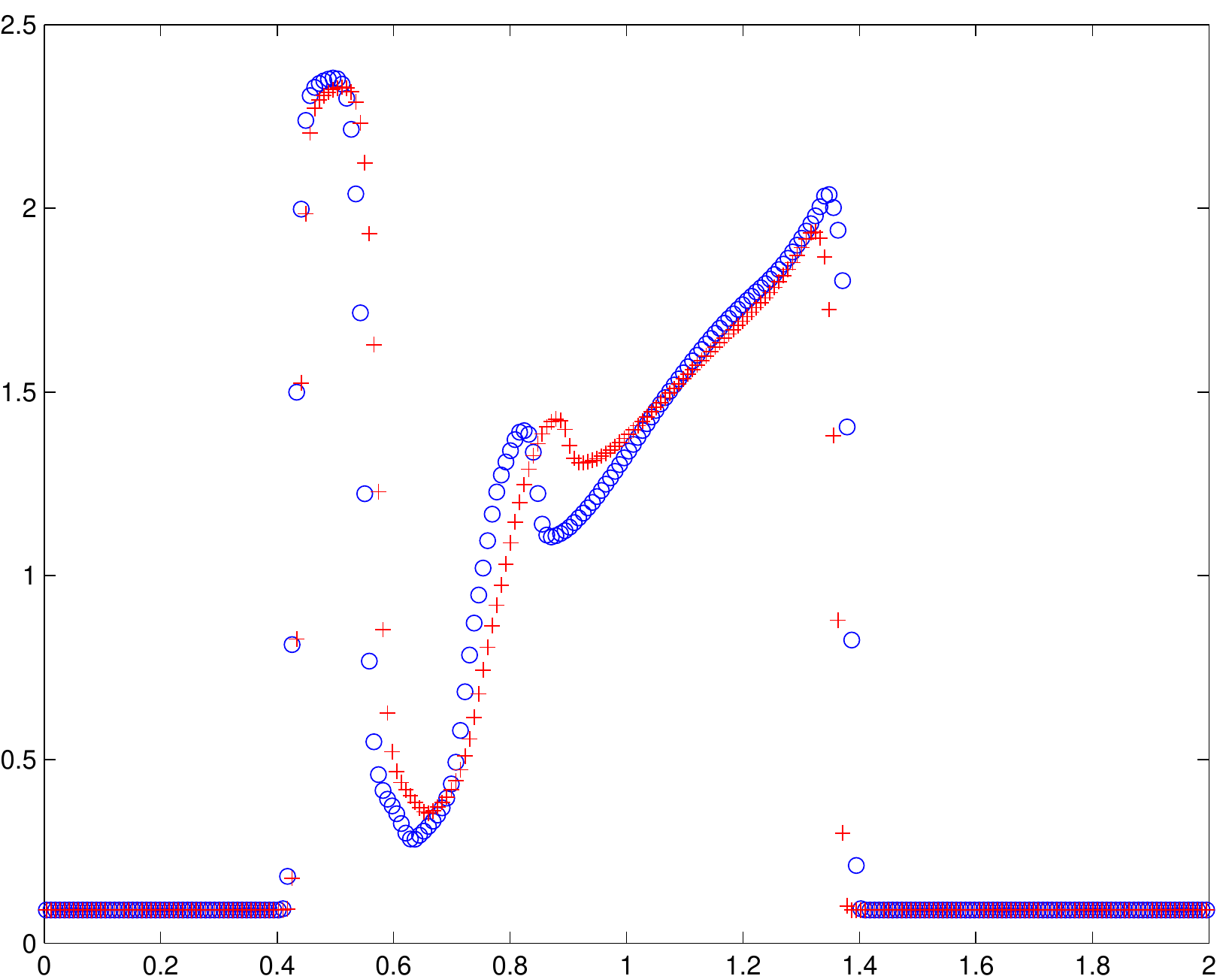} 
   \caption{Example \ref{ex:RP2D2}: The density (left) and temperature (right) at $t=0.6$ along $y=0.5$. }\label{fig:RP2Drhoc}
   \end{figure}

\section{Conclusions}
\label{sec:con}
The paper developed the second-order accurate direct Eulerian generalized Riemann
problem (GRP) scheme for the radiation hydrodynamical equations (RHE) in the
zero diffusion limit. The main difficulty came from no explicit expression of the flux
in terms of the conservative vector due to the nonlinearity in the radiation pressure and energy.
The characteristic fields and the relations between the left and
right states across the elementary-waves were first studied, and then the exact solution of
the one-dimensional Riemann problem was analyzed and given.
 The initial data reconstruction in space were done for the characteristic variables.
 Based on those, the
direct Eulerian GRP scheme was derived by directly using two main ingredients, the generalized Riemann
invariants and the Runkine-Hugoniot jump conditions, to analytically resolve the left
and right nonlinear waves of the local GRP in the Eulerian formulation
 and to obtain
the limiting values of the time derivatives of the conservative variables along the cell
interface and the numerical flux for the GRP scheme.
Several one- and two-dimensional numerical experiments were conducted to demonstrate the accuracy and resolution
of the proposed GRP schemes, in comparison with the second-order accurate MUSCL-Hancock scheme.


\section*{Acknowledgements}
This work was partially supported by the Special Project on High-performance Computing
under the National Key R\&D Program (No. 2016YFB0200603), Science Challenge Project
(No. JCKY2016212A502), and the National Natural Science Foundation of China (Nos.
91330205, 91630310, 11421101). The authors would also like to thank Mr. Pu Chen for his discussion.

\end{document}